\documentclass[11pt,reqno]{amsart}
\usepackage{esint}
\usepackage{bbm}
\usepackage{amssymb}
\usepackage{mathrsfs}
\usepackage{amsfonts}
\usepackage{amsfonts,amssymb,amsmath,amsthm}
\usepackage{url}
\usepackage{enumerate}
\usepackage{graphicx,xcolor}
\definecolor{refkey}{gray}{0.5}
\definecolor{labelkey}{gray}{0.1}

\newtheorem{theorem}{Theorem}[section]
\newtheorem{lemma}[theorem]{Lemma}
\newtheorem{proposition}[theorem]{Proposition}
\newtheorem{corollary}[theorem]{Corollary}
\newtheorem{claim}[theorem]{Claim}

\theoremstyle{definition}
\newtheorem{definition}[theorem]{Definition}
\newtheorem{remark}[theorem]{Remark}
\numberwithin{equation}{section}

\newtheorem{example}[theorem]{Example}

\newtheorem{problem}[theorem]{Problem}
\newtheorem{assumption}[theorem]{Assumption}
\newtheorem{notation}[theorem]{Notation}

\usepackage[left=2.0cm, right=2.0cm, top=2.0cm, bottom=2.0cm]{geometry}

\usepackage{graphicx}%
\usepackage{color}
\usepackage{cite}
\usepackage{hyperref}
\hypersetup{
	colorlinks = true,%
	citecolor = blue,%[rgb]{0.65,0.0,0.0},
	filecolor=red,%
	linkcolor = [rgb]{0.65,0.0,0.0},%
	anchorcolor = red,
	pagecolor = red,
	urlcolor= [rgb]{0.65,0.0,0.0}
	linktocpage=true,
	pdfpagelabels=true,
	bookmarksnumbered=true,
}

\DeclareMathOperator{\FAC}{\mathsf{FAC}}
\DeclareMathOperator{\GMM}{\mathsf{GMM}}

\DeclareMathOperator{\id}{Id}
\DeclareMathOperator{\diam}{Diam}

\DeclareMathOperator{\loc}{loc}

\DeclareMathOperator{\di}{div}

\DeclareMathOperator{\lo}{loc}

\DeclareMathOperator{\dd}{{d}}
\DeclareMathOperator{\ee}{{e}}

\DeclareMathOperator{\sgn}{sgn}

\DeclareMathOperator{\ac}{ac}

\DeclareMathOperator{\argmin}{argmin}

\DeclareMathOperator{\gtau}{\mathfrak{T}}
\DeclareMathOperator{\lra}{\longrightarrow}

\DeclareMathOperator{\w}{\mathbf{w}}
\DeclareMathOperator{\vv}{\mathbf{v}}
\DeclareMathOperator{\LL}{\mathfrak{L}}

\DeclareMathOperator{\m}{{\mathfrak{m}}}

\makeatletter
\@namedef{subjclassname@2020}{%
  \textup{2020} Mathematics Subject Classification}
\makeatother

\author{Shin-ichi Ohta}
\address{
Department of Mathematics\\
Osaka University\\
560-0043 Osaka\\
Japan\\
and RIKEN Center for Advanced Intelligence Project (AIP)\\
1-4-1 Nihonbashi \\ Tokyo 103-0027\\
Japan}
\email{s.ohta@math.sci.osaka-u.ac.jp}

\author{Wei Zhao}
\address{
School of Mathematics\\
East China University of Science and Technology\\
200237 Shanghai, China}
\email{szhao\underline{ }wei@yahoo.com}

\begin{document}

\title[Gradient flows in asymmetric metric spaces and applications]{Gradient flows in asymmetric metric spaces and applications}

\begin{abstract}
This paper is devoted to the investigation of gradient flows in asymmetric metric spaces
(for example, irreversible Finsler manifolds and Minkowski normed spaces)
by means of discrete approximation.
We study basic properties of curves and upper gradients in asymmetric metric spaces,
and establish the existence of a curve of maximal slope,
which is regarded as a gradient curve in the non-smooth setting.
{Introducing} a natural convexity assumption on the potential function,
{which is called the $(p,\lambda)$-convexity,}
we also obtain some regularizing effects on the asymptotic behavior of curves of maximal slope.
Applications include several existence results for gradient flows in Finsler manifolds,
doubly nonlinear differential evolution equations on {infinite-dimensional Funk spaces},
and heat flow on compact Finsler manifolds.
\end{abstract}

\maketitle

%\tableofcontents

\section{Introduction}%%%%%%%%%%%%%
%%%%%%%%%%%%%%%%%%

The aim of this article is to {develop the theory} of gradient flows in \emph{asymmetric} metric spaces
(i.e., the symmetry $d(x,y)=d(y,x)$ is not assumed; see Definition~\ref{generalsapcedef}).
Typical and important examples are gradient flows of geodesically convex functions
on irreversible Finsler manifolds (or Minkowski normed spaces).
The theory of gradient flows has been successfully developed in ``Riemannian-like'' spaces
such as CAT$(0)$-spaces and RCD-spaces
(see, e.g., \cite{AGS,Bacak,GN,Jost,Mayer,Ograd,OP1,OP2,Sav,Sturm});
nonetheless, the lack of the Riemannian-like structure causes a significant difference
and we know much less about gradient flows in ``Finsler-like'' spaces
(see \cite[Remark~3.2]{OP2}, \cite{OSnc} and Subsection~\ref{ssc:rem} for more details).
In this article, based on the recent work \cite{KZ} on the geometry of asymmetric metric spaces,
we investigate gradient flows in asymmetric metric spaces
by generalizing the minimizing movement scheme as in \cite{AGS}.
{
Compared with the preceding studies \cite{CRZ,RMS} on asymmetric metric spaces,
we will be able to remove some conditions on the space $(X,d)$ or the potential function $\phi$
by a more careful analysis
(see Remarks~\ref{diffbetRMS}, \ref{diffbettheree} for details).
Moreover, the notion of $(p,\lambda)$-convexity (discussed in Section~\ref{plambdaconvex})
seems new and of independent interest even in the symmetric case.}

Asymmetric metrics often occur in nature and can be represented as Finsler metrics;
a prominent example is the Matsumoto metric
describing the law of walking on a mountain slope under the action of gravity
(see \cite{Matsumoto}).
Randers metrics appearing as solutions to the Zermelo navigation problem
(concerning a Riemannian manifold with ``wind'' blown on it)
provide another important class of irreversible metrics (see \cite{BRS}).
A particular example of the latter metric is given as a ``non-symmetrization''
of the Klein metric on the $n$-dimensional Euclidean unit ball
$\mathbb{B}^n =\{ x \in \mathbb{R}^n \,|\, \|x\|<1\}$ ($n \ge 2$),
called the \textit{Funk metric} (see, e.g., \cite[Example~1.3.5]{Sh1}),
defined as $F:\mathbb{B}^n \times \mathbb{R}^n \lra [0,\infty)$ by
\begin{equation}\label{Funckmeatirc}
 F(x,v)
 =\frac{\sqrt{\|v\|^2-(\|x\|^2\|v\|^2-\langle x,v \rangle^2)} +\langle x,v\rangle}{1-\|x\|^2},
 \quad x \in \mathbb{B}^n,\ v \in T_x\mathbb{B}^n=\mathbb{R}^n,
\end{equation}
where $\|\cdot\|$ and $\langle\cdot, \cdot\rangle$ denote
the Euclidean norm and inner product, respectively.
The associated distance function $d_F$ is written as (see \cite[Example~1.1.2]{Sh1})
\begin{equation}\label{distFunk}
 d_{F}(x_1,x_2) =\log\Bigg(
 \frac{\sqrt{\|x_1-x_2\|^2-(\|x_1\|^2\|x_2\|^2-\langle x_1,x_2\rangle^2)}-\langle x_1,x_2-x_1\rangle}
 {\sqrt{\|x_1-x_2\|^2-(\|x_1\|^2\|x_2\|^2-\langle x_1,x_2\rangle^2)}-\langle x_2,x_2-x_1\rangle}
 \Bigg), \quad x_1,x_2 \in \mathbb{B}^n.
\end{equation}
It is readily seen that $d_F(x_1,x_2)\neq d_F(x_2,x_1)$ and,
for $\mathbf{0}=(0,\ldots,0)$,
\[
 \lim_{\|x\| \to 1} d_F(\mathbf{0},x)=\infty, \qquad
 \lim_{\|x\| \to 1} d_F(x,\mathbf{0})=\log 2.
\]
The Funk metric $(\mathbb{B}^n,d_F)$
will be one of the model asymmetric structures we have in mind.
We remark that the symmetrization $d(x_1,x_2):=\{d_F(x_1,x_2)+d_F(x_2,x_1)\}/2$
coincides with the Klein metric.

For functions on an asymmetric metric space $(X,d)$,
we shall study the associated \emph{curves of maximal slope} in $X$ (see Subsection~\ref{ssc:maxslope});
this conception generalizes gradient curves in the smooth setting.
In order to deal with the asymmetry, the \emph{reversibility} of $(X,d)$, defined by
\[
 \lambda_d(X):=\sup_{x\neq y} \frac{d(x,y)}{d(y,x)},
\]
will be instrumental.
Clearly $\lambda_d(X) \geq 1$,
and $\lambda_d(X)=1$ holds if and only if $d$ is symmetric.
The reversibility may be $\infty$ for noncompact asymmetric metric spaces.
Actually, for the model Funk metric \eqref{Funckmeatirc},
a direct calculation yields $\lambda_{d_F}(\mathbb{B}^n)=\infty$ and
$\lambda_{d_F}(B^+_\mathbf{0}(r)) \leq 2{\ee}^r-1$,
where $B^+_\mathbf{0}(r)$ is the forward open ball of radius $r$ centered at $\mathbf{0}$, i.e.,
$B^+_\mathbf{0}(r)=\{x\in \mathbb{B}^n|\, d_F(\mathbf{0},x)<r\}$.
The latter estimate suggests to consider a collection of pointed spaces
$(X,\star,d)$ whose reversibility satisfies $\lambda_d(B^+_\star(r))\leq \Theta(r)$
for some non-decreasing function $\Theta:(0,\infty) \lra [1,\infty)$.
Such spaces are called \emph{forward metric spaces}
(see Subsection~\ref{ssc:forward}) and intensively studied in \cite{KZ}.
For instance, the Gromov--Hausdorff topology and the theory of curvature-dimension condition
developed by Lott, Sturm and Villani can be generalized to such spaces.
Every forward complete Finsler manifold is a forward metric space (see \cite{KZ} for details).

In the present article, by generalizing the theory of \cite{AGS} to forward metric spaces,
we are able to obtain some existence and regularity results of curves of maximal slope.
Among others, we establish the existence of curves of maximal slope
satisfying the energy identity (Theorem~\ref{existsassupab}),
and show some estimates on the behavior of the potential function and its upper gradient
along curves of maximal slope (Theorem~\ref{theogmmd}).
{
In the latter result, we make use of the \emph{$(p,\lambda)$-convexity},
defined in Definition~\ref{df:pl-conv} by
\[
 \phi\big( \gamma(t) \big)
 \leq (1-t)\phi\big( \gamma(0) \big) +t\phi\big( \gamma(1) \big)
 -\frac{\lambda}{p} t(1-t^{p-1}) d^p\big( \gamma(0),\gamma(1) \big)
\]
(which slightly differs from the $(\lambda,p)$-convexity studied in \cite{RSS};
see Remark~\ref{rm:p-conv} for a further account), plays a role.}
As an application, we have the following in the Finsler setting
(see Corollaries~\ref{bascirgeulfinsler}, \ref{regularoffINSLERCASE} for the precise statements).

\begin{theorem}\label{Finslegradientflow1}
Let $(M,F)$ be a forward complete Finsler manifold and $\phi \in C^1(M)$.
Then, for any $x_0 \in M$,
there exists a $C^1$-curve $\xi:[0,T) \lra M$ solving the gradient flow equation
\[
 \xi'(t) =\nabla(-\phi)\big( \xi(t) \big), \qquad \xi(0)=x_0,
\]
where $\lim_{t \to T}d_F(x_0,\xi(t))=\infty$ if $T<\infty$.
If $\phi$ is $\lambda$-geodesically convex for some $\lambda>0$,
then $T=\infty$, $\xi(t)$ converges to a unique minimizer $\bar{x}$ of $\phi$
and $F(\nabla(-\phi)(\xi(t)))$ decreases to $0$ as $t \to \infty$.
\end{theorem}

Another application is concerned with ``infinite dimensional Finsler spaces"
such as the unit ball in a Hilbert space endowed with the Funk metric \eqref{distFunk}
and a reflexive Banach space $(\mathscr{B},\|\cdot\|)$ equipped with
a Randers-type metric $d(x,y)=\|y-x\|+\omega(y-x)$
for some $\omega\in \mathscr{B}^*$ with $\|\omega\|_*<1$.
For such spaces, we prove that curves of maximal slope satisfy
a doubly nonlinear differential evolution equation or inclusion,
{which generalizes some results in \cite{RMS}}
(see Subsection~\ref{funkspacesgrad} for details).

Last but not least, we also investigate gradient flows in the Wasserstein space
over a compact Finsler manifold.
We establish the equivalence between weak solutions to the heat equation
and trajectories of the gradient flow for the relative entropy (in the same spirit as \cite{Erb,OS}),
as well as the following existence result of weak solutions to the heat equation
as curves of maximal slope, which gives a direct construction rather than the extension
by the $L^2$-contraction as in \cite{OS} (see Subsection~\ref{heatflowfins} for a further account).

\begin{theorem}
Let $(M,F)$ be a compact Finsler manifold endowed with a smooth positive measure $\m$.
For any function $u \in L^2(M)$ bounded above,
there exists a weak solution $(u_t)_{t \ge 0}$
to the heat equation $\partial_t u_t =\Delta_{\m} u_t$ with $u_0=u$.
\end{theorem}

We stress that there remain important open problems for gradient flows
in Finsler-like spaces, even in (symmetric) normed spaces; see Subsection~\ref{ssc:rem}.
We hope that our work motivates further investigations.

This article is organized as follows.
In Section~\ref{sect1}, we introduce some necessary concepts to analyze
curves of maximal slope in forward metric spaces.
In Section~\ref{topologyexiscurvew}, we take advantage of the Moreau--Yosida approximation
to prove the existence of curves of maximal slope under mild topological assumptions.
Section~\ref{plambdaconvex} is devoted to the investigation of curves of maximal slope
{for $(p,\lambda)$-convex functions}, including some further existence and regularity results
and the study of heat flow.

\smallskip

\noindent{\textbf{Acknowledgements.}}
{The authors thank the anonymous referees for valuable comments.}
The first author was supported by JSPS Grant-in-Aid for Scientific Research (KAKENHI) 19H01786,
and the second author was supported by Natural Science Foundation of Shanghai
(Nos.\ 21ZR1418300, 19ZR1411700).

\section{Curves and upper gradients in asymmetric metric spaces}\label{sect1}%%%%%%
%%%%%%%%%%%%%%%%%%%%%%%

\subsection{Forward metric spaces}\label{ssc:forward}%%%%%%%%
%%%%%%%%%%%%%%%%%%%%

First we discuss the basic properties of asymmetric metric spaces.
We refer to \cite{CRZ,KZ,ACFM,ACFM2,RMS,SZ} for related studies (partly with different names).

\begin{definition}[Asymmetric metric spaces]\label{generalsapcedef}
Let $X$ be a nonempty set and $d:X\times X \lra [0,\infty)$ be a nonnegative function on $X \times X$.
We call $(X,d)$ an \emph{asymmetric metric space} if $d$ satisfies
\begin{enumerate}[(1)]
\item $d(x,y)\geq 0$ for all $x,y \in X$ with equality if and only if $x=y$;
\item $d(x,z)\leq d(x,y)+d(y,z)$ for all $x,y,z \in X$.
\end{enumerate}
The function $d$ is called a \emph{distance function} or a \emph{metric} on $X$.
\end{definition}

Since the function $d$ could be asymmetric (i.e., $d(x,y) \neq d(y,x)$),
there are two kinds of balls.
For a point $x \in X$ and $r>0$, the \emph{forward} and \emph{backward balls}
of radius $r$ centered at $x$ are defined as
\[
 B^+_x(r):=\{y\in X \,|\, d(x,y)<r\}, \qquad
 B^-_x(r):=\{y\in X \,|\, d(y,x)<r\}.
\]
Let $\mathcal{T}_+$ (resp.\ $\mathcal{T}_-$) denote the topology
induced from forward balls (resp.\ backward balls).
In order to study the relation between $\mathcal{T}_+$ and $\mathcal{T}_-$,
the following notion on the reversibility of $d$ was introduced in \cite{KZ}.

\begin{definition}[Pointed forward $\Theta$-metric spaces]\label{thetametricspace}
Let $\Theta:(0,\infty) \lra [1,\infty)$ be a (not necessarily continuous) non-decreasing function.
A triple $(X,\star,d)$ is called a \emph{pointed forward $\Theta$-metric space}
if $(X,d)$ is an asymmetric metric space and $\star$ is a point in $X$ such that
$\lambda_d (B^+_\star(r)) \leq \Theta(r)$ for all $r>0$, where
\[
 \lambda_d\big( B^+_\star(r) \big)
 :=\inf\big\{ \lambda\geq1 \,\big|\,
 d(x,y)\leq \lambda d(y,x) \text{ for any } x,y \in B^+_\star(r) \big\}.
\]
If we can take a constant function $\Theta\equiv\theta$ (i.e., $\lambda_d(X) \le \theta$),
then we call $(X,d)$ a \emph{$\theta$-metric space}.
\end{definition}

Note that we have the bounded reversibility $\lambda_d<\infty$ only on forward balls,
thereby the reversibility of backward balls may be infinite (like the Funk metric).
For pointed forward $\Theta$-metric spaces,
the backward topology is weaker than the forward topology as follows (see \cite[Theorem~2.6]{KZ}).

\begin{theorem}[Properties of $\mathcal{T}_+$]\label{topologychara}
Let $(X,\star, d)$ be a pointed forward $\Theta$-metric space.
Then,
\begin{enumerate}[{\rm (i)}]
\item $\mathcal{T}_- \subset \mathcal{T}_+$ holds and,
%% $\mathcal{T}_-$ is weaker than $\mathcal{T}_+$
in particular, $d$ is continuous in $\mathcal{T}_+ \times \mathcal{T}_+$
and $(X,\mathcal{T}_+)$ is a Hausdorff space$;$

\item $\mathcal{T}_+$ coincides with the topology $\widehat{\mathcal{T}}$
induced from the symmetrized distance function
\[
 \widehat{d}(x,y) := \frac12 \big\{ d(x,y)+d(y,x) \big\}.
\]
%\smallskip
%\item[(iii)] If $\mathcal{X}$ is sequentially compact with respect to $\mathcal{T}_+$, then $\mathcal{T}_+=\mathcal{T}_-$.
\end{enumerate}
\end{theorem}

We remark that the topology $\tau$ considered in \cite{ACFM,ACFM2}
coincides with $\widehat{\mathcal{T}}$.
Precisely, $\tau$ is induced from both forward and backward balls
and associated with $\widetilde{d}(x,y):=\max\{ d(x,y),d(y,x) \}$,
then $\widehat{d} \le \widetilde{d} \le 2\widehat{d}$ yields $\tau=\widehat{\mathcal{T}}$.
Some more remarks on pointed forward $\Theta$-metric spaces are in order.

\begin{remark}\label{forwardpointspaceandbackwardones}
\begin{enumerate}[(a)]
\item \label{pfms-a}
A sequence $(x_i)_{i \geq 1}$ in $X$ converges to $x$
with respect to $\mathcal{T}_+$ if and only if
\[
 \lim_{i \to \infty}d(x,x_i)=0,
\]
which implies the convergence $\lim_{i \to \infty}d(x_i,x)=0$ in $\mathcal{T}_-$
(thanks to $\lambda_d (B^+_\star(r)) \leq \Theta(r) <\infty$).
However, the converse does not necessarily hold true (when $d(\star,x_i) \to \infty$).

\item \label{pfms-b}
If $(X,\star,d)$ is a pointed forward $\Theta$-metric space, then, for every $x\in X$,
the triple $(X,x,d)$ is a pointed forward $\overline{\Theta}$-metric space
for $\overline{\Theta}(r):=\Theta(d(\star,x)+r)$.
Moreover, if $\diam(X):=\sup_{x,y\in X}d(x,y)<\infty$,
then $(X,d)$ is a $\theta$-metric space with $\theta:=\Theta(\diam(X))$.

\item \label{pfms-c}
One can similarly introduce a \emph{pointed backward $\Theta$-metric space}
 $(X,\star,d)$ by $\lambda_d (B^-_\star(r)) \leq \Theta(r)$ for $r>0$.
Note that a pointed backward $\Theta$-metric space may not be
a pointed forward $\overline{\Theta}$-metric space for any $\overline{\Theta}$;
recall the Funk metric in the introduction.
Since $(X,\star,d)$ is a pointed backward $\Theta$-metric space
if and only if $(X,\star,\overleftarrow{d})$ is a pointed forward $\Theta$-metric space
for the \emph{reverse metric} $\overleftarrow{d}(x,y):=d(y,x)$,
we will focus only on pointed forward $\Theta$-metric spaces.
\end{enumerate}
\end{remark}

\begin{notation}[Forward metric spaces]\label{notationX}
In the sequel, every pointed forward $\Theta$-metric space $(X,\star,d)$
is endowed with the forward topology $\mathcal{T}_+$.
Suppressing $\star$ and $\Theta$ for the sake of simplicity,
we will write $(X,d)$ and call it a \emph{forward metric space}.
\end{notation}

Forward metric spaces possess many fine properties.
For example, one can define a generalized Gromov--Hausdorff topology
to study the convergence of forward metric spaces,
and Gromov's precompactness theorem remains valid.
Besides, optimal transport theory can be developed and
the weak curvature-dimension condition in the sense of Lott--Sturm--Villani \cite{LV,Sturm-1,Sturm-2}
is stable also in this setting.
Furthermore, various geometric and functional inequalities
(such as Brunn--Minkowski, Bishop--Gromov, log-Sobolev and Lichnerowicz inequalities)
can be established.
We refer to \cite{KZ} for details and further results.

We next recall some concepts related to the completeness (cf.\ \cite{ACFM, ACFM2,KZ,SZ}).

\begin{definition}[Completeness]\label{df:complete}
Let $(X,d)$ be an asymmetric metric space.
\begin{enumerate}[(1)]
\item
A sequence $(x_i)_{i \geq 1}$ in $X$ is called
a \emph{forward} (resp.\ \emph{backward}) \emph{Cauchy sequence} if, for each $\varepsilon>0$,
there is $N \geq 1$ such that $d(x_i,x_j)<\varepsilon$ (resp.\ $d(x_j,x_i)<\varepsilon$) holds
for all $j \ge i>N$.

\item
$(X, d)$ is said to be \emph{forward} (resp.\ \emph{backward}) \emph{complete}
if every forward (resp.\ backward) Cauchy sequence in $X$ is convergent
with respect to $\mathcal{T}_+$.

\item We say that $(X,d)$ is \emph{forward} (resp.\ \emph{backward}) \emph{boundedly compact}
if every closed set in any forward (resp.\ backward) bounded ball is compact.
\end{enumerate}
\end{definition}

If $\Theta$ is bounded (i.e., $\lambda_d(X)<\infty$),
then the forward and backward properties are mutually equivalent.
However, they are not equivalent when $\Theta$ is unbounded;
e.g., the Funk metric is forward complete but backward incomplete.

\subsection{Forward absolutely continuous curves}%%%%%%%%
%%%%%%%%%%%%%%%%%%%%

Let $(X,\star,d)$ be a forward complete pointed forward $\Theta$-metric space in this subsection.

\begin{definition}[Forward absolutely continuous curves]\label{forwardabcontinc}
A curve $\gamma:I \lra X$ on an interval $I \subset \mathbb{R}$ is said to be
\emph{$p$-forward absolutely continuous} for $p\in [1,\infty]$
(denoted by $\gamma\in \FAC^p(I;X)$) if there is a nonnegative function $f\in L^p(I)$ such that
\begin{equation}\label{moredfabsc}
 d \big( \gamma(s),\gamma(t) \big) \leq
 \int^t_s f(r) \,{\dd}r \quad \text{for all}\ s,t \in I \ \text{with}\ s\leq t.
\end{equation}
We will denote $\FAC^1(I;X)$ by $\FAC(I;X)$
and call its element a \emph{forward absolutely continuous} curve.
\end{definition}

A standard argument combined with the forward completeness yields the following.

\begin{lemma}\label{uniformlyconverge}
Any curve $\gamma\in \FAC^p([a,b);X)$ is forward uniformly continuous,
i.e., for any $\varepsilon>0$, there exists $\delta>0$ such that
$d(\gamma(s),\gamma(t))<\varepsilon$ holds for any $s,t \in [a,b)$ with $0 \leq t-s<\delta$.
In particular, if $b<\infty$, then the left limit $\gamma(b_-):=\lim_{t \to b^-} \gamma(t)$ exists.
Moreover, if $\gamma\in \FAC([a,b);X)$, then $\gamma(b_-)$ exists even when $b=\infty$.
\end{lemma}

We remark that, for $\gamma\in \FAC^p((a,b);X)$, the right limit
$\gamma(a_+):=\lim_{t \to a^+} \gamma(t)$ may not exist.
This is the reason why we call $\gamma$ a \emph{forward} absolutely continuous curve.
For example, in the Funk space \eqref{Funckmeatirc},
consider the unit speed minimal geodesic $\gamma:(-\log 2,0] \lra \mathbb{B}^n$
such that $\gamma(0)=\mathbf{0}$ and $\gamma(t)$ converges to $(-1,0,\ldots,0)$
in $\mathbb{R}^n$ as $t \to -\log 2$.
Clearly $\gamma \in \FAC((-\log2,0];\mathbb{B}^n)$ with $f \equiv 1$ in (\ref{moredfabsc});
however, $\gamma$ is not defined at $t=-\log 2$.
In fact, $\lim_{t \to -\log 2} d_F(\mathbf{0},\gamma(t))=\infty$.

Thanks to Lemma~\ref{uniformlyconverge}, we always have
$\FAC^p([a,b);X)=\FAC^p([a,b];X)$ for $b<\infty$
and $\FAC([a,\infty);X)=\FAC([a,\infty];X)$.
Hence, we will mainly consider $\FAC^p([a,b];X)$ with $a \in \mathbb{R}$.
{Owing to \cite[Proposition~2.2]{RMS}, there always exists a minimal function $f$
satisfying \eqref{moredfabsc} as follows (we also refer to
\cite[Theorem~1.1.2]{AGS}, \cite[Theorem~3.5]{CRZ} and \cite[Lemma~7.1]{OS}).}

\begin{theorem}[Forward metric derivative]\label{vilocitythem}
Suppose that either $p=1$ with $b\leq \infty$ or $p\in (1,\infty]$ with $b<\infty$.
Then, for any curve $\gamma\in \FAC^p([a,b];X)$, the limit
\[
 |\gamma_+'|(t) :=\lim_{s \to t} \frac{d(\gamma(\min\{s,t\}),\gamma(\max\{s,t\}))}{|t-s|}
\]
exists for $\mathscr{L}^1$-a.e.\ $t \in (a,b)$ and we have
\[
 d\big( \gamma(s),\gamma(t) \big) \leq
 \int^t_s |\gamma_+'|(r) \,{\dd}r \quad \text{ for any } a\leq s\leq t\leq b.
\]
Moreover, $|\gamma_+'|$ belongs to $L^p([a,b])$
and satisfies $|\gamma_+'|(t)\leq f(t)$ for $\mathscr{L}^1$-a.e.\ $t \in (a,b)$
for any $f$ satisfying \eqref{moredfabsc}.
\end{theorem}

We call $|\gamma'_+|$ the \emph{forward metric derivative} of $\gamma$,
and $\mathscr{L}^1$ denotes the $1$-dimensional Lebesgue measure.
{
One can also define the \emph{$p$-backward absolute continuity} by requiring \eqref{moredfabsc}
for $s,t \in I$ with $t \le s$.
Then, due to the asymmetry, the resulting \emph{backward metric derivative} $|\gamma'_-|$
may not coincide with $|\gamma'_+|$.
}

\begin{definition}[Length]
For $\gamma\in \FAC([a,b];X)$, its \emph{length} $L(\gamma)$ is defined by
\[
 L(\gamma):=\int^b_a |\gamma'_+|(t) \,{\dd}t.
\]
\end{definition}

Thanks to Theorem~\ref{vilocitythem} and Lemma~\ref{uniformlyconverge},
it is not difficult to see that
\[
 L(\gamma) =\sup\left\{ \sum_{i=1}^N d\big( \gamma(t_{i-1}),\gamma(t_i) \big)
 \,\middle|\, N \in \mathbb{N},\ a=t_0<\cdots <t_N=b \right\}.
\]

\begin{definition}[Lipschitz curves]
A curve $\gamma: [a,b] \lra X$ is said to be \emph{$C$-Lipschitz} for $C>0$
if it satisfies $d(\gamma(s),\gamma(t)) \leq C(t-s)$ for all $a \leq  s\leq t\leq b$.
\end{definition}

Obviously, a $C$-Lipschitz curve belongs to $\FAC([a,b];X)$ (provided $b<\infty$)
and $|\gamma'_+|(t)\leq C$ for $\mathscr{L}^1$-a.e.\ $t\in (a,b)$.
The next lemma (proved in the same manner as \cite[Lemma 1.1.4]{AGS})
tells that every forward absolutely continuous curve can be viewed as a Lipschitz curve.

\begin{lemma}[Reparametrization]\label{Lisrepare}
Given $\gamma\in \FAC([a,b];X)$ with length $L:=L(\gamma)$, set
\[
 \mathfrak{s}(t) :=\int^t_a |\gamma'_+|(r) \,{\dd}r, \qquad
 \mathfrak{t}(s):=\min\{t\in [a,b] \,|\, \mathfrak{s}(t)=s\}.
\]
Then $\mathfrak{s}:[a,b] \lra [0,L]$ is a non-decreasing absolutely continuous function
and $\mathfrak{t}:[0,L] \lra [a,b]$ is a left-continuous increasing function such that
$\mathfrak{s}(a_+)=0$, $\mathfrak{s}(b_-)=L$ and $\mathfrak{s} \circ \mathfrak{t}(s)=s$.
Moreover, the curve $\hat{\gamma}:=\gamma\circ \mathfrak{t}:[0,L] \lra X$ is
$1$-Lipschitz and satisfies
$\gamma=\hat{\gamma}\circ\mathfrak{s}$ and $|\hat{\gamma}'_+| =1$
$\mathscr{L}^1$-a.e.\ in $[0,L]$.
\end{lemma}

\subsection{Asymmetric metrics on Finsler manifolds}\label{ssc:Finsler}%%%%%%%%%%
%%%%%%%%%%%%%%%%%%%%%%%%%

In this subsection, we discuss the case of Finsler manifolds as a typical example of
asymmetric metric spaces.

\subsubsection{Finsler manifolds}%%%%%%%%%%%
%%%%%%%%%%%%%%%%

We first recall the basics of Finsler geometry;
see \cite{BCS,Obook,Sh1} for further reading.
Let $M$ be an $n$-dimensional connected $C^{\infty}$-manifold without boundary,
and $TM=\bigcup_{x \in M}T_{x}M$ be its tangent bundle.
We call $(M,F)$ a \emph{Finsler manifold} if a nonnegative function $F:TM\lra [0,\infty)$
satisfies
\begin{enumerate}[(1)]
\item $F\in C^{\infty}(TM\setminus \{0\})$;
\item $F(cv)=cF(v)$ for all $v\in TM$ and $c \geq 0$;
\item For any $v\in T_xM \setminus \{0\}$, the $n \times n$ symmetric matrix
\[
 g_{ij}(v):=\frac{1}{2} \frac{\partial^2[F^2]}{\partial v^i \partial v^j}(v)
\]
is positive-definite, where $v=\sum_{i=1}^n v^i (\partial/\partial x^i)|_x$
in a chart $(x^i)_{i=1}^n$ around $x$.
\end{enumerate}
We remark that $g_{ij}(v)$ cannot be defined at $v=0$ unless $F$ is Riemannian.
For $v \in T_xM \setminus \{0\}$, we define a Riemannian metric $g_v$ of $T_xM$
by $g_v(w,\bar{w}) :=\sum_{i,j=1}^n g_{ij}(v) w^i \bar{w}^j$.
Euler's homogeneous function theorem yields that
$F^2(v)=g_v(v,v)$ for any $v \in T_xM \setminus \{0\}$.
The \emph{reversibility} of $F$ on $U\subset M$ is defined as
\[
 \lambda_F(U):=\sup_{v \in TU \setminus \{0\}} \frac{F(-v)}{F(v)}.
\]
If $\lambda_F(M)=1$ (i.e., $F(-v)=F(v)$ for all $v \in TM$),
then we say that $F$ is \emph{reversible}.
Note that $\lambda_F(U)<\infty$ for any compact set $U \subset M$
thanks to the smoothness of $F$.

We define the \emph{dual Finsler metric} $F^*$ of $F$ by
\begin{equation*}
 F^*(\zeta):=\sup_{v \in T_xM \setminus \{0\}} \frac{\zeta(v)}{F(v)}, \quad
 \zeta\in T_x^*M.
\end{equation*}
By definition, we have
\begin{equation}\label{tangninnerprot}
 \zeta(v) \leq F(v)F^*(\zeta).
\end{equation}
Then the \emph{Legendre transformation} $\mathfrak{L}:T_xM \lra T_x^*M$
is defined by $\mathfrak{L}(v):=\zeta$, where $\zeta$ is the unique element satisfying
$F(v)=F^*(\zeta)$ and $\zeta(v)=F^2(v)$.
Note that $\mathfrak{L}:TM \setminus \{0\} \lra T^*M\setminus \{0\}$ is a diffeomorphism.
For $f \in C^1(M)$,
the \emph{gradient vector field} of $f$ is defined by $\nabla f := \mathfrak{L}^{-1}({\dd}f)$.
We have ${\dd}f(v) =g_{\nabla f}(\nabla f,v)$ provided ${\dd}f \neq 0$.
We remark that the grandient $\nabla$ is nonlinear; indeed,
$\nabla(f+h)\neq\nabla f+\nabla h$ and $\nabla(-f) \neq -\nabla f$ in general
(the latter is due to the irreversibility of $F$).
Note also that
\begin{equation}\label{finsergradientdef}
 \limsup_{y \to x} \frac{f(y)-f(x)}{d_F(x,y)}
 =F^* \big( {\dd}f(x) \big) =F\big( \nabla f(x) \big),
\end{equation}
where $d_F$ is the Finsler distance function defined below.

\subsubsection{Length structure and absolutely continuous curves}%%%%%%%%
%%%%%%%%%%%%%%%%%%%%%

Let $\mathcal{A}_\infty([0,1];M)$ denote the class of piecewise smooth curves in $M$
defined on $[0,1]$.
Given $\gamma\in \mathcal{A}_\infty([0,1];M)$, we define its \emph{length} by
\[
 L_F(\gamma):=\int^1_0 F\big( \gamma'(t) \big) \,{\dd}t.
\]
Then the associated \emph{distance function} $d_F:M\times M \lra [0,\infty)$ is defined as
\[
 d_F(x,y):=\inf\{L_F(\gamma) \,|\,
 \gamma\in \mathcal{A}_\infty([0,1];M),\, \gamma(0)=x,\, \gamma(1)=y\}.
\]
Note that $d_F:M\times M \lra [0,\infty)$ is a continuous function
and $(M,d_F)$ is an asymmetric metric space in the sense of Definition~\ref{generalsapcedef}.
Indeed, $d_F(x,y)$ may not coincide with $d_F(y,x)$ unless $F$ is reversible.
Observe also that $\mathcal{T}_+= \mathcal{T}_-$ is exactly the original topology of $M$.

A Finsler manifold is said to be \emph{forward} (resp.\ \emph{backward}) \emph{complete}
if $(M,d_F)$ is forward (resp.\ backward) complete in the sense of Definition~\ref{df:complete}
(see \cite[Theorem~6.6.1]{BCS} for a Finsler analogue of the Hopf--Rinow theorem).
As for an estimate of the reversibility, one has the following (see \cite[Theorem 2.23]{KZ}).

\begin{theorem}\label{reversibifinslerdd}
Let $(M,F)$ be a forward complete Finsler manifold.
Then, for any $\star\in M$, the triple $(M,\star,d_F)$ is a pointed forward $\Theta$-metric space for
\[
 \Theta(r):=\lambda_F\Big( B^+_\star\big( 2r+\lambda_F \big( B^+_\star(r) \big) r \big) \Big).
\]
\end{theorem}

Since $M$ is smooth and the reversibility is finite on every compact set,
the study of absolutely continuous curves can be largely reduced
to the case of Riemannian manifolds or Euclidean spaces.
We only briefly explain for thoroughness.
According to \cite{AB},
we say that a curve $\gamma:[0,1] \lra M$ is \emph{absolutely continuous}
if, for any chart $\varphi:U \lra \mathbb{R}^n$ of $M$, the composition
$\varphi \circ \gamma: \gamma^{-1}(U) \lra \varphi(U) \subset \mathbb{R}^n$
is locally absolutely continuous, i.e., absolutely continuous on any closed subinterval of
$\gamma^{-1}(U)$.
Let $\mathcal{A}_{\ac}([0,1];M)$ denote the class of absolutely continuous curves defined on $[0,1]$.
For any $\gamma\in \mathcal{A}_{\ac}([0,1];M)$,
the derivative $\gamma'(t)$ exists for $\mathscr{L}^1$-a.e.\ $t\in [0,1]$ and we can define
\[
 L_{F}(\gamma):=\int^1_0 F \big( \gamma'(t) \big) \,{\dd}t.
\]
Note that $L_{F}(\gamma)<\infty$ and $F(\gamma')\in L^1([0,1])$.
Moreover, we have
\begin{equation}\label{eBusemanMayer}
 \lim_{\delta \to 0^+} \frac{d(\gamma(t),\gamma(t+\delta))}{\delta}
 =F\big( \gamma'(t) \big) \quad \text{for $\mathscr{L}^1$-a.e.}\ t \in (0,1).
\end{equation}
One can see $\mathcal{A}_{\ac}([0,1];M)=\FAC([0,1];M)$ by an argument similar to that of \cite[Proposition~3.18]{AB}.
Then, for $\gamma\in \FAC([0,1];M)$, we find from \eqref{eBusemanMayer} that
\begin{equation}\label{deriavetisfinsermorem}
 |\gamma'_+|(t)=F \big( \gamma'(t) \big) \quad \text{for $\mathscr{L}^1$-a.e.}\ t\in (0,1).
\end{equation}

\subsection{Upper gradients}\label{ssc:upper}%%%%%%%%%%%%%%%%%
%%%%%%%%%%%%%%%

In this subsection, we introduce upper gradients for functions on asymmetric metric spaces.
First, let us consider the case of Finsler manifolds.

\begin{example}\label{finslergradecase}
Let $(M,d_F)$ be a forward metric space induced by a forward complete Finsler manifold $(M,F)$.
Given $\phi\in C^1(M)$, for any $\gamma\in \FAC([a,b];M)$,
$\phi\circ\gamma$ is absolutely continuous and \eqref{tangninnerprot} yields
\begin{equation}\label{Fcauchyinequality}
 (\phi \circ \gamma)'(t) ={\dd}\phi\big( \gamma'(t) \big)
 \leq F^* \big({\dd}\phi \big( \gamma(t) \big) \big) F\big( \gamma'(t) \big)
 = F\big( \nabla\phi \big( \gamma(t) \big) \big) F\big( \gamma'(t) \big)
\end{equation}
for $\mathscr{L}^1$-a.e.\ $t\in (a,b)$.
Hence, a nonnegative function $\mathfrak{g}:M \lra [0,\infty)$
satisfies $F(\nabla\phi)\leq \mathfrak{g}$ if and only if
\[
 \phi\big( \gamma(t_2) \big) -\phi\big( \gamma(t_1) \big)
 \leq \int^{t_2}_{t_1}\mathfrak{g}\big( \gamma(t) \big) F\big( \gamma'(t) \big) \,{\dd}t
\]
for all $\gamma\in \FAC([a,b];M)$ and $a \leq t_1 \leq t_2 \leq b$.
\end{example}

Now, let $(X,d)$ be a forward complete forward metric space.
In what follows, let $\phi:X \lra (-\infty,\infty]$ denote a \emph{proper} function, i.e.,
its \emph{proper effective domain} $\mathfrak{D}(\phi):=\{x\in X \,|\, \phi(x)<\infty\}$
is nonempty.

\begin{definition}[Strong upper gradients]\label{strongdefg}%[AGS,Def 1.2.1]
A function $\mathfrak{g}:X \lra [0,\infty]$ is called a \emph{strong upper gradient} for $\phi$
if, for every curve $\gamma\in \FAC([a,b];X)$, $\mathfrak{g}\circ \gamma$ is Borel and satisfies
\begin{equation}\label{stronggrad}
 \phi \big( \gamma(t_2) \big) -\phi \big( \gamma(t_1) \big)
 \leq \int^{t_2}_{t_1} \mathfrak{g} \big( \gamma(t) \big) |\gamma'_+|(t) \,{\dd}t
 \quad \text{for all}\ a \leq  t_1\leq t_2 \leq b.
\end{equation}
\end{definition}

{
Notice that, as is natural in view of Example~\ref{finslergradecase}
(see also Example~\ref{finsbackgrad}),
we did not take the absolute value in the left-hand side of \eqref{stronggrad}.
Therefore, our definition of upper gradients is weaker than \cite[Definition~3.6]{CRZ}.
}

If $\mathfrak{g} \circ \gamma\, |\gamma'_+| \in L^1(a,b)$,
then $\phi\circ \gamma$ is absolutely continuous (see the remark below) and
\[
 (\phi\circ \gamma)'(t) \leq \mathfrak{g} \big( \gamma(t) \big) |\gamma'_+|(t) \quad
 \text{for }\mathscr{L}^1\text{-a.e.}\ t\in (a,b).
\]

\begin{remark}\label{trickforwardupp}
For $\gamma\in \FAC([a,b];X)$,
we deduce from \eqref{stronggrad} for $\gamma$ and its reverse curve that
\[
 \big| \phi \big( \gamma(t_1) \big) -\phi \big( \gamma(t_2) \big) \big|
 \leq \Theta\big( d\big( \star,\gamma(0) \big) +L(\gamma) \big)
 \int^{t_2}_{t_1} \mathfrak{g} \big( \gamma(t) \big) |\gamma'_+|(t) \,{\dd}t
\]
for all $a \leq t_1\leq t_2 \leq b$.
Hence, if $\mathfrak{g}\circ \gamma\, |\gamma'_+|\in L^1(a,b)$,
then $\phi\circ\gamma$ is absolutely continuous and
\[
 |(\phi\circ \gamma)'(t)| \leq \Theta\big( d \big( \star,\gamma(0) \big) +L(\gamma) \big)
 \mathfrak{g} \big( \gamma(t) \big) |\gamma'_+|(t) \quad \text{for } \mathscr{L}^1\text{-a.e.}\ t\in (a,b).
\]
\end{remark}

\begin{definition}[Weak upper gradients]\label{weakforwaupp22}%[AGS,Def 1.2.2]
A function $\mathfrak{g}:X \lra [0,\infty]$ is called a \emph{weak upper gradient} for $\phi$
if, for every curve $\gamma\in \FAC([a,b];X)$ satisfying
\begin{enumerate}[(1)]
\item $\mathfrak{g}\circ \gamma\,|\gamma'_+|\in L^1(a,b)$;
\item $\phi\circ\gamma$ is $\mathscr{L}^1$-a.e.\ equal to a function $\varphi$
with finite pointwise variation in $(a,b)$ (see \cite[(1.1.2)]{AGS}),
\end{enumerate}
we have
\begin{equation}\label{weakuppversion}
 \varphi'(t)\leq \mathfrak{g} \big( \gamma(t) \big) |\gamma_+'|(t) \quad
 \text{for } \mathscr{L}^1\text{-a.e.}\ t\in (a,b).
\end{equation}
%In this case, if $\phi\circ \gamma$ is absolutely continuous in $(a,b)$, then (\ref{stronggrad}) holds (by choosing $\varphi:=\phi\circ \gamma$).
\end{definition}

Note that a strong upper gradient is a weak upper gradient.
A sufficient condition for a weak upper gradient to be a strong upper gradient is as follows.

\begin{proposition}\label{weakbeomcestronggra}
Let $\mathfrak{g}$ be a weak upper gradient for $\phi$.
If $\phi\circ \gamma$ is absolutely continuous for every $\gamma\in \FAC([a,b];X)$
with $\mathfrak{g}\circ \gamma\,|\gamma'_+|\in L^1(a,b)$,
then $\mathfrak{g}$ is a strong upper gradient for $\phi$.
\end{proposition}

\begin{proof}
Let $\gamma\in \FAC([a,b];X)$ and $a \leq t_1 \le t_2 \leq b$.
On the one hand, if $\mathfrak{g}\circ \gamma\,|\gamma'_+|\in L^1(t_1,t_2)$,
then $\varphi=\phi \circ \gamma$ satisfies \eqref{weakuppversion} and hence \eqref{stronggrad} holds.
On the other hand,
\eqref{stronggrad} is trivial if $\mathfrak{g}\circ \gamma\,|\gamma'_+| \not\in L^1(t_1,t_2)$.
\end{proof}

\begin{definition}[Slopes]\label{wekafordef} %[AGS,Def 1.2.4]
The \emph{local} and \emph{global $($descending$)$ slopes} of $\phi$ at $x \in \mathfrak{D}(\phi)$
are defined by
\[
 |\partial\phi|(x) :=\limsup_{y \to x} \frac{[\phi(x)-\phi(y)]_+}{d(x,y)}, \qquad
 \mathfrak{l}_{\phi}(x) := \sup_{y\neq x} \frac{[\phi(x)-\phi(y)]_+}{d(x,y)},
\]
respectively, where $[a]_+:=\max\{a,0\}$.
For $x\in X \setminus \mathfrak{D}(\phi)$,
we set $|\partial\phi|(x)=\mathfrak{l}_{\phi}(x):=\infty$.
\end{definition}

The local slope $|\partial\phi|$ represents how fast the function $\phi$ can decrease.
We remark that $|\partial (-\phi)| \neq |\partial \phi|$ even in symmetric metric spaces.

\begin{example}\label{finsbackgrad}
Let $(M,d_F)$ and $\phi$ be as in Example \ref{finslergradecase}.
It follows from  \eqref{finsergradientdef} that $F(\nabla(-\phi)(x)) =|\partial\phi|(x)$.
We remark that $F(\nabla(-\phi))$ may not coincide with either $F(-\nabla\phi)$ or $F(\nabla\phi)$.
Note also that $|\partial\phi|=F(\nabla(-\phi))$ is a strong upper gradient for $-\phi$
by \eqref{deriavetisfinsermorem} and \eqref{Fcauchyinequality}.
\end{example}

In general, we have the following (cf.\ \cite[Theorem~1.2.5]{AGS}).

\begin{theorem}[Slopes are upper gradients]\label{slopeuppgr}%[AGS,Thm1.2.5]
Let $(X,d)$ be a forward complete forward metric space
and $\phi:X \lra (-\infty,\infty]$ be a proper function.
\begin{enumerate}[{\rm (i)}]
\item \label{slope-1}
$|\partial\phi|$ is a weak upper gradient for $-\phi;$

\item \label{slope-2}
If $\phi$ is lower semicontinuous $($with respect to $\mathcal{T}_+)$,
then $\mathfrak{l}_{\phi}$ is lower semicontinuous and a strong upper gradient for $-\phi$.
\end{enumerate}
\end{theorem}

\begin{proof}
\eqref{slope-1}
Let $\gamma \in \FAC([a,b];X)$ and $\varphi$ satisfy the assumptions (1), (2)
in Definition \ref{weakforwaupp22} for $-\phi$.
Set
\[
 A:= \big\{ t\in (a,b) \,\big|\, {-}\phi\big( \gamma(t) \big)=\varphi(t),\,
 \text{$\varphi$ is differentiable at }t,\, |\gamma'_+|(t) \text{ exists} \big\}.
\]
Note that $(a,b) \setminus A$ is $\mathscr{L}^1$-negligible (by Theorem~\ref{vilocitythem}).
To see \eqref{weakuppversion}, it suffices to consider $t\in A$ with $\varphi'(t)>0$.
Since $\varphi'(t)>0$,
we may assume that $\gamma(s) \neq \gamma(t)$ for $s\, (\neq t)$ close to $t$.
Then we have
\begin{align*}
 \varphi'(t)
 &= \lim_{s \to t^+,\, s\in A}\frac{\phi(\gamma(t))-\phi(\gamma(s))}{s-t}
 \leq \limsup_{s \to t^+,\, s\in A} \frac{\phi(\gamma(t))-\phi(\gamma(s))}{d(\gamma(t),\gamma(s))}
 \limsup_{s \to t^+,\, s\in A} \frac{d(\gamma(t),\gamma(s))}{s-t} \\
 &\leq |\partial\phi| \big( \gamma(t) \big) |\gamma'_+|(t).
\end{align*}
Therefore, $|\partial\phi|$ is a weak upper gradient for $-\phi$.

\eqref{slope-2}
We first prove the lower semicontinuity.
It follows from the assumption that $x \longmapsto [\phi(x)-\phi(y)]_+$ is lower semicontinuous.
Hence, for any sequence $x_i \to x$ and $y \neq x$,
we have $x_i \neq y$ for large $i$ and
\[
 \liminf_{i \to \infty} \mathfrak{l}_\phi(x_i) \geq
 \liminf_{i \to \infty} \frac{[\phi(x_i)-\phi(y)]_+}{d(x_i,y)}
 \geq \frac{[\phi(x)-\phi(y)]_+}{d(x,y)}.
\]
Then, taking the supremum in $y \neq x$ furnishes the lower semicontinuity of $\mathfrak{l}_{\phi}$.

Next, we prove that $\mathfrak{l}_\phi$ is a strong upper gradient for $-\phi$.
Since $|\partial\phi|\leq \mathfrak{l}_\phi$ by definition,
$\mathfrak{l}_\phi$ is a weak upper gradient for $-\phi$.
Therefore, in view of Proposition~\ref{weakbeomcestronggra},
it is sufficient to show the following.

\begin{claim}\label{claim}
For any $\gamma\in \FAC([a,b];X)$
with $\mathfrak{l}_\phi \circ \gamma\,|\gamma'_+ |\in L^1(a,b)$,
$-\phi\circ\gamma$ is absolutely continuous.
\end{claim}

To this end, for $\mathfrak{t}:[0,L] \lra [a,b]$ given by Lemma~\ref{Lisrepare}
with $L=L(\gamma)$, we set
\[
 \hat{\gamma} :=\gamma \circ \mathfrak{t}, \quad
 \varphi :=-\phi \circ \hat{\gamma}, \quad
 g :=\mathfrak{l}_\phi \circ \hat{\gamma}.
\]
Since $L<\infty$ and $\hat{\gamma}([0,L])=\gamma([a,b])$,
the triangle inequality yields
\[
 \lambda :=\sup_{s\in [0,L]} \Theta \big( d \big( \star,\hat{\gamma}(s) \big) \big)
 \leq \Theta \big( d \big( \star,\gamma(a) \big) +L \big) <\infty.
\]
Recalling that $\hat{\gamma}$ is a $1$-Lipschitz curve, we have
$d( \hat{\gamma}(s_1),\hat{\gamma}(s_2)) \leq \lambda |s_1-s_2|$ for all $s_1,s_2\in [0,L]$
(regardless of the order of $s_1,s_2$), thereby
\begin{equation}\label{stracontall}
 [\varphi(s_2)-\varphi(s_1)]_+\leq \lambda g(s_1) |s_1-s_2|
 \quad \text{for all}\ s_1,s_2\in [0,L].
\end{equation}
Thus, we have
\begin{equation}\label{phiforumul}
 |\varphi(s_1)-\varphi(s_2)|
 =[\varphi(s_1)-\varphi(s_2)]_+ +[\varphi(s_2)-\varphi(s_1)]_+
 \leq \lambda \big( g(s_1)+g(s_2) \big) |s_1-s_2|.
\end{equation}
Moreover, by the hypothesis $\mathfrak{l}_\phi \circ \gamma \, |\gamma'_+|\in L^1(a,b)$
and Lemma~\ref{Lisrepare}, we find
\[
 \int^L_0 g(s) \,{\dd}s
 =\int_0^L \mathfrak{l}_\phi \big( \gamma \big( \mathfrak{t}(s) \big) \big) \,{\dd}s
 =\int^b_a \mathfrak{l}_\phi \big( \gamma(t) \big) |\gamma'_+|(t) \,{\dd}t
 <\infty.
\]
Therefore, we obtain $g\in L^1(0,L)$, which together with \eqref{phiforumul}
and \cite[Lemma~1.2.6]{AGS} yields that $\varphi$ belongs to $W^{1,1}(0,L)$
with $|\varphi'| \leq 2\lambda g$ and has a continuous representative.

To see that $\varphi$ itself is continuous, on the one hand,
note that $\varphi$ is upper semicontinuous by the lower semicontinuity of $\phi$.
On the other hand, we infer from \eqref{stracontall} and $g\in L^1(0,L)$ that
\[
 \liminf_{\varepsilon \to 0^+} \frac{1}{2\varepsilon} \int^{\varepsilon}_{-\varepsilon} \varphi(s+r) \,{\dd}r
 \geq \varphi(s)
 \quad \text{for all}\ s \in (0,L).
\]
This implies that $\varphi$ is continuous and, since it lives in $W^{1,1}(0,L)$,
absolutely continuous.
Then we observe from $-\phi(\gamma(t))=\varphi(\mathfrak{s}(t))$
that $-\phi\circ \gamma$ is absolutely continuous.
This completes the proof of Claim~\ref{claim}.
\end{proof}

\subsection{Curves of maximal slope}\label{ssc:maxslope}%%%%%%%%%%%
%%%%%%%%%%%%%%%

Let $(X,d)$ be a forward complete forward metric space.
For $p\in [1,\infty]$, denote by $\FAC^p_{\loc}((a,b);X)$
the class of \emph{locally $p$-forward absolutely continuous curves} $\xi$ defined on $(a,b)$,
i.e., $\xi|_{[s,t]}\in \FAC^p([s,t];X)$ for every $a<s<t<b$.
%We call $\xi \in \FAC^1_{\loc}((a,b);X)$ simply a \emph{locally forward absolutely continuous curve}.

\begin{definition}[Curves of maximal slope]\label{maxslp}%[AGS,Def1.3.2]
Let $\phi:X \lra (-\infty,\infty]$ be a proper function,
$\mathfrak{g}$ be a weak upper gradient for $-\phi$, and $p\in (1,\infty)$.
We call $\xi \in \FAC^1_{\loc}((a,b);X)$
a \emph{$p$-curve of maximal slope} for $\phi$ with respect to $\mathfrak{g}$
if $\phi\circ \xi$ is $\mathscr{L}^1$-a.e.\ equal to a non-increasing function $\varphi$ satisfying
\begin{equation}\label{curvemaxforslop}
 \varphi'(t) \leq -\frac{1}{p}|\xi'_+|^p(t)-\frac{1}{q}\mathfrak{g}^q \big( \xi(t) \big)
 \quad \text{for $\mathscr{L}^1$-a.e.\ $t\in (a,b)$},
\end{equation}
where $1/p+1/q=1$.
In the case of $p=2$, we simply call $\xi$ a \emph{curve of maximal slope}.
\end{definition}

In fact, equality holds in \eqref{curvemaxforslop} for $\mathscr{L}^1$-a.e.\ $t\in (a,b)$ as follows.

\begin{proposition}[Energy identity]\label{curaboslpssss}%[AGS,Rem1.3.3]
If $\xi:(a,b) \lra X$ is a $p$-curve of maximal slope for $\phi$ with respect to $\mathfrak{g}$,
then we have $\xi \in \FAC_{\lo}^p((a,b);X)$ and $\mathfrak{g} \circ \xi \in L^q_{\lo}(a,b)$ with
\begin{equation}\label{maxlinequforthreego}
 |\xi'_+|^p(t) =\mathfrak{g}^q \big( \xi(t) \big) =-\varphi'(t) \quad \text{for $\mathscr{L}^1$-a.e.}\ t\in (a,b).
\end{equation}
Moreover, if $\mathfrak{g}$ is a strong upper gradient for $-\phi$,
then $\varphi =\phi\circ \xi$ is locally absolutely continuous in $(a,b)$ and satisfies
the \emph{energy identity}
\begin{equation}\label{energyident}
 \frac1p \int^t_s |\xi'_+|^p(r) \,{\dd}r +\frac1q \int^t_s \mathfrak{g}^q \big( \xi(r) \big) \,{\dd}r
 =\phi\big( \xi(s) \big) -\phi\big( \xi(t) \big)
 \quad \text{for all}\,\ a<s<t<b.
\end{equation}
\end{proposition}

\begin{proof}
Since $\varphi$ is non-increasing, $\varphi'(t)$ is locally integrable.
This together with \eqref{curvemaxforslop} implies $\mathfrak{g} \circ \xi\in L^q_{\lo}(a,b)$
and $|\xi'_+|\in L^p_{\lo}(a,b)$, and hence $\xi \in \FAC_{\lo}^p((a,b);X)$ and
$\mathfrak{g}\circ \xi\, |\xi'_+|\in L^1_{\loc}(a,b)$.
For $\mathscr{L}^1$-a.e.\ $t \in (a,b)$,
since $\mathfrak{g}$ is a weak upper gradient for $-\phi$,
we have $-\varphi'(t)\leq \mathfrak{g}(\xi(t))|\xi'_+|(t)$.
Combining this with the Young inequality yields the reverse inequality to
\eqref{curvemaxforslop}, thereby we obtain \eqref{maxlinequforthreego}.

If $\mathfrak{g}$ is additionally a strong upper gradient,
then $\phi \circ \xi$ is locally absolutely continuous in $(a,b)$ (recall Remark~\ref{trickforwardupp})
and we have $\varphi =\phi \circ \xi$.
Then it follows from the Young inequality and \eqref{curvemaxforslop} that
\begin{align*}
 \phi\big( \xi(s) \big) -\phi\big( \xi(t) \big)
 &\leq \int^t_s \mathfrak{g}\big( \xi(r) \big) |\xi'_+|(r) \,{\dd}r
 \leq \frac1p \int^t_s |\xi'_+|^p(r) \,{\dd}r
 +\frac1q \int^t_s \mathfrak{g}^q \big( \xi (r) \big) \,{\dd}r \\
 &\leq -\int^t_s \varphi'(r) \,{\dd}r
 =\phi\big( \xi(s) \big) -\phi\big( \xi(t) \big).
\end{align*}
This furnishes the energy identity \eqref{energyident}.
\end{proof}

\begin{example}\label{Nonfinslerexamp}
Let $(M,d_F)$ and $\phi$ be as in Example \ref{finslergradecase}.
According to \eqref{deriavetisfinsermorem} and Example \ref{finsbackgrad},
if $\xi:(a,b) \lra M$ is a $p$-curve of maximal slope for $\phi$
with respect to $F(\nabla(-\phi))=F^*(-{\dd}\phi)$, then we have
\begin{equation}\label{finslercasemaximalsop}
 (\phi\circ \xi)'(t) = -F^p\big( \xi'(t) \big) = -F^* \big( {-}{\dd}\phi \big(\xi(t) \big) \big)^q.
\end{equation}
This implies $(\phi\circ \xi)'(t)=-F(\xi'(t)) F^*(-{\dd}\phi(\xi(t)))$,
and hence $\xi'(t)=\alpha(t) \nabla(-\phi)(\xi(t))$ holds for some $\alpha(t) \ge 0$.
Actually, we deduce from \eqref{finslercasemaximalsop} that
\begin{equation}\label{finserpflow}
 \xi'(t)=\begin{cases}
 F^{\frac{2-p}{p-1}} \big( \nabla(-\phi)\big( \xi(t) \big) \big) \cdot \nabla(-\phi)\big( \xi(t) \big) &
 \text{ if } \nabla(-\phi)\big( \xi(t) \big) \neq0,\\
 0 & \text{ if } \nabla(-\phi) \big( \xi(t) \big)=0.
\end{cases}
\end{equation}
In particular, $\xi$ is $C^1$ since $\phi$ is $C^1$,
thereby \eqref{finslercasemaximalsop} holds for all $t \in (a,b)$.
We may rewrite \eqref{finserpflow} as $\mathfrak{j}_p(\xi'(t))=\nabla(-\phi)(\xi(t))$
by introducing an operator $\mathfrak{j}_p:TM \lra TM$ defined by
$\mathfrak{j}_p(v):=F^{p-2}(v)v$ if $v\neq0$ and $\mathfrak{j}_p(0):=0$.
In the case of $p=2$, we obtain the usual gradient flow equation $\xi'(t)=\nabla(-\phi)(\xi(t))$.
We stress that $\xi'(t)=-\nabla\phi(\xi(t))$ holds only when $F$ is reversible.
\end{example}

{
We conclude the section with a comparison to the setting of \cite{RMS}.

\begin{remark}\label{diffbetRMS}
In \cite{RMS}, they considered a convex function $\psi:[0,\infty) \lra [0,\infty]$
satisfying some natural conditions (see \cite[(2.30)]{RMS}) and investigated curves $\xi$ fulfilling
\begin{equation}\label{generalgradientflow}
\varphi'(t) \leq -\psi \big( |\xi'_+|(t) \big) -\psi^* \big( \mathfrak{g}(\xi(t)) \big)
\end{equation}
instead of \eqref{curvemaxforslop},
where $\psi^*$ is the Legendre--Fenchel--Moreau transform of $\psi$.
Choosing $\psi(x)=x^p/p$ recovers \eqref{curvemaxforslop}.
Nonetheless, to establish the corresponding existence theory,
they assumed that $\phi$ is bounded from below (see \cite[(2.19b)]{RMS}),
which is unnecessary in the present paper.
\end{remark}
}

\section{Generalized minimizing movements and curves of maximal slope}\label{topologyexiscurvew}%%
%%%%%%%%%%%%%%%%%%%%%%

\subsection{Problem and strategy}%%%%%%%%%%
%%%%%%%%%%%%%%%%%

Throughout this section, let $(X,d)$ be a forward complete forward metric space,
$\phi:X \lra (-\infty,\infty]$ be a proper function, and $p \in (1,\infty)$.
The main objective of this section is to study the following problem.

\begin{problem}\label{firsslopque}
Given an initial datum $x_0\in \mathfrak{D}(\phi)$,
does there exist a $p$-curve $\xi:(0,\infty) \lra X$ of maximal slope for $\phi$
such that $\lim_{t \to 0}\xi(t)=x_0$?
\end{problem}

We shall solve this problem via a discrete approximation.
We begin with some definitions and notations.

\begin{definition}[Resolvent operator]\label{mosydefam}
We define the \emph{$p$-resolvent operator} by, for $\tau>0$ and $x\in X$,
\[
 J_\tau[x]:=\argmin \Phi(\tau,x;\cdot),
\]
where
\[
 \Phi(\tau,x;y)
 :=\phi(y) +\frac{d^p(x,y)}{p\tau^{p-1}}, \quad y\in X.
\]
That is to say, $y \in J_\tau[x]$ if and only if $\Phi(\tau,x;y) \leq \Phi(\tau,x;z)$ for all $z\in X$.
\end{definition}

Let $P_{\gtau}
 :=\{0=t^0_{\gtau}<t^1_{\gtau}<\cdots<t^k_{\gtau}<\cdots \}$
be a partition of the time interval $[0,\infty)$ corresponding to a sequence of
positive time steps $\gtau =(\tau_k)_{k \ge 1}$ in the sense that
\[
 \tau_k =t^k_{\gtau}-t^{k-1}_{\gtau},\qquad
 \lim_{k \to \infty}t^k_{\gtau} =\sum_{k=1}^\infty \tau_k =\infty.
\]
Set $\|{\gtau}\|:=\sup_{k \geq 1} \tau_k$.
We will consider the following recursive scheme:
\begin{equation}\label{discreteequation}
 \text{Given } \Xi^0_{\gtau} \in X, \text{ whenever } \Xi^1_{\gtau},\ldots, \Xi^{k-1}_{\gtau}
 \text{ are known, take } \Xi^k_{\gtau}\in J_{\tau_k} \big[ \Xi^{k-1}_{\gtau} \big].
\end{equation}
This is a well-known scheme to construct (descending) gradient curves of $\phi$.
The following example in the case of Minkowski spaces may be helpful to understand
the choice of $J_{\tau}[x]$ as above.

\begin{example}\label{finslergradent}
Let $(\mathbb{R}^n,F)$ be a \emph{Minkowski space}, i.e.,
each of its tangent spaces is canonically isometric to $(\mathbb{R}^n,F)$,
and $\phi \in C^1(\mathbb{R}^n)$.
For any $C^1$-curve $\gamma:(-\varepsilon,\varepsilon) \lra \mathbb{R}^n$
with $\gamma(0)=\Xi^{k}_{\gtau}$ and $\gamma'(0)=v$, we have
\[
 \frac{\dd}{{\dd}t}\bigg|_{t=0} d\big( \Xi^{k-1}_{\gtau},\gamma(t) \big)
 =\frac{\dd}{{\dd}t}\bigg|_{t=0} F\big( \gamma(t)-\Xi^{k-1}_{\gtau} \big)
 =\frac{g_{\Xi^k_{\gtau}-\Xi^{k-1}_{\gtau}}({\Xi^k_{\gtau}-\Xi^{k-1}_{\gtau}},v)}{F(\Xi^k_{\gtau}-\Xi^{k-1}_{\gtau})},
\]
provided $\Xi^k_{\gtau} \neq \Xi^{k-1}_{\gtau}$.
Combining this with $\frac{\dd}{{\dd}t}|_{t=0}\Phi(\tau_k,\Xi^{k-1}_{\gtau};\gamma(t))=0$
by the choice \eqref{discreteequation} of $\Xi^{k}_{\gtau}$, we find
\[
 -{\dd}\phi(v) =g_{\Xi^k_{\gtau}-\Xi^{k-1}_{\gtau}}(w,v) =[\mathfrak{L}(w)](v),
 \quad \text{where }\,
 w=\frac{F^{p-2}(\Xi^k_{\gtau}-\Xi^{k-1}_{\gtau})}{\tau^{p-1}_k} (\Xi^k_{\gtau}-\Xi^{k-1}_{\gtau}).
\]
Since $v$ was arbitrary, we arrive at the equation
$w=\mathfrak{L}^{-1}(-{\dd}\phi(\Xi^k_{\gtau})) =\nabla(-\phi)(\Xi^k_{\gtau})$.
By the choice of $w$, this is equivalent to
\[
 \frac{\Xi^k_{\gtau}-\Xi^{k-1}_{\gtau}}{\tau_k}
 =F^{\frac{2-p}{p-1}}\big( \nabla(-\phi)(\Xi^k_{\gtau}) \big) \cdot \nabla(-\phi)(\Xi^k_{\gtau}),
\]
which can be regarded as a discrete version of \eqref{finserpflow}.
\end{example}

\begin{definition}[Discrete solutions]\label{discreesolu}
Given ${\gtau}$, $\Xi^0_{\gtau} \in X$
and a sequence $(\Xi^k_{\gtau})_{k \ge 1}$ solving \eqref{discreteequation},
we define a piecewise constant curve $\overline{\Xi}_{\gtau}:[0,\infty) \lra X$ by
\[
 \overline{\Xi}_{\gtau}(0) :=\Xi^0_{\gtau}, \qquad
 \overline{\Xi}_{\gtau}(t) :=\Xi^k_{\gtau} \quad \text{for }t \in (t^{k-1}_{\gtau},t^k_{\gtau}], \ k\geq 1.
\]
We call $\overline{\Xi}_{\gtau}$ a \emph{discrete solution} corresponding to the partition $P_{\gtau}$.
\end{definition}

Under appropriate conditions on $(X,d)$ and $\phi$,
we shall solve Problem~\ref{firsslopque} in the following steps:
\begin{itemize}
\item
Show that the minimization algorithm \eqref{discreteequation} starting from $x_0$ is solvable;

\item
Find a sequence $(P_{\gtau_m})_m$ of admissible partitions with $\|{\gtau_m}\| \to 0$
such that the discrete solutions $(\overline{\Xi}{}_{{\gtau}_m}^k)_{k \ge 1}$ converge to
a solution to Problem~\ref{firsslopque} with respect to a suitable topology $\sigma$ on $X$.
\end{itemize}

\subsection{Topological assumptions}\label{ssc:top}%%%%%%%%%%%%
%%%%%%%%%%%

In the sequel, we always assume that $\sigma$ is a Hausdorff topology on $X$,
possibly different from $\mathcal{T}_{\pm}$, compatible with $d$ in the following sense:
\begin{enumerate}[(1)]
\item
$\sigma$ is weaker than the forward topology $\mathcal{T}_+$ induced from $d$
(i.e.,  $x_i \to x$ in $\mathcal{T}_+$ implies $x_i \to x$ in $\sigma$);
\item
$d$ is $\sigma$-sequentially lower semicontinuous
(i.e., if $x_i \to x$ and $y_i \to y$ in $\sigma$, then $\liminf_{i \to \infty} d(x_i,y_i)\geq d(x,y)$).
\end{enumerate}
We will denote by $x_i \,\overset{\sigma}{\lra}\, x$
the convergence with respect to the topology $\sigma$.

\begin{remark}[Topology comparison]\label{samelimitsimgad}
\begin{enumerate}[(a)]
\item \label{top-a}
Recall that $d$ is $\mathcal{T}_+$-continuous by Theorem~\ref{topologychara},
thereby $\sigma=\mathcal{T}_+$ always satisfies (1) and (2) above.

\item \label{top-b}
If $x_i \,\overset{\mathcal{T}_-}{\lra}\, x$,
then we deduce from the triangle inequality the following upper semicontinuity:
\[
 \limsup_{i \to \infty}d(x_i,y) \leq \lim_{i \to \infty} \big\{ d(x_i,x)+d(x,y) \big\}
 =d(x,y).
\]

\item \label{top-c}
By the $\sigma$-sequential lower semicontinuity of $d$, the limit under $\sigma$ is unique.
Indeed, if $x$ and $x'$ are $\sigma$-limit points of $(x_i)_{i \ge 1}$,
then $0=\liminf_{i \to \infty}d(x_i,x_i)\geq d(x,x')$ necessarily holds,
thereby $x=x'$.
\end{enumerate}
\end{remark}

We give an example where $\sigma$ is different from $\mathcal{T}_{\pm}$
(see \cite[Remark~2.3.9]{AGS} for another example).

\begin{example}[Randers-like spaces]\label{ex:Randers}
Let $(X,\langle\cdot,\cdot\rangle)$ be a Hilbert space
and $\|x\|:=\sqrt{\langle x,x \rangle}$.
Choose $a \in X$ with $\|a\|<1$ and define a function $d:X\times X \lra [0,\infty)$ by
\[
 d(x,y) :=\|y-x\|+\langle a, y-x \rangle.
\]
Then $(X,d)$ is a $[(1+\|a\|)/(1-\|a\|)]$-metric space (recall Definition~\ref{thetametricspace}), and
$\mathcal{T}_+=\mathcal{T}_-$ coincides with the (strong) topology of $(X,\langle\cdot,\cdot\rangle)$.
Now, let $\sigma$ be the weak topology of $X$.
Since $\|\cdot\|$ is $\sigma$-sequentially lower semicontinuous, so is $d$.
\end{example}

\begin{lemma}\label{sigcompimpdcom}
Every $\sigma$-sequentially compact set $K \subset X$ is forward complete.
\end{lemma}

\begin{proof}
Given a forward Cauchy sequence $(x_i)_{i \ge 1}$ in $K$, on the one hand,
the $\sigma$-sequential compactness of $K$ yields
a subsequence $(x_{i_j})_{j \ge 1}$ of $(x_i)_{i \ge 1}$
converging to a point $x' \in K$ in $\sigma$.
On the other hand,
the forward completeness of $X$ furnishes a point $x \in X$
such that $x_i$ converges to $x$ in $\mathcal{T}_+$.
Since $\sigma$ is weaker than $\mathcal{T}_+$, $x_i$ converges to $x$ in $\sigma$ as well.
Hence, $x_{i_j}$ converges to both $x$ and $x'$ in $\sigma$,
and we find from Remark~\ref{samelimitsimgad}\eqref{top-c} that $x=x'\in K$.
Thus, $K$ is forward complete.
\end{proof}

A set $A \subset X$ is said to be \emph{forward bounded}
if $A \subset B^+_\star(r)$ for some $r>0$.
We remark that, thanks to $\lambda_d(B^+_{\star}(r)) \le \Theta(r)$
in Definition~\ref{thetametricspace},
$A$ is forward bounded if and only if $\sup_{x,y \in A} d(x,y)<\infty$.

Now we introduce our main assumptions on $(X,d)$ and $\phi$.

\begin{assumption}\label{continudef}
\begin{enumerate}[(a)]
\item\textbf{Lower semicontinuity.}\label{ass-a}
$\phi$ is $\sigma$-sequentially lower semicontinuous on forward bounded sets, i.e.,
if $\sup_{i,j}d(x_i,x_j)<\infty$ and $x_i \,\overset{\sigma}{\lra}\, x$,
then we have $\liminf_{i \to \infty} \phi(x_i)\geq \phi(x)$.
(In particular, $\phi$ is $\mathcal{T}_+$-lower semicontinuous.)

\item\textbf{Coercivity.}\label{ass-b}
There exist $\tau_*>0$ and $x_*\in X$ such that
\[
\Phi_{\tau_*}(x_*) :=\inf_{y \in X}\Phi(\tau_*,x_*;y)
 =\inf_{y \in X} \Bigg\{ \phi(y) +\frac{d^p(x_*,y)}{p\tau^{p-1}_*} \Bigg\} >-\infty.
\]

\item\textbf{Compactness.}\label{ass-c}
Every forward bounded set contained in a sublevel set of $\phi$
is relatively $\sigma$-sequentially compact, i.e.,
if a sequence $(x_i)_{i \ge 1}$ in $X$ satisfies
$\sup_i \phi(x_i) <\infty$ and $\sup_{i,j}d(x_i,x_j)<\infty$,
then it admits a $\sigma$-convergent subsequence.
\end{enumerate}
\end{assumption}

\begin{remark}[$\sigma=\mathcal{T}_+$ case]\label{strongassupm}
When $\sigma=\mathcal{T}_+$,
\eqref{ass-a} and \eqref{ass-c} above can be rewritten as follows, respectively:
\begin{enumerate}
\item [(a')]
$\phi$ is $\mathcal{T}_+$-lower semicontinuous;

\item[(c')]\label{basicassc'}
Every forward bounded set in a sublevel set of $\phi$ is relatively compact in $X$.
\end{enumerate}
\end{remark}

The next proposition presents one of the simplest situations
where Assumption~\ref{continudef} holds (cf.\ \cite[Remark~2.1.1]{AGS}).
We remark that $\sigma$ may be different from $\mathcal{T}_+$.

\begin{proposition}\label{simpassump}
Suppose that every sublevel set of $\phi$ is compact in $\mathcal{T}_+$.
Then Assumption~$\ref{continudef}$ holds.
\end{proposition}

\begin{proof}
\eqref{ass-a}
Assume that $x_i \,\overset{\sigma}{\lra}\, x$
and the limit $\alpha:=\lim_{i \to \infty}\phi(x_i)$ exists.
To see $\phi(x) \le \alpha$, suppose $\alpha<\infty$ without loss of generality.
For any $\varepsilon>0$,
$A_{\varepsilon}:=\{ y \in X \,|\, \phi(y) \le \alpha+\varepsilon \}$
is compact by hypothesis, and hence a subsequence of $(x_i)_{i \ge 1}$
converges to some $x' \in A_{\varepsilon}$ in $\mathcal{T}_+$.
Then $x=x' \in A_{\varepsilon}$ by Remark~\ref{samelimitsimgad}\eqref{top-c},
and the arbitrariness of $\varepsilon$ yields $\phi(x) \leq \alpha$ as desired.

\eqref{ass-b} is seen by noticing $\inf_X \phi >-\infty$, which follows from \eqref{top-a}.
\eqref{ass-c} is clear by hypothesis.
\end{proof}

{
\begin{remark}\label{diffbettheree}
In \cite{RMS}, Rossi, Mielke and Savar\'e investigated the \emph{doubly nonlinear evolution equation} (DNE),
which is more general than the gradient flow equation.
Their topological requirements are close to ours, however,
recall from Remark~\ref{diffbetRMS} that they assumed $\inf_X \phi >-\infty$,
which is stronger than the coercivity above and can simplify some arguments below.
We remark that Chenchiah--Rieger--Zimmer's \cite{CRZ}
is also concerned with the existence of ($2$-)curves of maximal slope
in the asymmetric setting.
On the one hand, they assumed the lower semicontinuity of $d$ only in the second argument.
On the other hand, they assumed that the backward convergence implies the forward convergence
(see \cite[Assumption~4.3]{CRZ}; compare it with \cite[Remark 2.9]{RMS} and
Remark~\ref{forwardpointspaceandbackwardones}\eqref{pfms-a})
and used a stronger notion of upper gradient (recall Definition~\ref{strongdefg}).
\end{remark}
}

\subsection{Moreau--Yosida approximation}\label{ssc:MY}%%%%%%%%%%%%
%%%%%%%%%%%%%%

In this subsection, we will present an existence result of solutions to \eqref{discreteequation}.
For this purpose, we recall the  definition of Moreau--Yosida approximation.

\begin{definition}[Moreau--Yosida approximation]\label{myadef}
For $\tau>0$ and $x \in X$, the \emph{Moreau--Yosida approximation} $\Phi_\tau$
is defined as
\[
 \Phi_\tau(x) :=\inf_{y \in X} \Phi(\tau,x;y)
 =\inf_{y \in X}\ \Bigg\{ \phi(y) +\frac{d^p(x,y)}{p\tau^{p-1}} \Bigg\}.
\]
We also set
\[
 \tau_*(\phi) :=\sup\{ \tau>0 \,|\, \Phi_\tau(x)>-\infty \text{ for some }x \in X \}.
\]
\end{definition}

Note that Assumption~\ref{continudef}\eqref{ass-b} is equivalent to $\tau_*(\phi)>0$.
Moreover, we have the following (cf.\ \cite[Lemma~2.2.1]{AGS}).

\begin{lemma}\label{basicestphi} %[AGS, Lemma 2.2.1]
Suppose Assumption~$\ref{continudef}$\eqref{ass-b}.
For $0<\tau<\tau_*\leq \tau_*(\phi)$, set
\[
 \epsilon =\epsilon(p,\tau_*,\tau)
 :=\frac{\tau_*^{p-1}-\tau^{p-1}}{2\tau^{p-1}}>0, \qquad
 C(p,\tau_*,\tau) :=\frac{\mathfrak{C}(p,\epsilon)}{p\tau_*^{p-1}}>0,
\]
where $\mathfrak{C}(p,\epsilon)$ is the constant introduced in Lemma~{\rm \ref{basicpestimate}}.
Then we have
\begin{align}
 \Phi_\tau(x)
 &\geq \Phi_{\tau_*}(x_*)-C(p,\tau_*,\tau)\,d^p(x_*,x),
 \label{existI}\\
 d^p(x,y)
 &\leq \frac{2p\tau^{p-1}\tau_*^{p-1}}{\tau^{p-1}_*-\tau^{p-1}}
 \big\{ \Phi(\tau,x;y) -\Phi_{\tau_*}(x_*) +C(p,\tau_*,\tau) d^p(x_*,x) \big\}
 \label{estimad}
\end{align}
for all $x,y \in X$.
In particular, sublevel sets of $\Phi(\tau,x;\cdot)$ are forward bounded.
\end{lemma}

\begin{proof}
We deduce from Lemma~\ref{basicpestimate} (with $a=d(x,y)$ and $b=d(x_*,x)$)
and the triangle inequality that
\[
 \frac{\tau^{p-1}_*+\tau^{p-1}}{2p\tau^{p-1}\tau_*^{p-1}} d^p(x,y) +C(p,\tau_*,\tau)\,d^p(x_*,x)
 \geq \frac{d^p(x_*,y)}{p\tau_*^{p-1}}
\]
for any $x,y \in X$.
By the definition of $\Phi_{\tau_*}$, this implies
\[
 \phi(y) +\frac{\tau^{p-1}_*+\tau^{p-1}}{2p\tau^{p-1}\tau_*^{p-1}} d^p(x,y)
 +C(p,\tau_*,\tau) d^p(x_*,x)
 \geq \Phi_{\tau_*}(x_*).
\]
Then the first claim \eqref{existI} follows since, for any $y \in X$,
\begin{align*}
 \Phi(\tau,x;y)
 &= \phi(y) +\frac{\tau^{p-1}_*+\tau^{p-1}}{2p\tau^{p-1}\tau_*^{p-1}} d^p(x,y)
  +\frac{\tau_*^{p-1}-\tau^{p-1}}{2p\tau^{p-1}\tau_*^{p-1}}d^p(x,y) \\
 &\geq \Phi_{\tau_*}(x_*)-C(p,\tau_*,\tau) d^p(x_*,x)
  +\frac{\tau_*^{p-1}-\tau^{p-1}}{2p\tau^{p-1}\tau_*^{p-1}}d^p(x,y) \\
 &\geq \Phi_{\tau_*}(x_*)-C(p,\tau_*,\tau) d^p(x_*,x).
\end{align*}
Observe also that the first inequality corresponds to the second claim \eqref{estimad}.
The forward boundedness of sublevel sets of $\Phi(\tau,x;\cdot)$
readily follows from \eqref{estimad}.
\end{proof}

Now we prove the existence of a solution to \eqref{discreteequation},
giving a discrete solution as in Definition~\ref{discreesolu}.

\begin{theorem}[Existence of discrete solutions]\label{existenceofdisc} %[AGS, Corollary 2.2.2]
Suppose Assumption~$\ref{continudef}$\eqref{ass-a}--\eqref{ass-c}.
Then, for every $\tau \in (0,\tau_*(\phi))$ and $x \in X$, we have $J_{\tau}[x] \neq \emptyset$.
In particular, for any $\Xi^0_{\gtau} \in X$ and partition $P_{\gtau}$ with $\|{\gtau}\|<\tau_*(\phi)$,
there exists a discrete solution $\overline{\Xi}_{\gtau}$ corresponding to $P_{\gtau}$.
\end{theorem}

\begin{proof}
Given $c>\Phi_{\tau}(x)$, consider the sublevel set
$A:=\{y \in X \,|\, \Phi(\tau,x;y) \leq c \}$.
Recall from Lemma~\ref{basicestphi} that $A$ is forward bounded,
and hence $d$ is bounded on $A \times A$.
Moreover, for any $y \in A$, we have
\[
 \Phi_{\tau}(x) -\frac{d^p(x,y)}{p\tau^{p-1}}
 \le \phi(y) \le \Phi(\tau,x;y) \le c.
\]
Thus, $\phi$ is also bounded on $A$.

Next, we show that $A$ is $\sigma$-sequentially compact.
For any sequence $(y_i)_{i \ge 1}$ in $A$,
since $\sup_{i,j} d(y_i,y_j)$ and $\sup_i \phi(y_i)$ are bounded,
Assumption~\ref{continudef}\eqref{ass-c} yields a subsequence $(y_{i_k})_{k \ge 1}$
which is $\sigma$-convergent to some $y_{\infty} \in X$.
Then, since both $\phi$ and $d$ are $\sigma$-sequentially lower semicontinuous,
we find $\Phi(\tau,x;y_{\infty}) \le c$, thereby $y_{\infty} \in A$.
Hence, $A$ is $\sigma$-sequentially compact.

By the $\sigma$-sequential compactness of $A$ and
the $\sigma$-sequential lower semicontinuity of $\phi$ and $d$,
we can take $y_* \in A$ with $\Phi(\tau,x;y_*) =\inf_{y \in A}\Phi(\tau,x;y)=\Phi_{\tau}(x)$.
This completes the proof.
\end{proof}

{
\begin{remark}\label{rm:psi-MY}
As in \cite[(3.2)]{RMS},
one can also consider a resolvent operator associated with a convex function $\psi$:
\[
 J_\tau[x]:=\underset{y\in X}{\argmin} \left\{ \phi(y)+\tau\psi\left( \frac{d(x,y)}{\tau} \right) \right\}
\]
corresponding to the equation \eqref{generalgradientflow} in Remark~\ref{diffbetRMS}.
See \cite[Lemma 3.2]{RMS} for the existence of discrete solutions in this context.
\end{remark}
}

In the rest of this subsection,
we study some further properties of the Moreau--Yosida approximation.
For $x \in X$ and $\tau>0$ with $J_\tau[x] \neq \emptyset$, we set
\[
 d_\tau^+(x) :=\sup_{y \in J_\tau[x]}d(x,y), \qquad
 d_\tau^-(x):=\inf_{y \in J_\tau[x]}d(x,y).
\]
We introduce the following assumption for convenience;
note that it is stronger than Assumption~\ref{continudef}\eqref{ass-b}.

\begin{assumption}\label{nonemtpassumpto}
For any $x \in X$ and $\tau \in (0,\tau_*(\phi))$, $J_\tau[x] \neq \emptyset$ holds.
\end{assumption}

\begin{remark}\label{existsju}
By Theorem~\ref{existenceofdisc},
if Assumption~\ref{continudef}\eqref{ass-a}--\eqref{ass-c} hold,
then Assumption~\ref{nonemtpassumpto} holds as well.
See also Remark~\ref{strongassupm}, Proposition~\ref{simpassump} and Remark~\ref{assmupabcimpb}
for some situations where Assumption~\ref{continudef} holds.
\end{remark}

We first discuss some continuity and monotonicity properties
(cf.\ \cite[Lemma~3.1.2]{AGS}).

\begin{lemma}\label{monocontiofphi} %[AGS, Lemma 3.1.2]
Suppose Assumption~$\ref{nonemtpassumpto}$ in \eqref{Phi-2}--\eqref{Phi-5} below.
\begin{enumerate}[{\rm (i)}]
\item \label{Phi-1}
The function $(\tau,x) \longmapsto \Phi_\tau(x)$ is continuous in $(0,\tau_*(\phi)) \times X$.

\item \label{Phi-2}
For any $x \in X$, $0<\tau_0<\tau_1$ and $y_i \in J_{\tau_i}[x]$ $(i=0,1)$, we have
\begin{equation}\label{moninpro}
 \phi(x) \geq \Phi_{\tau_0}(x) \geq \Phi_{\tau_1}(x), \quad
 d(x,y_0) \leq d(x,y_1), \quad
 \phi(x)\geq \phi(y_0) \geq \phi(y_1), \quad
 d^+_{\tau_0}(x) \leq d^-_{\tau_1}(x).
\end{equation}

\item \label{Phi-3}
If $x \in \overline{\mathfrak{D}(\phi)}$, then $\lim_{\tau \to 0}d^+_\tau(x)=0$.

\item \label{Phi-4}
For any $x \in X$,
there exists an at most countable set $\mathscr{N}_x \subset (0,\tau_*(\phi))$ such that
$d^-_\tau(x) =d^+_\tau(x)$ for all $\tau \in (0,\tau_*(\phi)) \setminus \mathscr{N}_x$.

\item \label{Phi-5}
If $\phi$ is $\mathcal{T}_+$-lower semicontinuous, then we have,
for all $x \in \overline{\mathfrak{D}(\phi)}$,
\[
 \lim_{\tau \to 0} \Phi_\tau(x) =\lim_{\tau \to 0} \inf_{y \in J_\tau[x]} \phi(y)
 =\phi(x).
\]
Moreover, $\lim_{\tau \to 0} \Phi_\tau(x)=\phi(x)$ holds for all $x \in X$.
\end{enumerate}
\end{lemma}

\begin{proof}
\eqref{Phi-1}
Take $(\tau,x)\in (0,\tau_*(\phi))\times X$ and a sequence
$((\tau_i,x_i))_{i \ge 1}$ in $(0,\tau_*(\phi)) \times X$ converging to $(\tau,x)$.
On the one hand, for any $y \in X$, we have
\[
 \limsup_{i \to \infty}\Phi_{\tau_i}(x_i)
 \leq \limsup_{i \to \infty} \Phi(\tau_i,x_i;y) =\Phi(\tau,x;y).
\]
Taking the infimum in $y \in X$ yields the upper semicontinuity
$\limsup_{i \to \infty} \Phi_{\tau_i}(x_i) \leq \Phi_\tau(x)$.
On the other hand, to see the lower semicontinuity,
let $(y_i)_{i \ge 1} \subset \mathfrak{D}(\phi)$ be a sequence such that
\[
 \lim_{i \to \infty} \big\{ \Phi(\tau_i,x_i;y_i)-\Phi_{\tau_i}(x_i) \big\} =0.
\]
Since $\sup_{i \ge 1} \Phi(\tau_i,x_i;y_i)<\infty$,
we find from \eqref{estimad} that $D:=\sup_{i \ge 1} d(x_i,y_i) <\infty$.
Thus, the triangle inequality implies $\sup_{i \ge 1} d(x,y_i)<\infty$.
It follows from Lemma~\ref{basicpestimate} that
\[ d^p(x,y_i) \le \{ d(x,x_i) +d(x_i,y_i) \}^p
 \le (1+\epsilon) d^p(x_i,y_i) +\mathfrak{C}(p,\epsilon) d^p(x,x_i) \]
for any $\epsilon>0$.
Hence, we have
\[
 \liminf_{i \to \infty}\Phi_{\tau_i}(x_i)
 = \liminf_{i \to \infty} \Phi(\tau_i,x_i;y_i)
 \geq \liminf_{i \to \infty} \Bigg\{ \phi(y_i) +\frac{d^p(x,y_i)}{p\tau^{p-1}_i} \Bigg\}
 -\frac{\epsilon D^p}{p\tau^{p-1}}
 \geq \Phi_\tau(x) -\frac{\epsilon D^p}{p\tau^{p-1}}.
\]
Letting $\epsilon \to 0$ furnishes the lower semicontinuity
$\Phi_{\tau}(x) \le \liminf_{i \to \infty} \Phi_{\tau_i}(x_i)$,
which completes the proof.

\eqref{Phi-2}
The first claim is clear by the definition of $\Phi_{\tau}(x)$,
and the second claim follows from
\begin{align*}
 \phi(y_0) +\frac{d^p(x,y_0)}{p\tau^{p-1}_0}
 &=\Phi_{\tau_0}(x) \leq \Phi(\tau_0,x;y_1)
 =\Phi_{\tau_1}(x) +\Bigg( \frac{1}{p\tau^{p-1}_0}-\frac{1}{p\tau^{p-1}_1} \Bigg) d^p(x,y_1) \\
 &\leq \phi(y_0) +\frac{d^p(x,y_0)}{p\tau^{p-1}_1}
 +\Bigg( \frac{1}{p\tau^{p-1}_0}-\frac{1}{p\tau^{p-1}_1} \Bigg) d^p(x,y_1).
\end{align*}
Note also that the fourth claim is an immediate consequence of the second claim.
Finally, in the third claim, the first inequality is obvious and the second one
is a consequence of the second claim as
\[
 \phi(y_1) +\frac{d^p(x,y_1)}{p\tau^{p-1}_1}
 \leq \phi(y_0) +\frac{d^p(x,y_0)}{p\tau^{p-1}_1}
 \leq \phi(y_0) +\frac{d^p(x,y_1)}{p\tau^{p-1}_1}.
\]

\eqref{Phi-3}
For $x \in \overline{\mathfrak{D}(\phi)}$, $y_{\tau} \in J_{\tau}[x]$ and {any $y \in \mathfrak{D}(\phi)$},
we deduce from \eqref{existI} that
\begin{equation}\label{loeresmtiss}
 {\infty>\Phi(\tau,x;y) \geq \Phi(\tau,x;y_{\tau})  \geq} \phi(y_\tau) \geq \Phi_\tau(y_\tau)
 \geq \Phi_{\tau_*}(x_*)-C(p,\tau_*,\tau) d^p(x_*,y_\tau).
\end{equation}
Then, since $\tau<\tau_*$, $y_* \in J_{\tau_*}[x]$ satisfies $d(x,y_{\tau}) \le d(x,y_*)$
by the second claim in \eqref{Phi-2} and we find
\[
 {\Phi(\tau,x;y)} \geq \phi(y_\tau) \geq \Phi_{\tau_*}(x_*) -C(p,\tau_*,\tau)
 \big(d(x_*,x)+d(x,y_*) \big)^p.
\]
Combining this with
\begin{equation}\label{clmmaconst}
 \lim_{\tau \to 0} C(p,\tau_*,\tau)=\frac{1}{p\tau^{p-1}_*}<\infty
\end{equation}
from Lemmas~\ref{basicestphi} and \ref{basicpestimate}, we find
\begin{equation}\label{inflimitd}
 \lim_{\tau \to 0} \tau^{p-1}\inf_{y_\tau \in J_\tau[x]} \phi(y_\tau)=0.
\end{equation}
Now, for any $y \in \mathfrak{D}(\phi)$, we have
\[
 d^+_\tau(x)^p
 =\sup_{y_\tau \in J_\tau[x]} p\tau^{p-1} \big( \Phi(\tau,x;y_\tau)-\phi(y_\tau) \big)
 \leq p\tau^{p-1} \phi(y) +d^p(x,y)
 -p\tau^{p-1}\inf_{y_\tau \in J_\tau[x]} \phi(y_\tau).
\]
Then, \eqref{inflimitd} yields $\limsup_{\tau \to 0}d^+_\tau(x)^p \leq d^p(x,y)$
for any $y \in \mathfrak{D}(\phi)$, and $\lim_{\tau \to 0} d^+_\tau(x)=0$
since $x \in \overline{\mathfrak{D}(\phi)}$.

\eqref{Phi-4}
Given $x \in X$, since the function $\tau \longmapsto d_{\tau}^-(x)$ is non-decreasing,
it is continuous except for at most countably many points.
Then, at any continuous point $\tau \in (0,\tau_*(\phi))$,
the last claim in \eqref{Phi-2} implies
\[ d_{\tau}^+(x) \le \lim_{t \to \tau^+} d_t^-(x) =d_{\tau}^-(x) \le d_{\tau}^+(x). \]
Hence, $d_{\tau}^+(x)=d_{\tau}^-(x)$.

\eqref{Phi-5}
Given $x \in X$, we deduce from \eqref{Phi-4} that
$\phi$ is constant on $J_\tau[x]$ for $\tau \in (0,\tau_*(\phi)) \setminus \mathscr{N}_x$.
Moreover, thanks to the monotonicity as in \eqref{Phi-2},
it suffices to show the convergence within $(0,\tau_*(\phi)) \setminus \mathscr{N}_x$.
Thus, we choose $y_\tau \in J_\tau[x]$ for each $\tau \in (0,\tau_*(\phi)) \setminus \mathscr{N}_x$
and consider the convergence of $\phi(y_\tau)$.

First, if $x \in \overline{\mathfrak{D}(\phi)}$,
then $\lim_{\tau \to 0}d^+_\tau(x)=0$ by \eqref{Phi-3}
and hence $y_\tau \to x$ as $\tau \to 0$.
Thus, on the one hand, the $\mathcal{T}_+$-lower semicontinuity of $\phi$ yields
\[
 \liminf_{\tau \to 0}\Phi_\tau(x)
 =\liminf_{\tau \to 0} \bigg\{ \phi(y_\tau) +\frac{d^p(x,y_\tau)}{p\tau^{p-1}} \bigg\}
 \geq \liminf_{\tau \to 0} \phi(y_\tau) \geq \phi(x).
\]
On the other hand, we have
$\phi(x) \geq \limsup_{\tau \to 0} \Phi_\tau(x) \geq \limsup_{\tau \to 0} \phi(y_\tau)$.
Combining these shows the first assertion.

Next, let $x \not\in \overline{\mathfrak{D}(\phi)}$.
On the one hand, clearly $\liminf_{\tau \to 0}d(x,y_\tau)>0$ holds
since $y_{\tau} \in \mathfrak{D}(\phi)$.
On the other hand,
by \eqref{loeresmtiss}, \eqref{clmmaconst}, the triangle inequality and
$d(x,y_\tau)\leq d(x,y_{\tau_*})$ from \eqref{Phi-2}, we find
\[
 \liminf_{\tau \to 0} \phi(y_\tau)
 \geq \Phi_{\tau_*}(x_*)-\frac{1}{p\tau^{p-1}_*} \big( d(x_*,x) +d(x,y_{\tau_*}) \big)^p>-\infty.
\]
This is enough to obtain the claim, indeed,
\[
 \liminf_{\tau \to 0} \Phi_\tau(x)
 \geq \liminf_{\tau \to 0}\phi(y_\tau) +\liminf_{\tau \to 0} \frac{d^p(x,y_\tau)}{p\tau^{p-1}}
 =\infty =\phi(x).
\]
\end{proof}

Next we investigate the derivative of the Moreau--Yosida approximation
(cf.\ \cite[Theorem~3.1.4]{AGS}, {\cite[Lemma~4.5]{RMS}}).

\begin{theorem}[Derivative of $\Phi_{\tau}(x)$]\label{th:Phi'} %[AGS, Theorem 3.1.4]
Suppose Assumption~$\ref{nonemtpassumpto}$.
Then, for any $x \in X$,  the function $\tau \longmapsto \Phi_\tau(x)$
is locally Lipschitz in $(0,\tau_*(\phi))$ and
\begin{equation}\label{dervaofphi}
 \frac{\dd}{{\dd}\tau}[\Phi_\tau(x)]
 =-\frac{p-1}{p} \bigg( \frac{d^+_\tau(x)}{\tau} \bigg)^p
 =-\frac{p-1}{p} \bigg( \frac{d^-_\tau(x)}{\tau} \bigg)^p
\end{equation}
holds for $\mathscr{L}^1$-a.e.\ $\tau \in (0,\tau_*(\phi))$.
Moreover, if in addition $\phi$ is $\mathcal{T}_+$-lower semicontinuous,
then we have
\begin{equation}\label{interdervaofphi}
 \frac{d^p(x,y_\tau)}{p\tau^{p-1}}
 +\int^{\tau}_0 \frac{p-1}{p} \bigg( \frac{d^\pm_r(x)}{r} \bigg)^p \,{\dd}r
 =\phi(x)-\phi(y_\tau)
\end{equation}
for all $\tau\in (0,\tau_*(\phi))$ and $y_\tau \in J_\tau[x]$.
\end{theorem}

\begin{proof}
For $0<\tau_0<\tau_1<\tau_*(\phi)$ and any $y_{\tau_1}\in J_{\tau_1}[x]$,
we have
\[
 \Phi_{\tau_0}(x)-\Phi_{\tau_1}(x)
 \leq \Phi(\tau_0,x;y_{\tau_1}) -\Phi(\tau_1,x;y_{\tau_1})
 =\frac{\tau^{p-1}_1-\tau^{p-1}_0}{p\tau_0^{p-1}\tau_1^{p-1}}d^p(x,y_{\tau_1}).
\]
Note that
\[
 \tau_1^{p-1} -\tau_0^{p-1} \le (p-1)\tau_i^{p-2}(\tau_1 -\tau_0) \quad
 \text{with}\,\
 i=\begin{cases} 0 &\text{for}\ 1<p \le 2, \\ 1 &\text{for}\ 2 \le p<\infty \end{cases}
\]
by the concavity ($1<p \le 2$) or the convexity ($2 \le p<\infty$) of $t^{p-1}$ in $t>0$.
This implies
\[
 \Phi_{\tau_0}(x)-\Phi_{\tau_1}(x)
 \leq \frac{\tau^{p-1}_1-\tau^{p-1}_0}{p\tau_0^{p-1}\tau_1^{p-1}} \big( d^-_{\tau_1}(x) \big)^p
 \leq \frac{(p-1)(\tau_1-\tau_0)}{p\tau_{1-i}^{p-1} \tau_i} \big( d^-_{\tau_1}(x) \big)^p.
\]
We similarly observe
\[
 \Phi_{\tau_0}(x)-\Phi_{\tau_1}(x)
 \geq \frac{\tau^{p-1}_1-\tau^{p-1}_0}{p\tau^{p-1}_0\tau^{p-1}_1} \big( d^+_{\tau_0}(x) \big)^p
 \geq \frac{(p-1)(\tau_1-\tau_0)}{p\tau_{1-i} \tau^{p-1}_i} \big( d^+_{\tau_0}(x) \big)^p.
\]
Combining them furnishes
\[
 \frac{p-1}{p}\frac{(d^+_{\tau_0}(x))^p}{\tau_{1-i} \tau^{p-1}_i}
 \leq \frac{\Phi_{\tau_0}(x)-\Phi_{\tau_1}(x)}{\tau_1-\tau_0}
 \leq \frac{p-1}p\frac{(d^-_{\tau_1}(x))^p}{\tau^{p-1}_{1-i} \tau_i},
\]
which shows that $\tau \longmapsto \Phi_\tau(x)$ is locally Lipschitz in $(0,\tau_*(\phi))$.

Take $\tau \in (0,\tau_*(\phi)) \setminus \mathscr{N}_x$
at where $\Phi_{\tau}(x)$ is differentiable,
with $\mathscr{N}_x$ from Lemma~\ref{monocontiofphi}\eqref{Phi-4}.
Then, by applying the above estimates to $\tau_1 \to \tau^+$ ($\tau_0=\tau$)
and $\tau_0 \to \tau^-$ ($\tau_1=\tau$),
we have
\[
 \frac{p-1}{p}\frac{(d^+_{\tau}(x))^p}{\tau^p}
 \leq -\frac{\dd}{{\dd}\tau}[\Phi_\tau(x)]
 \leq \frac{p-1}{p}\frac{(d^-_{\tau}(x))^p}{\tau^p}.
\]
Therefore, \eqref{dervaofphi} follows.
Integrating \eqref{dervaofphi} on $[\tau_0,\tau] \subset (0,\tau_*(\phi))$ yields
\[
 \Phi_\tau(x) +\int^\tau_{\tau_0} \frac{p-1}{p} \bigg( \frac{d^\pm_r(x)}{r} \bigg)^p \,{\dd}r
 =\Phi_{\tau_0}(x).
\]
Letting $\tau_0 \to 0$, we deduce \eqref{interdervaofphi} from
Lemma~\ref{monocontiofphi}\eqref{Phi-5} and $\Phi_\tau(x)=\Phi(\tau,x;y_\tau)$.
\end{proof}

Now we estimate the local slope $|\partial\phi|$
defined in Definition \ref{wekafordef} (cf.\ \cite[Lemmas~3.1.3, 3.1.5]{AGS}).
Let $\mathfrak{D}(|\partial\phi|) :=\{x\in \mathfrak{D}(\phi) \,|\, |\partial\phi|(x)<\infty\}$
be the proper effective domain of $|\partial\phi|$.
Recall that $q=p/(p-1)$.

\begin{lemma}\label{slopfirstexsts} %[AGS, Lemma 3.1.3]
Suppose Assumption~$\ref{nonemtpassumpto}$.
Then we have, for any $x \in X$, $\tau \in (0,\tau_*(\phi))$ and $y_\tau\in J_{\tau}[x]$,
\[
 |\partial\phi|^q(y_\tau) \leq \frac{d^p(x,y_\tau)}{\tau^p}.
\]
In particular, $y_\tau \in \mathfrak{D}(|\partial\phi|)$ holds and
$\mathfrak{D}(|\partial\phi|)$ is dense in $\overline{\mathfrak{D}(\phi)}$
with respect to $\mathcal{T}_+$.
\end{lemma}

\begin{proof}
Assume $|\partial\phi|(y_\tau)>0$ without loss of generality
and take a sequence $(y_i)_{i \ge 1}$ converging to $y_{\tau}$ with
\[
 \lim_{i \to \infty} \frac{\phi(y_{\tau})-\phi(y_i)}{d(y_{\tau},y_i)} =|\partial\phi|(y_\tau).
\]
Observe from the choice of $y_\tau$ that
\[
 \phi(y_\tau)-\phi(y_i)
 \leq \frac{d^p(x,y_i)}{p\tau^{p-1}} -\frac{d^p(x,y_\tau)}{p\tau^{p-1}}.
\]
Dividing this inequality by $d(y_\tau,y_i)$ and using the mean value theorem,
we obtain the first claim as
\[
 |\partial\phi|(y_\tau)
 \leq \frac{d^{p-1}(x,y_{\tau})}{\tau^{p-1}} \limsup_{i \to \infty} \frac{d(x,y_i)-d(x,y_{\tau})}{d(y_{\tau},y_i)}
 \le \frac{d^{p-1}(x,y_\tau)}{\tau^{p-1}}.
\]
Thus $y_\tau \in \mathfrak{D}(|\partial\phi|)$,
and the density of $\mathfrak{D}(|\partial\phi|)$ in $\overline{\mathfrak{D}(\phi)}$
follows from Lemma~\ref{monocontiofphi}\eqref{Phi-3}.
\end{proof}

\begin{lemma}\label{dulecolp} %[AGS, Lemma 3.1.5]
Suppose Assumption~$\ref{nonemtpassumpto}$.
Then we have
\begin{equation}\label{contrllphandtheota}
 \limsup_{\tau \to 0} \frac{\phi(x)-\Phi_\tau(x)}{\tau}
 =\frac{|\partial\phi|^q(x)}{q}
 \quad \text{for all}\,\ x \in \mathfrak{D}(\phi).
\end{equation}
Moreover, there exists a sequence $\tau_i \to 0$ such that, for $y_{\tau} \in J_{\tau}[x]$,
\begin{equation}\label{delphi}
 |\partial\phi|^q(x)
 =\lim_{i \to \infty} \frac{\phi(x)-\phi(y_{\tau_i})}{\tau_i}
 =\lim_{i \to \infty} \frac{d^p(x,y_{\tau_i})}{\tau_i^p}
 \geq \liminf_{\tau \to 0}|\partial\phi|^q(y_\tau).
\end{equation}
\end{lemma}

\begin{proof}
We first show \eqref{contrllphandtheota}.
By the Young inequality, for any $s>0$, we have
\begin{equation}\label{normestimate}
 \frac{s^q}{q}
 =\sup_{t>0} \bigg( st-\frac{t^p}{p} \bigg)
 =\sup_{t>a} \bigg( st-\frac{t^p}{p} \bigg)
\end{equation}
for all $a \in (0,s^{1/(p-1)})$.
This implies
\begin{align*}
 \limsup_{\tau \to 0} \frac{\phi(x)-\Phi_\tau(x)}{\tau}
 &=\limsup_{\tau \to 0} \bigg\{ \frac{\phi(x)-\phi(y_\tau)}{\tau} -\frac{d^p(x,y_\tau)}{p\tau^p} \bigg\} \\
 &\leq \limsup_{\tau \to 0} \bigg\{ \frac{[\phi(x)-\phi(y_\tau)]_+}{d(x,y_\tau)}\frac{d(x,y_\tau)}{\tau}
 -\frac{d^p(x,y_\tau)}{p\tau^p} \bigg\} \\
 &\leq \limsup_{\tau \to 0} \bigg\{ |\partial\phi|(x) \frac{d(x,y_\tau)}{\tau}
 -\frac{d^p(x,y_\tau)}{p\tau^p} \bigg\}
 \leq \frac{|\partial\phi|^q(x)}{q},
\end{align*}
where we used Lemma~\ref{monocontiofphi}\eqref{Phi-3} in the second inequality.
If $|\partial\phi|(x)=0$,
then we are done with the help of the first claim in \eqref{moninpro}.
To show the reverse inequality under $|\partial\phi|(x)>0$,
we again use \eqref{normestimate} to see
\begin{align*}
 \frac{|\partial\phi|^q(x)}{q}
 &= \limsup_{y \to x} \frac1q \bigg( \frac{[\phi(x)-\phi(y)]_+}{{d(x,y)}} \bigg)^q \\
 &= \limsup_{y \to x} \sup_{0<\tau<d(x,y)/a}
 \bigg\{ \frac{[\phi(x)-\phi(y)]_+}{d(x,y)} \frac{d(x,y)}{\tau} -\frac{d^p(x,y)}{p\tau^{p}} \bigg\}
\end{align*}
for any $a \in (0,|\partial\phi|^{1/(p-1)}(x))$.
Hence, for small $\varepsilon>0$,
\[
 \frac{|\partial\phi|^q(x)}{q}
 \leq \sup_{y \neq x} \sup_{0<\tau<\varepsilon}
 \bigg\{ \frac{\phi(x)-\phi(y)}{\tau} -\frac{d^p(x,y)}{p\tau^p} \bigg\}
 = \sup_{0<\tau<\varepsilon} \frac{\phi(x)-\Phi_\tau(x)}{\tau}.
\]
Letting $\varepsilon \to 0$ shows
\[
 \frac{|\partial\phi|^q(x)}{q} \leq \limsup_{\tau \to 0} \frac{\phi(x)-\Phi_\tau(x)}{\tau}
\]
and completes the proof of \eqref{contrllphandtheota}.

We next prove \eqref{delphi}.
First of all,
the last inequality in \eqref{delphi} immediately follows from Lemma~\ref{slopfirstexsts}.
Now, on the one hand, \eqref{contrllphandtheota} furnishes a sequence $\tau_i \to 0$ such that
\begin{equation}\label{eq:delphi+}
 \lim_{i \to \infty} \frac{\phi(x)-\Phi_{\tau_i}(x)}{\tau_i} =\frac{|\partial\phi|^q(x)}{q}.
\end{equation}
On the other hand, we find from the above argument that
\[
 \frac{|\partial\phi|^q(x)}{q}
 =\limsup_{\tau \to 0} \bigg\{ |\partial\phi|(x)\frac{d(x,y_\tau)}{\tau} -\frac{d^p(x,y_\tau)}{p\tau^p} \bigg\}.
\]
Comparing this with the Young inequality \eqref{normestimate},
taking a subsequence if necessary, we obtain
\[
 |\partial\phi|^q(x) =\lim_{i \to \infty} \bigg( \frac{d(x,y_{\tau_i})}{\tau_i} \bigg)^p.
\]
Finally, substituting this into \eqref{eq:delphi+} yields
\[
 |\partial\phi|^q(x)
 =q \lim_{i \to \infty} \bigg\{ \frac{\phi(x)-\phi(y_{\tau_i})}{\tau_i}
 -\frac{d^p(x,y_{\tau_i})}{p\tau_i^p} \bigg\}
 =q \lim_{i \to \infty} \frac{\phi(x)-\phi(y_{\tau_i})}{\tau_i} -(q-1)|\partial\phi|^q(x).
\]
This completes the proof of \eqref{delphi}.
\end{proof}

\subsection{A priori estimates for discrete solutions and a compactness result}%%%%%
%%%%%%%%%%%%%%%%%%

We introduce De Giorgi's variational interpolation associated with
a solution to the recursive scheme \eqref{discreteequation}
(cf.\ \cite[Definition~3.2.1]{AGS}).

\begin{definition}[Variational interpolation]\label{df:var} %[AGS, Definition 3.2.1]
Given a solution $(\Xi^k_{\gtau})_{k \ge 0}$ to \eqref{discreteequation},
we denote by $\widetilde{\Xi}_{\gtau}:[0,\infty) \lra X$ its arbitrary interpolation satisfying
\[
 \widetilde{\Xi}_{\gtau}(t^k_{\gtau}) =\Xi^k_{\gtau},\qquad
 \widetilde{\Xi}_{\gtau}(t^{k-1}_{\gtau}+\delta)\in J_\delta \big[\Xi^{k-1}_{\gtau} \big] \quad
 \text{for}\ \delta \in (0,t^k_{\gtau} -t^{k-1}_{\gtau}).
\]
\end{definition}

Define a function $G_{\gtau}$ on $(0,\infty)$ associated with the discrete solution
$(\Xi^k_{\gtau})_{k \ge 0}$ as, for $\delta \in (0,t^k_{\gtau} -t^{k-1}_{\gtau}]$,
\[
 G_{\gtau}(t^{k-1}_{\gtau}+\delta)
 :=\frac{d^+_\delta(\Xi^{k-1}_{\gtau})}{\delta}
 \geq \frac{d(\Xi^{k-1}_{\gtau},\widetilde{\Xi}_{\gtau}(t^{k-1}_{\gtau}+\delta))}{\delta}.
\]
Observe that $G_{\gtau}$ is a Borel function by \eqref{moninpro}.
Moreover, we find from Lemma~\ref{slopfirstexsts} that
\begin{equation}\label{niforgG}
 |\partial\phi|^q \big( \widetilde{\Xi}_{\gtau}(t) \big)
 \leq G^p_{\gtau}(t)
 \quad \text{for}\ t>0.
\end{equation}
We also define a piecewise constant function $|\Xi'_{{\gtau}}|$ on $(0,\infty)$ by
\begin{equation}\label{detrivtauU'}
 |\Xi'_{\gtau}|(t)
 :=\frac{d(\Xi^{k-1}_{\gtau},\Xi^k_{\gtau})}{t^k_{\gtau}-t^{k-1}_{\gtau}}
 =\frac{d(\Xi^{k-1}_{\gtau},\Xi^k_{\gtau})}{\tau_k} \quad
 \text{for}\ t \in (t^{k-1}_{\gtau},t^k_{\gtau}].
\end{equation}

Then we have the following a priori estimates
(cf.\ \cite[Lemma~3.2.2, Remark~3.2.5]{AGS}, {\cite[Proposition~4.7]{RMS}}).
Recall Definition~\ref{discreesolu} for the piecewise constant interpolation
$\overline{\Xi}_{\gtau}(t)= \Xi^k_{\gtau}$ for $t \in (t^{k-1}_{\gtau},t^k_{\gtau}]$,
and Notation~\ref{notationX} that $(X,\star,d)$ is always a pointed forward $\Theta$-metric space.

\begin{lemma}[A priori estimates]\label{prooriesima} %[AGS, Lemma 3.2.2]
Suppose that $\phi$ is $\mathcal{T}_+$-lower semicontinuous
and Assumption~$\ref{nonemtpassumpto}$ holds,
and let $(\Xi^k_{\gtau})_{k \ge 0}$ be a solution to \eqref{discreteequation}.
\begin{enumerate}[{\rm (i)}]
\item \label{apri-1}
If $\|{\gtau}\|\in (0,\tau_*(\phi))$, then, for each couple of integers $0\leq k<l$, we have
\[
 \frac1p \int^{t^l_{\gtau}}_{t^k_{\gtau}} |\Xi'_{\gtau}|^p(t) \,{\dd}t
 +\frac1q \int^{t^l_{\gtau}}_{t^k_{\gtau}} G^p_{\gtau}(t) \,{\dd}t
 +\phi(\Xi^l_{\gtau}) =\phi(\Xi^k_{\gtau}).
\]

\item \label{apri-2}
For any $x_* \in X$ and $S,T>0$, there exist positive constants
$C_1=C_1(p,x_*,\tau_*(\phi),S,T)$ and $C_2=C_2(\Theta,p,x_*,\tau_*(\phi),S,T)$ such that, if
\[
 \phi(\Xi^0_{\gtau}) \leq S,\qquad
 d^p(x_*,\Xi^0_{\gtau}) \leq S,\qquad
 t^{N-1}_{\gtau} \leq T,\qquad
 \|{\gtau}\| \leq \frac{\tau_*(\phi)}{2^{p/(p-1)}p^2},
\]
then we have
\begin{align}
 &d^p(x_*,\Xi^k_{\gtau}) \leq C_1,\ \quad
 \sum_{i=1}^k \frac{d^p(\Xi^{i-1}_{\gtau},\Xi^i_{\gtau})}{p\tau^{p-1}_i}
 \leq \phi(\Xi^0_{\gtau}) -\phi(\Xi^k_{\gtau}) \leq C_1
 \quad \text{for all}\,\ 1 \leq k \leq N,
 \label{eq:apriori1} \\
 &\max\left\{ d^p \big( \overline{\Xi}_{\gtau}(t),\widetilde{\Xi}_{\gtau}(t) \big),
 d^p \big( \widetilde{\Xi}_{\gtau}(t),\overline{\Xi}_{\gtau}(t) \big) \right\}
 \leq C_2 \|{\gtau}\|^{p-1}
 \quad \text{for all}\,\ t \in [0,t_{\gtau}^N].
 \label{eq:apriori2}
\end{align}
\end{enumerate}
\end{lemma}

\begin{proof}
\eqref{apri-1}
It follows from \eqref{interdervaofphi} that, for each $k+1 \le i \le l$,
\[
 \phi(\Xi^{i-1}_{\gtau}) -\phi(\Xi^i_{\gtau})
 =\frac{d^p(\Xi^{i-1}_{\gtau},\Xi^i_{\gtau})}{p\tau^{p-1}_i}
 +\frac1q \int_{t_{\gtau}^{i-1}}^{t_{\gtau}^i} G_{\gtau}^p(t) \,{\dd}t
 =\frac1p \int_{t_{\gtau}^{i-1}}^{t_{\gtau}^i} |\Xi'_{\gtau}|^p(t) \,{\dd}t
 +\frac1q \int_{t_{\gtau}^{i-1}}^{t_{\gtau}^i} G_{\gtau}^p(t) \,{\dd}t.
\]
Summing up in $k+1 \le i \le l$ shows the claim.

\eqref{apri-2}
Observe from \eqref{apri-1} that
\begin{equation}\label{threedintg}
 \sum_{i=1}^k \frac{d^p(\Xi^{i-1}_{\gtau},\Xi^i_{\gtau})}{p\tau_i^{p-1}}
 =\frac1p \int^{t^k_{\gtau}}_0 |\Xi'_{\gtau}|^p(t) \,{\dd}t
 \leq \phi(\Xi^0_{\gtau}) -\phi(\Xi^k_{\gtau}).
\end{equation}
Note also that
\[
 \frac{d^p(x_*,\Xi^i_{\gtau})}{p} -\frac{d^p(x_*,\Xi^{i-1}_{\gtau})}{p}
 \leq d(\Xi^{i-1}_{\gtau},\Xi^i_{\gtau})\,d^{p-1}(x_*,\Xi^i_{\gtau}),
\]
which follows from the convexity of the function $t^p$ in $t>0$
and the triangle inequality if $d(x_*,\Xi^i_{\gtau})>d(x_*,\Xi^{i-1}_{\gtau})$
(and it is clear if $d(x_*,\Xi^i_{\gtau}) \le d(x_*,\Xi^{i-1}_{\gtau})$).
We deduce from these inequalities and the Young inequality that,
for any $\varepsilon>0$ and $\tau_*:=\tau_*(\phi)/p<\tau_*(\phi)$,
\begin{align*}
\frac{d^p(x_*,\Xi^k_{\gtau})}{p} -\frac{d^p(x_*,\Xi^0_{\gtau})}{p}
 &\leq \sum_{i=1}^k d(\Xi^{i-1}_{\gtau},\Xi^i_{\gtau})  d^{p-1}(x_*,\Xi^i_{\gtau}) \\
 &\leq \sum_{i=1}^k \Bigg\{ \varepsilon^{p-1} \frac{d^p(\Xi^{i-1}_{\gtau},\Xi^i_{\gtau})}{p\tau^{p-1}_i}
 +\frac{p-1}{p\varepsilon} \tau_i d^p(x_*,\Xi^i_{\gtau}) \Bigg\} \\
 &\leq \varepsilon^{p-1} \big( \phi(\Xi^0_{\gtau}) -\phi(\Xi^k_{\gtau}) \big)
 +\frac{p-1}{p\varepsilon} \sum_{i=1}^k \tau_i d^p(x_*,\Xi^i_{\gtau}) \\
 &\leq \varepsilon^{p-1} \big( \phi(\Xi^0_{\gtau}) -\Phi_{\tau_*}(x_*) \big)
 +\frac{\varepsilon^{p-1}}{p\tau^{p-1}_*} d^p(x_*,\Xi^k_{\gtau})
 +\frac{p-1}{p\varepsilon} \sum_{i=1}^k \tau_i d^p(x_*,\Xi^i_{\gtau}).
\end{align*}
Substituting $\varepsilon=2^{-1/(p-1)}\tau_*$ and recalling our hypotheses, we obtain
\begin{align*}
 d^p(x_*,\Xi^k_{\gtau})
 &\leq 2 d^p(x_*,\Xi^0_{\gtau}) +p\tau^{p-1}_* \big( \phi(\Xi^0_{\gtau}) -\Phi_{\tau_*}(x_*) \big)
 +\frac{2^{p/(p-1)}(p-1)}{\tau_*} \sum_{i=1}^k \tau_i d^p(x_*,\Xi^i_{\gtau}) \\
 &\leq 2S+p\tau^{p-1}_* S-p\tau^{p-1}_*\Phi_{\tau_*}(x_*)
 +\frac{2^{p/(p-1)}(p-1)}{\tau_*} \sum_{i=1}^k \tau_i d^p(x_*,\Xi^i_{\gtau}).
\end{align*}
Then it follows from \cite[Lemma~3.2.4]{AGS} with
\[
 a_i =d^p(x_*,\Xi^i_{\gtau}), \quad
 A =2S+p\tau^{p-1}_* S-p\tau^{p-1}_* \Phi_{\tau_*}(x_*), \quad
 \alpha =\frac{2^{p/(p-1)}(p-1)}{\tau_*}, \quad
 m =\alpha\|{\gtau}\| \leq \frac{p-1}{p}
\]
that
\[
 a_k \leq B {\ee}^{\beta t^{k-1}_{\gtau}} \leq B {\ee}^{\beta T},
\]
where $B=A/(1-m)$ and $\beta=\alpha/(1-m)$.
We remark that $A \geq S>0$ since
\[
 \Phi_{\tau_*}(x_*)
 \leq \phi(\Xi_{\gtau}^0) +\frac{d^p(x_*,\Xi_{\gtau}^0)}{p\tau_*^{p-1}}
 \leq S+\frac{S}{p\tau_*^{p-1}}. \]
Note also that $\Phi_{\tau_*}(x_*)$ depends only on $p$, $\tau_*(\phi)$ and $x_*$ (and $\phi$).
Thus, we obtain the first claim in \eqref{eq:apriori1}.
Moreover, we find from \eqref{threedintg} that
\[
 \sum_{i=1}^k \frac{d^p(\Xi^{i-1}_{\gtau},\Xi^i_{\gtau})}{p\tau^{p-1}_i}
 \leq \phi(\Xi^0_{\gtau}) -\phi(\Xi^k_{\gtau})
 \leq \phi(\Xi^0_{\gtau}) -\Phi_{\tau_*}(x_*) +\frac{d^p(x_*,\Xi^k_{\gtau})}{p\tau^{p-1}_*},
\]
which together with the first claim implies the second claim in \eqref{eq:apriori1}.

As for \eqref{eq:apriori2},
since $\widetilde{\Xi}_{\gtau}(t) \in J_{t-t^{k-1}_{\gtau}}[\Xi^{k-1}_{\gtau}]$
for $t\in (t^{k-1}_{\gtau}, t^k_{\gtau}]$ with $1 \le k \le N$,
we see from the second claim in \eqref{moninpro},
the former claim in \eqref{eq:apriori1} and Definition~\ref{thetametricspace} that
\begin{align*}
 d\big( \overline{\Xi}_{\gtau}(t),\widetilde{\Xi}_{\gtau}(t) \big)
 &\leq d\big( \Xi^k_{\gtau},\Xi^{k-1}_{\gtau} \big) +d\big( \Xi^{k-1}_{\gtau},\widetilde{\Xi}_{\gtau}(t) \big)
 \leq d\big( \Xi^k_{\gtau},\Xi^{k-1}_{\gtau} \big) +d\big( \Xi^{k-1}_{\gtau},\Xi^k_{\gtau} \big) \\
 &\leq \big( \Theta \big( d(\star,x_*)+C_1^{1/p} \big) +1 \big)
 d\big( \Xi^{k-1}_{\gtau},\Xi^k_{\gtau} \big).
\end{align*}
Combining this with the latter claim in \eqref{eq:apriori1} furnishes
\[
 d^p \big(\overline{\Xi}_{\gtau}(t),\widetilde{\Xi}_{\gtau}(t) \big)
 \leq \big( \Theta \big( d(\star,x_*)+C_1^{1/p} \big)+1 \big)^p p\|{\gtau}\|^{p-1} C_1.
\]
Furthermore, one can bound $d^p(\widetilde{\Xi}_{\gtau}(t),\overline{\Xi}_{\gtau}(t))$
by Definition~\ref{thetametricspace}, thanks to
\begin{align*}
 d\big( \star,\overline{\Xi}_{\gtau}(t) \big)
 &\leq d(\star,x_*) +d\big( x_*,\Xi_{\gtau}^k \big)
 \leq d(\star,x_*) +C_1^{1/p}, \\
 d\big( \star,\widetilde{\Xi}_{\gtau}(t) \big)
 &\leq  d\big( \star,\overline{\Xi}_{\gtau}(t) \big)
 +d\big( \overline{\Xi}_{\gtau}(t),\widetilde{\Xi}_{\gtau}(t) \big) \\
 &\leq d(\star,x_*) +C_1^{1/p}
 +\big( \Theta \big( d(\star,x_*)+C_1^{1/p} \big)+1 \big) (p\|{\gtau}\|^{p-1} C_1)^{1/p}
\end{align*}
for $t\in (t^{k-1}_{\gtau}, t^k_{\gtau}]$ with $1 \le k \le N$.
This completes the proof.
\end{proof}

We shall need the following version of Ascoli--Arzel\`{a} theorem.
Owing to Lemma~\ref{sigcompimpdcom},
it is shown in the same way as \cite[Proposition 3.3.1]{AGS} and hence we omit the proof.

\begin{proposition}\label{proconveglmma} %[AGS, Prop 3.3.1]
For $T>0$ and a $\sigma$-sequentially compact set $K \subset X$,
let $\xi_i:[0,T] \lra K$, $i \geq 1$, be a sequence of $($possibly discontinuous$)$ curves
such that
\[
 \limsup_{i \to \infty} d\big( \xi_i(s),\xi_i(t) \big) \leq \omega(s,t) \quad
 \text{for all}\ s,t \in [0,T]
\]
for a symmetric function $\omega:[0,T]\times [0,T] \lra [0,\infty)$ satisfying
$\lim_{(s,t) \to (r,r)} \omega(s,t)=0$ for all $r \in [0,T] \setminus \mathscr{C}$,
where $\mathscr{C} \subset [0,T]$ is an at most countable set.
Then there exist a subsequence $(\xi_{i_j})_{j \ge 1}$
and a limit curve $\xi:[0,T] \lra X$ such that
$\xi_{i_j}(t) \,\overset{\sigma}{\lra}\, \xi(t)$ for all $t \in [0,T]$
and $\xi$ is $\mathcal{T}_+$-continuous in $[0,T] \setminus \mathscr{C}$.
\end{proposition}

We remark that $\omega(s,s)>0$ can occur for $s \in \mathscr{C}$.
For example, let $\mu$ be a nonnegative finite measure on $[0,T]$ and set
$\omega(s,t)=\omega(t,s):=\mu([s,t])$ for $0 \leq s \leq t \leq T$.
Then $\mathscr{C}$ is the set of atoms of $\mu$ (see \cite[Remark~3.3.2]{AGS}).
We also need the following variant of the local slope $|\partial\phi|$.

\begin{definition}[Relaxed slope]\label{wekarkparg} %[AGS, (2.3.1)]
We define the \emph{relaxed slope} of $\phi$ by
\[
 |\partial^-\phi|(x) :=\inf\Big\{ \liminf_{i \to \infty} |\partial\phi|(x_i)
 \,\Big|\, x_i \,\overset{\sigma}{\lra }\, x,\, \sup_i \big\{ d(x,x_i),\phi(x_i) \big\} <\infty \Big\}
\]
for $x \in \mathfrak{D}(\phi)$, and $|\partial^-\phi|(x):=\infty$ otherwise.
\end{definition}

Note that, by choosing $x_i=x$ for all $i$, $|\partial^-\phi|(x) \le |\partial\phi|(x)$ always holds.

\begin{remark}\label{lowerbackreamrk}
If $|\partial\phi|$ is $\sigma$-sequentially lower semicontinuous,
then we have $|\partial\phi|(x) \leq |\partial^-\phi|(x)$ for any $x \in \mathfrak{D}(\phi)$,
and hence $|\partial^-\phi|=|\partial\phi|$ holds.
\end{remark}

Now, under Assumption~\ref{continudef}, we have a compactness result
(generalizing \cite[Corollary~3.3.4]{AGS}; {see also \cite[Proposition~4.9]{RMS}}).

\begin{theorem}[Compactness]\label{discreecoroll} %[AGS, Cor 3.3.4]
Suppose that Assumption~$\ref{continudef}$\eqref{ass-a}--\eqref{ass-c} hold.
Let $\Lambda$ be a family of sequences of time steps $\gtau$
corresponding to partitions of $[0,\infty)$ such that
$\inf_{\gtau \in \Lambda}\|{\gtau}\|=0$,
and let $\{\Xi^0_{\gtau}\}_{\gtau\in \Lambda}$ be a family of initial data satisfying
\[
 \Xi^0_{\gtau} \,\overset{\sigma}{\lra}\, x_0
 \,\ \text{and}\,\ \phi(\Xi^0_{\gtau}) \to \phi(x_0)\,\ \text{as}\ \|{\gtau}\| \to 0,\qquad
 \sup_{\gtau\in \Lambda} d(x_0,\Xi^0_{\gtau}) <\infty,
\]
where $x_0 \in \mathfrak{D}(\phi)$ is a fixed point.
Then there exist a sequence $(\gtau_{\alpha})_{\alpha \ge 1}$ in $\Lambda$
with $\|{\gtau_{\alpha}}\| \to 0$, a limit curve $\xi \in \FAC^p_{\loc}([0,\infty);X)$,
a non-increasing function $\varphi:[0,\infty) \lra \mathbb{R}$,
and $A \in L^p_{\loc}([0,\infty))$ such that
\begin{enumerate}[{\rm (i)}]
\item\label{frstconve1}
$\overline{\Xi}_{\gtau_{\alpha}}(t) \,\overset{\sigma}{\lra}\, \xi(t)$ and
$\widetilde{\Xi}_{\gtau_{\alpha}}(t) \,\overset{\sigma}{\lra}\, \xi(t)$
as $\alpha \to \infty$ for all $t \geq 0;$

\item\label{frstconve2}
$\varphi(t)
 =\lim_{\alpha \to \infty} \phi\big( \overline{\Xi}_{\gtau_{\alpha}}(t) \big)\geq \phi\big( \xi(t) \big)$
for all $t \geq 0$, and $\xi(0)=x_0;$

\item\label{frstconve3}
$|\Xi'_{\gtau_{\alpha}}| \to A$ weakly in $L^p_{\loc}\big( [0,\infty) \big)$,
and $A(t) \geq |\xi'_+|(t)$ for $\mathscr{L}^1$-a.e.\ $t \in (0,\infty);$

\item\label{frstconve4}
$\liminf_{\alpha \to \infty} G_{\gtau_{\alpha}}^p(t) \geq |\partial^-\phi|^q \big( \xi(t) \big)$
for all $t>0$.
\end{enumerate}
\end{theorem}

\begin{proof}
By hypotheses, we can choose a sequence $(\gtau_{\alpha})_{\alpha \ge 1}$
in $\Lambda$ such that $\|{\gtau_{\alpha}}\| \to 0$ and
\[
 \|{\gtau_{\alpha}}\| \leq \frac{\tau_*(\phi)}{2^{p/(p-1)}p^2},\qquad
 |\phi(\Xi^0_{\gtau_{\alpha}})- \phi(x_0)| <1.
\]
In order to apply Lemma~\ref{prooriesima} (with the help of Remark~\ref{existsju}),
set $S:=\max\{\phi(x_0)+1,\sup_{\gtau \in \Lambda}d^p(x_0,\Xi^0_{\gtau}) \}$.
Then, for $T>0$,
$N(\alpha)$ with $t^{N(\alpha)-1}_{\gtau_{\alpha}} \leq T<t^{N(\alpha)}_{\gtau_{\alpha}}$
and $C:=\max\{C_1,C_2\}$ in Lemma~\ref{prooriesima}, we have
\begin{align}
 &d^p(x_0,\Xi^k_{\gtau_{\alpha}}) \leq C, \qquad
 \sum_{i=1}^k
 \frac{d^p(\Xi^{i-1}_{\gtau_{\alpha}},\Xi^i_{\gtau_{\alpha}})}{p\tau^{p-1}_{\alpha,i}}
 \leq \phi(\Xi^0_{\gtau_{\alpha}}) -\phi(\Xi^k_{\gtau_{\alpha}})
 \leq C\quad \text{for}\ 1 \leq k \leq N(\alpha),
 \label{trangleinequaconvex}\\
 &\max \big\{ d^p \big( \overline{\Xi}_{\gtau_{\alpha}}(t),\widetilde{\Xi}_{\gtau_{\alpha}}(t) \big),
 d^p \big( \widetilde{\Xi}_{\gtau_{\alpha}}(t),\overline{\Xi}_{\gtau_{\alpha}}(t) \big) \big\}
 \leq C \|{\gtau_{\alpha}}\|^{p-1} \quad \text{for}\ t\in [0,t^{N(\alpha)}_{\gtau_{\alpha}}],
 \label{**duwtilcont}
\end{align}
where we put $\gtau_{\alpha}=(\tau_{\alpha,i})_{i \ge 1}$.
Note that \eqref{trangleinequaconvex} implies
\[
 \sup_{\alpha} \sup_{t \in [0,T]} d^p\big( x_0,\overline{\Xi}_{\gtau_{\alpha}}(t) \big) \leq C,\qquad
 \sup_{\alpha} \sup_{t \in [0,T]} \phi\big( \overline{\Xi}_{\gtau_{\alpha}}(t) \big)
 \leq \sup_{\alpha} \phi(\Xi^0_{\gtau_{\alpha}}) \leq S.
\]
Now, consider a set
\[
 \mathscr{K} :=\{ x \in X \,|\, d^p(x_0,x) \leq C,\, \phi(x)\leq S \},
\]
and observe $\overline{\Xi}_{\gtau_{\alpha}}([0,T]) \subset \mathscr{K}$.
It follows from Assumption~\ref{continudef}\eqref{ass-a}, \eqref{ass-c} that
$\mathscr{K}$ is $\sigma$-sequentially compact.

Note that $\varphi_{\alpha}(t):=\phi(\overline{\Xi}_{\gtau_{\alpha}}(t))$ is non-increasing by definition.
Therefore, by passing to a subsequence if necessary,
we may assume that $\varphi(t):=\lim_{\alpha \to \infty}\varphi_{\alpha}(t)$ exists for all $t \in[0,T]$
and $\varphi$ is a non-increasing function (see \cite[Lemma~3.3.3]{AGS}).
Moreover, by a diagonal argument,
one can assume that $\varphi(t):=\lim_{\alpha \to \infty}\varphi_{\alpha}(t)$ exists for all $t \geq 0$.
We also find from Lemma~\ref{prooriesima}\eqref{apri-1} and \eqref{trangleinequaconvex} that
\[
 \int^T_0 |\Xi'_{\gtau_{\alpha}}|^p(t) \,{\dd}t
 \leq p\big\{ \phi\big( \overline{\Xi}_{\gtau_{\alpha}}(0) \big)
 -\phi\big( \overline{\Xi}_{\gtau_{\alpha}}(T) \big) \big\}
 \leq pC,
\]
which together with the reflexivity of $L^p([0,T])$ yields
a weakly convergent subsequence of $(|\Xi'_{\gtau_{\alpha}}|)_{\alpha \ge 1}$ in $L^p([0,T])$.
Again by a diagonal argument, we can assume that $(|\Xi'_{\gtau_{\alpha}}|)_{\alpha \ge 1}$
converges weakly to some function $A \in L_{\loc}^p([0,\infty))$.

Let $P_{\gtau_{\alpha}}$ be the partition of $[0,\infty)$ corresponding to $\gtau_{\alpha}$.
For $0 \leq s<t$, set $s_{\alpha} :=\max\{ r \in P_{\gtau_{\alpha}} \,|\, r \le s \}$
and $t_{\alpha} :=\min\{ r \in P_{\gtau_{\alpha}} \,|\, t \le r \}$.
Then we have
\[
 d\big( \overline{\Xi}_{\gtau_{\alpha}}(s),\overline{\Xi}_{\gtau_{\alpha}}(t) \big)
 \leq \int^{t_{\alpha}}_{s_{\alpha}} |\Xi'_{\gtau_{\alpha}}|(r) \,{\dd}r
\]
by the definition \eqref{detrivtauU'} of $|\Xi'_{\gtau_{\alpha}}|$, and hence
\begin{equation}\label{locl2u}
 \limsup_{\alpha \to \infty}
 d\big( \overline{\Xi}_{\gtau_{\alpha}}(s),\overline{\Xi}_{\gtau_{\alpha}}(t) \big)
 \leq \int^t_s A(r) \,{\dd}r.
\end{equation}
Therefore, since $\mathscr{K} \subset \overline{B^+_{\star}(d(\star,x_0)+C^{1/p})}$,
we can apply Proposition~\ref{proconveglmma}
to $\overline{\Xi}_{\gtau_{\alpha}}:[0,T] \lra \mathscr{K}$
with $\omega(s,t)=\omega(t,s):=\Theta(d(\star,x_0)+C^{1/p})\int^t_s A(r) \,{\dd}r$
to find a subsequence of $(\overline{\Xi}_{\gtau_{\alpha}})_{\alpha \ge 1}$
$\sigma$-converging to a $\mathcal{T}_+$-continuous curve $\xi:[0,T] \lra \mathscr{K}$.
Together with \eqref{**duwtilcont} and by a diagonal argument,
we obtain \eqref{frstconve1}.

It is straightforward from Assumption~\ref{continudef}\eqref{ass-a} and \eqref{frstconve1} that
\[
 \varphi(t) =\lim_{\alpha \to \infty} \phi\big( \overline{\Xi}_{\gtau_\alpha}(t) \big)
 \geq \phi\big( \xi(t) \big).
\]
Moreover, the assumption $\Xi^0_{\gtau_\alpha} \,\overset{\sigma}{\lra}\, x_0$
combined with \eqref{frstconve1} furnishes (as in Remark~\ref{samelimitsimgad}\eqref{top-c})
\[
 d\big( \xi(0),x_0 \big)
 \leq \liminf_{\alpha \to \infty} d\big( \overline{\Xi}_{\gtau_\alpha}(0),\Xi^0_{\gtau_\alpha} \big)
 =0,
\]
which shows $\xi(0)=x_0$ and completes the proof of \eqref{frstconve2}.
Next, we deduce from \eqref{frstconve1} and \eqref{locl2u} that
\[
 d\big( \xi(s),\xi(t) \big)
 \leq \liminf_{\alpha \to \infty}
 d\big( \overline{\Xi}_{\gtau_\alpha}(s),\overline{\Xi}_{\gtau_\alpha}(t) \big)
 \leq \limsup_{\alpha \to \infty}
 d\big( \overline{\Xi}_{\gtau_\alpha}(s),\overline{\Xi}_{\gtau_\alpha}(t) \big)
 \leq \int^t_s A(r) \,{\dd}r,
\]
which implies that $|\xi'_+|(t) \leq A(t)$ holds for $\mathscr{L}^1$-a.e.\ $t \in (0,\infty)$.
Thus we have \eqref{frstconve3} and, since $A \in L^p_{\loc}([0,\infty))$,
$\xi \in \FAC^p_{\loc}([0,\infty);X)$.
Moreover, it follows from \eqref{frstconve1} and \eqref{niforgG} that, for all $t>0$,
\[
 |\partial^-\phi|^q \big( \xi(t) \big)
 \leq \liminf_{\alpha \to \infty} |\partial\phi|^q \big( \widetilde{\Xi}_{\gtau_\alpha}(t) \big)
 \leq \liminf_{\alpha \to \infty} G^p_{\gtau_\alpha}(t).
\]
This yields \eqref{frstconve4} and completes the proof.
\end{proof}

\subsection{Generalized minimizing movements}\label{generaexstsnce}%
%%%%%%%%%%%%%%%%%%%%%%%%

In this subsection, we will take advantage of generalized minimizing movements
to obtain the existence of curves of maximal slope (recall Definition~\ref{maxslp}).

\begin{definition}[Generalized minimizing movements]\label{defominvcurve} %[AGS, Def 2.0.6]
Let $p \in (1,\infty)$ and $x_0 \in X$.
We say that a curve $\xi:[0,\infty) \lra X$ is
a \emph{$p$-generalized minimizing movement} for $\phi$ starting from $x_0$
if $\xi(0)=x_0$ and
there is a sequence $({\gtau}_{\alpha})_{\alpha \ge 1}$
of $\|{\gtau}_\alpha\| \to 0$ along with corresponding discrete solutions
$(\overline{\Xi}_{{\gtau}_\alpha})_{\alpha \ge 1}$ such that
\[
 \lim_{\alpha \to \infty} \phi(\Xi^0_{{\gtau}_\alpha}) =\phi(x_0), \qquad
 \limsup_{\alpha \to \infty} d(x_0,\Xi^0_{{\gtau}_\alpha}) <\infty, \qquad
 \overline{\Xi}_{{\gtau}_\alpha}(t) \,\overset{\sigma}{\lra}\, \xi(t)\,\ \text{for all}\ t \geq 0.
\]
We denote by $\GMM_p(\phi;x_0)$ the collection of
$p$-generalized minimizing movements for $\phi$ starting from $x_0$.
\end{definition}

We remark that the construction of discrete solutions depends on $p$.
Note also that, under $\Xi^0_{{\gtau}_\alpha} \,\overset{\sigma}{\lra}\, \xi(0)$,
$\xi(0)=x_0$ is equivalent to $\Xi^0_{{\gtau}_\alpha} \,\overset{\sigma}{\lra}\, x_0$
(see Remark~\ref{samelimitsimgad}\eqref{top-c}).
The following variant of Theorem~\ref{discreecoroll}
ensures that $\GMM_p(\phi;x_0)$ is nonempty under mild assumptions.

\begin{corollary}\label{compacttexistsencesolu}
Suppose that Assumption~$\ref{continudef}$\eqref{ass-a}--\eqref{ass-c} hold.
Let $\Lambda$ be a family of sequences of time steps $\gtau$
corresponding to partitions of $[0,\infty)$ with $\inf_{{\gtau}\in \Lambda}\|{\gtau}\|=0$.
If a family of initial data $\{\Xi^0_{\gtau}\}_{{\gtau} \in \Lambda}$ satisfies
\[
 \sup_{{\gtau} \in \Lambda} \phi(\Xi^0_{\gtau})<\infty, \qquad
 \sup_{{\gtau} \in \Lambda} d(x_0,\Xi^0_{\gtau})<\infty,
\]
then there exist a sequence $({\gtau}_\alpha)_{\alpha \ge 1}$ in $\Lambda$
with $\|{\gtau}_\alpha\| \to 0$ and $\xi \in \FAC^p_{\loc}([0,\infty);X)$
such that $\overline{\Xi}_{{\gtau}_\alpha}(t) \,\overset{\sigma}{\lra}\, \xi(t)$
for all $t \geq 0$.
Moreover, if $\Xi^0_{{\gtau}_\alpha} \,\overset{\sigma}{\lra}\, x_0$
and $\phi(\Xi^0_{\gtau_\alpha}) \to \phi(x_0)$ as $\alpha \to \infty$,
then we have $\xi \in \GMM_p(\phi;x_0)$.
\end{corollary}

\begin{proof}
Since $\inf_{{\gtau}\in \Lambda}\|{\gtau}\|=0$,
we can choose a sequence $({\gtau}_{\alpha})_{\alpha \ge 1}$ with $\|{\gtau}_{\alpha} \| \to 0$.
By letting $S:=\max\{\sup_{{\gtau} \in \Lambda} \phi(\Xi^0_{\gtau}),
 \sup_{{\gtau} \in \Lambda} d^p(x_0,\Xi^0_{\gtau}) \}$,
the existence of $\xi$ is shown in the same way as in Theorem~\ref{discreecoroll}.
Moreover, if $\Xi^0_{{\gtau}_\alpha} \,\overset{\sigma}{\lra}\, x_0$,
then we find $\xi(0)=x_0$
and $\phi(\Xi^0_{\gtau_\alpha}) \to \phi(x_0)$ implies $\xi \in \GMM_p(\phi;x_0)$.
\end{proof}

We remark that, in Theorem~\ref{discreecoroll}, we assumed $x_0 \in \mathfrak{D}(\phi)$
to obtain an upper bound of $\sup_{{\gtau}\in \Lambda} \phi(\Xi^0_{\gtau})$
and apply Lemma~\ref{prooriesima}\eqref{apri-2}.
In Corollary~\ref{compacttexistsencesolu}, however,
we assumed $\sup_{{\gtau}\in \Lambda} \phi(\Xi^0_{\gtau})<\infty$ instead and
$x_0\in \mathfrak{D}(\phi)$ was removed.
Note also that $\sup_{{\gtau}\in \Lambda} \phi(\Xi^0_{\gtau})<\infty$
and $\phi(\Xi^0_{\gtau_\alpha}) \to \phi(x_0)$ imply $x_0\in \mathfrak{D}(\phi)$.

In the next two theorems (generalizing \cite[Theorems~2.3.3, 2.3.1]{AGS}),
we see that $\GMM_p(\phi;x_0)$ consists of $p$-curves of maximal slope.
Recall \eqref{energyident} for the energy identity of $p$-curves of maximal slope.

\begin{theorem}\label{existsassupab} %[AGS, Theorem 2.3.3}
Suppose that Assumptions~$\ref{continudef}$\eqref{ass-a}
and $\ref{nonemtpassumpto}$ hold
and $|\partial^-\phi|$ is a strong upper gradient for $-\phi$.
Let $x_0 \in \mathfrak{D}(\phi)$ and $(\overline{\Xi}_{{\gtau}_\alpha})_{\alpha \ge 1}$ be a sequence
of discrete solutions with a curve $\xi:[0,\infty) \lra X$
such that $\xi(0)=x_0$ and
\[
 \|{\gtau}_\alpha\| \to 0, \quad
 \phi(\Xi^0_{{\gtau}_\alpha}) \to \phi(x_0), \quad
 \sup_{\alpha} d(x_0,\Xi^0_{{\gtau}_\alpha}) <\infty, \quad
 \overline{\Xi}_{{\gtau}_\alpha}(t) \,\overset{\sigma}{\lra}\, \xi(t)\,\ \text{for all}\ t \geq 0.
\]
Then we have $\xi \in \FAC^p_{\loc}([0,\infty);X)$ and
\begin{enumerate}[{\rm (i)}]
\item\label{exstasolucon1}
$\lim_{\alpha \to \infty} \phi \big( \overline{\Xi}_{{\gtau}_\alpha}(t) \big)
 =\phi \big( \xi(t) \big)$
for all $t \geq 0;$

\item\label{exstasolucon2}
$|\Xi'_{{\gtau}_\alpha}| \to |\xi'_+|$
in $L^p_{\loc} \big( [0,\infty) \big);$

\item\label{exstasolucon3}
$|\partial\phi| \circ \overline{\Xi}_{{\gtau}_\alpha} \to |\partial^-\phi| \circ \xi$
in $L^q_{\loc} \big( [0,\infty) \big)$.
\end{enumerate}
In particular, every curve $\xi \in \GMM_p(\phi;x_0)$
is a $p$-curve of maximal slope for $\phi$ with respect to $|\partial^-\phi|$
and we have the energy identity
\begin{equation}\label{eq:energy}
 \frac1p \int^T_0 |\xi'_+|^p(t) \,{\dd}t
 +\frac1q \int^T_0 |\partial^-\phi|^q \big( \xi(t) \big) \,{\dd}t +\phi\big( \xi(T) \big)
 =\phi(x_0) \quad
 \text{for all}\,\ T>0.
\end{equation}
 \end{theorem}

\begin{proof}
We obtain from Theorem~\ref{discreecoroll}
a subsequence $(\gtau_{\alpha_i})_{i \ge 1}$
of $(\gtau_\alpha)_{\alpha \ge 1}$, denoted by $(\gtau_i)_{i \ge 1}$ for simplicity,
satisfying Theorem~\ref{discreecoroll}\eqref{frstconve1}--\eqref{frstconve4}.
We remark that, in the proof of Theorem~\ref{discreecoroll},
Assumption~\ref{continudef}\eqref{ass-c} was used only to find a convergent subsequence
and a limit curve $\xi$ satisfying Theorem~\ref{discreecoroll}\eqref{frstconve1}.
In the current theorem, $\overline{\Xi}_{\gtau_\alpha}(t) \,\overset{\sigma}{\lra}\, \xi(t)$
is included in the hypotheses
and we do not need Assumption~\ref{continudef}\eqref{ass-c}.

It follows from Theorem~\ref{discreecoroll}, Fatou's lemma
and Lemma~\ref{prooriesima}\eqref{apri-1} that
\begin{align}
 &\frac1p \int^T_0 |\xi'_+|^p(t) \,{\dd}t
 +\frac1q \int^T_0 |\partial^-\phi|^q \big( \xi(t) \big) \,{\dd}t
 +\phi\big( \xi(T) \big) \nonumber\\
 &\leq \frac1p \int^T_0 A^p(t) \,{\dd}t
 +\frac1q \int^T_0 \liminf_{i \to \infty} G^p_{\gtau_i}(t) \,{\dd}t
 +\lim_{i \to \infty} \phi\big( \overline{\Xi}_{\gtau_i}(T) \big) \nonumber\\
 &\leq \liminf_{i \to \infty} \Bigg\{
 \frac1p \int^T_0 |\Xi'_{\gtau_i}|^p(t) \,{\dd}t
 +\frac1q \int^T_0 G^p_{\gtau_i}(t) \,{\dd}t
 +\phi\big( \overline{\Xi}_{\gtau_i}(T) \big)  \Bigg\} \nonumber\\
 &= \lim_{i \to \infty} \phi(\Xi_{\gtau_i}^0) =\phi(x_0).
 \label{leftgmmu}
\end{align}
Furthermore, since $|\partial^-\phi|$ is a strong upper gradient for $-\phi$ by assumption,
Definition~\ref{strongdefg} furnishes
\begin{equation}\label{rightgumm}
 \phi(x_0) = \phi\big( \xi(0) \big)
 \leq \phi\big( \xi(T) \big) +\int^T_0 |\partial^-\phi| \big( \xi(t) \big) |\xi'_+|(t) \,{\dd}t.
\end{equation}
Combining this with \eqref{leftgmmu} shows
\[
 \frac1p \int^T_0 |\xi'_+|^p(t) \,{\dd}t
 +\frac1q \int_0^T |\partial^-\phi|^q \big( \xi(t) \big) \,{\dd}t
 \leq \int^T_0 |\partial^-\phi| \big(\xi(t) \big) |\xi'_+|(t) \,{\dd}t.
\]
Therefore, we obtain from the equality condition of the Young inequality that
\begin{equation}\label{roughslopmaxo11}
 |\xi'_+|^p(t) =|\partial^-\phi|^q \big( \xi(t) \big) \quad \text{for $\mathscr{L}^1$-a.e.}\ t \in (0,\infty),
\end{equation}
and all the inequalities in \eqref{leftgmmu} and \eqref{rightgumm} are in fact equalities.
In particular, \eqref{leftgmmu} yields \eqref{eq:energy} and
\begin{align}
 &|\Xi'_{\gtau_i}| \to |\xi'_+|
 \,\ \text{in}\ L^p_{\loc} \big( [0,\infty) \big),
 \label{eq:A-limit}\\
 &\liminf_{i \to \infty} G^p_{\gtau_i} =|\partial^-\phi|^q\circ \xi
 \,\ \text{in}\ L^1_{\loc} \big( [0,\infty) \big), \qquad
 \lim_{i \to \infty} \phi\big( \overline{\Xi}_{\gtau_i}(t) \big)
 =\phi\big( \xi(t) \big)
 \,\ \text{for all}\ t \geq 0.
 \nonumber
\end{align}
Thus, we have \eqref{exstasolucon1} and  \eqref{exstasolucon2} for $(\gtau_i)_{i \ge 1}$.

We next show \eqref{exstasolucon3}.
On the one hand, observe from Definition~\ref{wekarkparg} that
$\liminf_{i \to \infty} |\partial\phi| (\overline{\Xi}_{\gtau_i}(t)) \geq |\partial^-\phi|(\xi(t))$.
On the other hand, we deduce from
\eqref{roughslopmaxo11}, \eqref{eq:A-limit}, \eqref{detrivtauU'} and Lemma~\ref{slopfirstexsts} that
\begin{align*}
 \int_0^T |\partial^-\phi|^q \big( \xi(t) \big) \,{\dd}t
 &= \int_0^T |\xi'_+|^p \,{\dd}t
 =\lim_{i \to \infty} \int_0^T |\Xi'_{\gtau_i}|^p \,{\dd}t
 \geq \limsup_{i \to \infty} \int_0^T |\partial\phi|^q \big( \overline{\Xi}_{\gtau_i}(t) \big) \,{\dd}t \\
 &\geq \liminf_{i \to \infty} \int_0^T |\partial\phi|^q \big( \overline{\Xi}_{\gtau_i}(t) \big) \,{\dd}t
 \geq \int_0^T \liminf_{i \to \infty} |\partial\phi|^q \big( \overline{\Xi}_{\gtau_i}(t) \big) \,{\dd}t.
\end{align*}
Hence, we obtain \eqref{exstasolucon3} for $(\gtau_i)_{i \ge 1}$.
Note finally that,
since every subsequence of $(\gtau_\alpha)_{\alpha \ge 1}$ includes a further subsequence
satisfying \eqref{exstasolucon1}--\eqref{exstasolucon3},
the original sequence $(\gtau_\alpha)_{\alpha \ge 1}$
necessarily satisfies \eqref{exstasolucon1}--\eqref{exstasolucon3}.

For $\xi \in \GMM_p(\phi;x_0)$, one can apply the above argument
and obtain the energy identity \eqref{eq:energy}.
Moreover, we deduce from the equality in \eqref{rightgumm} as well as \eqref{roughslopmaxo11} that
\[
 \frac{{\dd}}{{\dd}t} \phi \big(\xi(t) \big)
 =-|\partial^-\phi| \big( \xi(t) \big) |\xi'_+|(t)
 =-\frac1p |\xi'_+|^p(t) -\frac1q |\partial^-\phi|^q \big( \xi(t) \big)
\]
for $\mathscr{L}^1$-a.e.\ $t \in (0,\infty)$.
Therefore, $\xi$ is a $p$-curve of maximal slope for $\phi$ with respect to $|\partial^-\phi|$.
\end{proof}

{
See \cite[Theorem~3.5]{RMS} and \cite[Theorem~4.21]{CRZ} for the results
corresponding to Theorem~\ref{existsassupab} in their settings.
Theorem~\ref{existsassupab} cannot be deduced from them (and vice versa),
due to the difference of assumptions we explained in Remarks~\ref{diffbetRMS} and \ref{diffbettheree}.
}

\begin{theorem}\label{gmmslopI} %[AGS, Theorem 2.3.1}
Suppose that Assumptions~$\ref{continudef}$\eqref{ass-a} and $\ref{nonemtpassumpto}$ hold.
If $|\partial^-\phi|$ is a weak upper gradient for $-\phi$
and $\phi$ satisfies the \emph{continuity condition:}
\[
 \sup_{i \ge 1} \big\{ |\partial\phi|(x_i),\, d(x,x_i),\, \phi(x_i) \big\}<\infty,
 \ x_i \,\overset{\sigma}{\lra}\, x
 \quad \Longrightarrow\quad \phi(x_i) \to \phi(x),
\]
then every $\xi \in \GMM_p(\phi;x_0)$ with $x_0 \in \mathfrak{D}(\phi)$
is a $p$-curve of maximal slope for $\phi$ with respect to $|\partial^-\phi|$.
\end{theorem}

\begin{proof}
Let $(\gtau_\alpha)_{\alpha \ge 1}$ be as in Definition~\ref{defominvcurve}.
Similarly to Theorem~\ref{existsassupab},
we can choose a subsequence $(\gtau_i)_{i \ge 1}$ of $(\gtau_\alpha)_{\alpha \ge 1}$
satisfying Theorem~\ref{discreecoroll}\eqref{frstconve1}--\eqref{frstconve4}.
Setting $\varphi(t)=\lim_{i \to \infty} \phi(\overline{\Xi}_{\gtau_i}(t))$ as in Theorem~\ref{discreecoroll},
we infer from \eqref{leftgmmu} and Lemma~\ref{prooriesima}\eqref{apri-1} that
\begin{align}
 &\frac1p \int^t_s |\xi'_+|^p(r) \,{\dd}r
 +\frac1q \int^t_s |\partial^-\phi|^q\big( \xi(r) \big) \,{\dd}r
 \leq \varphi(s)-\varphi(t) \quad \text{for all}\ 0\leq s\leq t,
 \label{phidierve}\\
 &\frac1p \int^t_s \liminf_{i \to \infty} |\Xi'_{\gtau_i}|^p (r) \,{\dd}r
 \leq \varphi(s)-\varphi(t)<\infty \quad \text{for all}\ 0\leq s\leq t. \nonumber
\end{align}
The latter inequality implies that
$\liminf_{i \to \infty}|\Xi'_{\gtau_i}|(t)<\infty$ for $\mathscr{L}^1$-a.e.\ $t \in (0,\infty)$.
Since $\varphi$ is non-increasing, $\varphi'$ exists $\mathscr{L}^1$-a.e.\
and we find from \eqref{phidierve} that
\[
 \varphi'(t) \leq -\frac1p |\xi'_+|^p(t) -\frac1q |\partial^-\phi|^q \big( \xi(t) \big) \quad
 \text{for $\mathscr{L}^1$-a.e.}\ t \in (0,\infty).
\]
As $|\partial^-\phi|$ is a weak upper gradient by assumption,
it suffices to show that $\varphi$ is $\mathscr{L}^1$-a.e.\ equal to $\phi \circ \xi$.

Observe from Lemma~\ref{slopfirstexsts} (as in the proof of Theorem~\ref{existsassupab}) that
\[
 \liminf_{i \to \infty} |\partial\phi|^q \big( \overline{\Xi}_{\gtau_i}(t) \big)
 \leq \liminf_{i \to \infty} |\Xi'_{\gtau_i}|^p(t) <\infty \quad \text{for $\mathscr{L}^1$-a.e.}\ t \in (0,\infty).
\]
Thus, for $\mathscr{L}^1$-a.e.\ $t \in (0,\infty)$, the assumed continuity condition yields
$\varphi(t) =\lim_{i \to \infty} \phi(\overline{\Xi}_{\gtau_i}(t)) =\phi(\xi(t))$.
This completes the proof.
\end{proof}

Now, we present an existence result to Problem~\ref{firsslopque}.
We say that a function $\phi:X \lra \mathbb{R}$ is of \emph{lower $p$-growth}
if there are constants $C,D \geq 0$ such that
$\phi(x) \ge -C-Dd^p(\star,x)$ holds for all $x \in X$.
For instance, if $\inf_X \phi>-\infty$, then $\phi$ is of lower $p$-growth
with $C=-\min\{\inf_X \phi,0\}$.

\begin{corollary}\label{generalforwardexistssolu}
Let $(X,d)$ be a forward boundedly compact forward metric space, $p \in (1,\infty)$,
and $\phi:X \lra \mathbb{R}$ be a continuous function of lower $p$-growth.
If $|\partial^-\phi|$ is a weak upper gradient for $-\phi$,
then, for any $x_0\in X$,
there exists a $p$-curve $\xi:[0,\infty) \lra X$ of maximal slope for $\phi$
with respect to $|\partial^-\phi|$ with $\xi(0)=x_0$.
If additionally $|\partial^-\phi|$ is a strong upper gradient,
then $\xi$ satisfies the energy identity \eqref{eq:energy}.
\end{corollary}

\begin{proof}
Let $\sigma=\mathcal{T}_+$.
Observe from the triangle inequality and Lemma~\ref{basicpestimate} that
\[
 \Phi(\tau,x;y) \geq -C -Dd^p(\star,y) +\frac{d^p(x,y)}{p\tau^{p-1}}
 \geq -C-D \big\{ \mathfrak{C}d^p(\star,x) +(1+\epsilon)d^p(x,y) \big\} +\frac{d^p(x,y)}{p\tau^{p-1}}.
\]
Hence, for sufficiently small $\tau>0$, we have $\inf_{y \in X} \Phi(\tau,x;y)>-\infty$
and Assumption~\ref{continudef}\eqref{ass-b} holds.
Moreover, Assumption~\ref{continudef}\eqref{ass-a}, \eqref{ass-c} hold by hypotheses.
Therefore, given any sequence of time steps $(\gtau_\alpha)_{\alpha \ge 1}$
with initial data $(\Xi^0_{\gtau_\alpha})_{\alpha \ge 1}$
such that $\|{\gtau}_\alpha\| \to 0$ and $\Xi_{\gtau_\alpha}^0 \to x_0$,
Corollary~\ref{compacttexistsencesolu} provides
a limit curve $\xi \in \GMM_p(\phi;x_0)$.
Then the claims follow from Theorems~\ref{gmmslopI} and \ref{existsassupab}.
\end{proof}

We remark that the lower $p$-growth was assumed merely for ensuring
that Assumption~\ref{continudef}\eqref{ass-b} holds.
Another condition implying Assumption~\ref{continudef}\eqref{ass-b} is,
e.g., that every sublevel set of $\phi$ is compact (see Proposition~\ref{simpassump}).

In view of Example~\ref{Nonfinslerexamp},
we also give an existence result for gradient curves in the Finsler case.

\begin{corollary}\label{bascirgeulfinsler}
Let $(M,F)$ be a forward complete Finsler manifold and $p \in (1,\infty)$.
For any $\phi \in C^1(M)$ and $x_0 \in M$,
there exists a $C^1$-curve $\xi:[0,T) \lra M$ solving the gradient flow equation
\[
 \mathfrak{j}_p \big( \xi'(t) \big) =\nabla(-\phi) \big( \xi(t) \big), \qquad \xi(0)=x_0,
\]
where $\lim_{t \to T} d_F(x_0,\xi(t))=\infty$ if $T<\infty$.
\end{corollary}

\begin{proof}
Let $\sigma=\mathcal{T}_+$, and note that the forward metric space $(M,d_F)$
is forward boundedly compact by the Hopf--Rinow theorem.
Since $\phi\in C^1(M)$,
it follows from Example~\ref{finsbackgrad} and Remark~\ref{lowerbackreamrk} that
$|\partial^-\phi|=|\partial\phi|=F(\nabla(-\phi))$ is a strong upper gradient for $-\phi$.
Now, if $\phi$ is bounded below (or of lower $p$-growth),
then Corollary~\ref{generalforwardexistssolu}
furnishes a $p$-curve $\xi:[0,\infty) \lra M$ of maximal slope for $\phi$
with respect to $|\partial^-\phi|$ with $\xi(0)=x_0$,
and we conclude the proof by the argument in Example~\ref{Nonfinslerexamp}.

When $\inf_X \phi=-\infty$,
we replace $\phi$ with $\phi_r:=\max\{\phi,\inf_{B^+_{x_0}(r)} \phi\}$ for large $r>0$
and construct a gradient curve $\xi$ within $B^+_{x_0}(r)$.
If $\xi$ does not reach $\partial B^+_{x_0}(r)$, then $\xi$ is defined on $[0,\infty)$.
If $\xi$ reaches $\partial B^+_{x_0}(r)$ at some $T_1 \in (0,\infty)$,
then we continue the construction for $\phi_{2r}$ from $x_1:=\xi(T_1)$.
Iterating this procedure, since $\nabla(-\phi_r)=\nabla(-\phi)$ in $B^+_{x_0}(r)$,
we eventually obtain a $C^1$-curve $\xi:[0,T) \lra M$
satisfying $\mathfrak{j}_p (\xi'(t)) =\nabla(-\phi)(\xi(t))$
and $\lim_{t \to T} d_F(x_0,\xi(t))=\infty$ if $T<\infty$.
\end{proof}

\begin{remark}\label{regularifinslerm}
%By \eqref{finserpflow}, for $\phi \in C^1(M)$, $\xi$ is at most $C^1$ in general.
For $\phi \in C^l(M)$ with $l \ge 2$,
$\nabla(-\phi)$ is only continuous at its zeros while $C^{l-1}$ at other points
(see \cite{GS,Obook,OS}).
Hence, in order to get a higher regularity of $\xi$,
we need to restrict ourselves to an interval in which $\nabla(-\phi)(\xi(t)) \neq 0$.
See Corollary~\ref{regularoffINSLERCASE} below for instance.
\end{remark}

\subsection{{Doubly nonlinear evolution equations}}\label{funkspacesgrad}
%%%%%%%%%%%%%%%%%%%%%%%%%

This subsection is devoted to gradient flows in infinite dimensional Funk and
{Randers-like spaces}.
Let $(\mathscr{H},\langle\cdot,\cdot\rangle)$ be a Hilbert space
and set $\|\cdot\|:=\sqrt{\langle\cdot,\cdot\rangle}$.
On the unit ball $\mathbb{B}:=\{x\in \mathscr{H} \,|\,\|x\|<1\}$,
we define an asymmetric distance function $d:\mathbb{B} \times \mathbb{B} \lra [0,\infty)$
in the same way as \eqref{distFunk}.
Similarly to \eqref{Funckmeatirc}, $d$ is associated with the (infinite dimensional) Finsler structure
\[
 F(x,v)
 :=\lim_{\varepsilon\to 0^+} \frac{d(x,x+\varepsilon v)}{\varepsilon}
 =\frac{\sqrt{(1-\|x\|^2)\|v\|^2+\langle x,v \rangle^2} +\langle x,v \rangle}{1-\|x\|^2},
 \quad v \in T_x\mathbb{B} \cong \mathscr{H}.
\]
We call $(\mathbb{B},d)$ a \emph{generalized Funk space} and observe the following
(see \cite[Remark 1.1.3]{AGS} for \eqref{funpor3}).

\begin{proposition}\label{basicproperofFunk}
Let $(\mathbb{B},d)$ be a generalized Funk space.
\begin{enumerate}[{\rm (i)}]
\item\label{funpor1}
$(\mathbb{B},\mathbf{0},d)$ is a forward complete pointed forward $\Theta$-metric space
with $\Theta(r) =2{\ee}^r-1$.

\item\label{funpor2}
Both the forward and backward topologies coincide with the original topology of
$\mathbb{B}\subset \mathscr{H}$.

\item\label{funpor3}
A curve $\gamma$ belongs to ${\FAC}^p((a,b);\mathbb{B})$ if and only if
it is differentiable at $\mathscr{L}^1$-a.e.\ $t\in (a,b)$ with the derivative $\gamma'$
satisfying $F(\gamma,\gamma') \in L^p(a,b)$ and
\[
 \gamma(t)-\gamma(s)=\int^t_s \gamma'(r) \,{\dd}r \quad \text{for any }a<s\leq t<b.
\]
Moreover, for $\gamma \in {\FAC}^p((a,b);\mathbb{B})$, we have
$|\gamma'_+|(t)=F(\gamma(t),\gamma'(t))$ for $\mathscr{L}^1$-a.e.\ $t\in (a,b)$.
\end{enumerate}
\end{proposition}

Let $\mathscr{H}^*$ denote the dual space of $\mathscr{H}$.
For $\zeta \in T^*_x\mathbb{B} \cong \mathscr{H}^*$ ($\cong \mathscr{H}$),
the dual norm of $F$ is defined as
\[
 F^*(x,\zeta)
 :=\sup_{v \in T_x\mathbb{B} \setminus \{\mathbf{0}\}} \frac{\langle \zeta,v \rangle}{F(x,v)}.
\]
For $p\in (1,\infty)$, $q=p/(p-1)$ and $v \in T_x\mathbb{B}$,
define $\mathfrak{J}_p(x,v) \subset T_x^*\mathbb{B}$ by
\[
 \mathfrak{J}_p(x,v) :=\{ \zeta \in T_x^*\mathbb{B} \,|\,
 \langle \zeta,v \rangle =F^p(x,v) =F^*(x,\zeta)^q =F(x,v)F^*(x,\zeta) \}.
\]
Note that $\mathfrak{J}_p(x,y)$ is at most a singleton by the differentiability of $\|\cdot\|$
(see Proposition~\ref{bascifunkproperty}).

Given $\phi:\mathbb{B} \lra (-\infty,\infty]$ and $x \in \mathbb{B}$, we set
\begin{align*}
 \partial\phi(x)
 &:= \bigg\{ \zeta \in T_x^*\mathbb{B} \,\bigg|\,
 \liminf_{v \to \mathbf{0}} \frac{\phi(x+v) -\phi(x) -\langle \zeta,v \rangle}{d(x,x+v)} \ge 0 \bigg\}
 \,\ \text{for}\ x \in \mathfrak{D}(\phi), \\
 \partial^{\circ} \phi(x)
 &:=\{ \zeta \in \partial\phi(x) \,|\, F^*(x,-\zeta) \leq F^*(x,-\eta) \ \text{for all}\ \eta\in \partial\phi(x) \}
 \,\ \text{for}\ x \in \mathfrak{D}(\partial\phi), \\
 F^*\big( {-}\partial^{\circ} \phi(x) \big)
 &:= \begin{cases}
 \inf_{\zeta \in \partial\phi(x)} F^*(x,-\zeta) & \text{ if } x \in \mathfrak{D}(\partial\phi), \\
 +\infty & \text{ if } x \notin \mathfrak{D}(\partial\phi),
 \end{cases}
\end{align*}
where $\partial\phi(x)$ is the \emph{Fr\'echet subdifferential} of $\phi$ at $x$
and $\mathfrak{D}(\partial\phi):=\{x\in \mathfrak{D}(\phi) \,|\, \partial \phi(x) \neq \emptyset \}$.
The function $x \longmapsto F^*(-\partial^\circ\phi(x))$ is a weak upper gradient for $-\phi$
(see Proposition~\ref{phigradproe}).

The main result of this subsection reads as follows
(cf.\ \cite[Proposition~1.4.1, Theorem~2.3.7]{AGS}).

\begin{theorem}\label{basicreal} %[AGS, Cor 1.4.5, Prop 1.4.1, Thm 2.3.7]
Let $(\mathbb{B},d)$ be a generalized Funk space,
and let $\phi:\mathbb{B} \longmapsto (-\infty,\infty]$ admit a decomposition $\phi=\phi_1+\phi_2$
where $\phi_1$ is a proper, lower semicontinuous, convex function
and $\phi_2$ is of class $C^1$ in the Fr\'echet sense.
\begin{enumerate}[{\rm (i)}]
\item\label{envo1}
The local slope $|\partial\phi|$ is a strong upper gradient for $-\phi$.
Moreover, for every $x \in \mathbb{B}$, we have
$|\partial\phi|(x) =F^*(-\partial^\circ\phi(x)) =|\partial^-\phi|(x)$.

\item\label{envo2}
If $\xi:(a,b) \lra \mathbb{B}$ is a $p$-curve of maximal slope for $\phi$ with respect to $|\partial\phi|$,
then it satisfies the following \emph{doubly nonlinear differential equation:}
\begin{equation}\label{basicubangr}
 \mathfrak{J}_p \big( \xi(t),\xi'(t) \big)
 =-\partial^\circ \phi\big( \xi(t) \big) \neq \emptyset \quad \text{for $\mathscr{L}^1$-a.e.\ $t\in (a,b)$.}
\end{equation}

\item\label{envo3}
If Assumption~$\ref{continudef}$\eqref{ass-b}, \eqref{ass-c} hold with $\sigma=\mathcal{T}_+$,
then, for every $x_0 \in \mathfrak{D}(\phi)$, $\GMM_p(\phi;x_0)$ is nonempty
and its element $\xi$ satisfies \eqref{basicubangr} and the energy identity
\[
 \int^T_0 |\xi'_+(t)|^p \,{\dd}t +\phi\big( \xi(T) \big) =\phi(x_0) \quad \text{for all}\,\ T>0.
\]
\end{enumerate}
\end{theorem}

\begin{proof}
\eqref{envo1}
By Proposition~\ref{phigradproe},
the same argument as in the proof of \cite[Corollary~1.4.5]{AGS} yields that
$\partial\phi=\partial\phi_1+{\dd}\phi_2$ satisfies \eqref{basiupperproper1} and \eqref{strongweakclosed},
and $|\partial\phi|$ is a strong upper gradient for $-\phi$ with $|\partial\phi|(x) =F^*(-\partial^\circ\phi(x))$.
Moreover, $|\partial\phi|=|\partial^-\phi|$ follows from (the proof of) \cite[Lemma~2.3.6]{AGS}
(with $\sigma=\mathcal{T}_+$).

\eqref{envo2}
Since $\xi'(t)$ exists $\mathscr{L}^1$-a.e.\ due to Proposition~\ref{basicproperofFunk}\eqref{funpor3},
 Proposition~\ref{curaboslpssss} combined with \eqref{envo1} implies
\[
 \frac{\dd}{{\dd}t} \phi \big( \xi(t) \big)
 =-\frac1p |\xi'_+|^p(t) -\frac1q F^*\big( {-}\partial^\circ\phi(\xi(t)) \big)^q
 \quad \text{for $\mathscr{L}^1$-a.e.}\ t \in (a,b).
\]
On the other hand, for $\mathscr{L}^1$-a.e.\ $t \in (a,b)$ and any $\zeta \in \partial^\circ\phi(\xi(t))$, we have
\[
 \frac{\dd}{{\dd}t} \phi \big( \xi(t) \big)
 =\lim_{\varepsilon \to 0^+} \frac{\phi(\xi(t+\varepsilon))-\phi(\xi(t))}{\varepsilon}
 \geq \liminf_{\varepsilon \to 0^+} \frac{\langle \zeta,\xi(t+\varepsilon)-\xi(t) \rangle}{\varepsilon}
 =-\langle -\zeta,\xi'(t) \rangle.
\]
The Young inequality then furnishes $-\zeta \in \mathfrak{J}_p(\xi(t),\xi'(t))$,
and Proposition~\ref{bascifunkproperty}\eqref{fungrad3} yields \eqref{basicubangr}.
%The converse implication follows from a similar but easier argument.

\eqref{envo3}
This follows from Corollary~\ref{compacttexistsencesolu}, Theorem~\ref{existsassupab},
and \eqref{envo1}, \eqref{envo2} above.
\end{proof}

We remark that the above argument also applies to the unit balls in general Banach spaces $\mathscr{B}$
as well as Minkowski normed spaces of infinite dimension.
For example, let $(\mathscr{B},\|\cdot\|)$ be a reflexive Banach space
and $(\mathscr{B}^*,\|\cdot\|_*)$ be its dual space.
Given $\omega \in \mathscr{B}^*$ with $\|\omega\|_*<1$,
we can define an asymmetric metric $d_\omega$ on $\mathscr{B}$ by
$d_\omega(x,y) :=\|y-x\| +\omega(y-x)$
(cf.\ Example~\ref{ex:Randers}).
Then the Minkowski normed space $(\mathscr{B},d_\omega)$
is a $[({1+\|\omega\|_*})/({1-\|\omega\|_*})]$-metric space and
Proposition~\ref{basicproperofFunk}\eqref{funpor2}, \eqref{funpor3} hold.
Moreover, since
\[
 F(x,v) =\lim_{\varepsilon \to 0^+} \frac{d_\omega(x,x+\varepsilon v)}{\varepsilon}
 =\|v\| +\omega(v),
\]
both $F^*(x,\zeta)$ and $\mathfrak{J}_p(x,v)$ are independent of $x$.
Theorem~\ref{basicreal} remains valid by replacing \eqref{basicubangr}
with the \emph{doubly nonlinear differential inclusion}
$\mathfrak{J}_p(\xi'(t))\supset -\partial^\circ \phi(\xi(t))$
(if $\|\cdot\|$ is differentiable, then $\mathfrak{J}_p(v)$ is at most a singleton
and Theorem~\ref{basicreal} holds as it is).

{
A generalization of the theory of doubly nonlinear evolution equations (DNE)
to asymmetric Finsler-like metrics on Banach spaces was investigated in \cite[\S 8]{RMS},
under the finite reversibility and the lower boundedness of $\phi_1$
(see \cite[(5.1c), Definition~5.10]{RMS}).
We remark that the reversibility of generalized Funk spaces is infinite
by Proposition~\ref{basicproperofFunk}\eqref{funpor1}.
}

\section{$(p,\lambda)$-convexity}\label{plambdaconvex}%%%%%%%%%%
%%%%%%%%%%%%%%%%%%%%%%

In this section, we study the behavior of curves of maximal slope
under a certain convexity assumption on $\Phi$,
where $\Phi$ was defined in Definition~\ref{mosydefam}
as a combination of the target function $\phi$ and the distance function.
Let $(X,d)$ be a forward complete forward metric space,
$\phi:X \lra (-\infty,\infty]$ be a proper function, and $p \in (1,\infty)$
throughout this section.

\subsection{Convexity assumption}%%%%%%%%%%%%
%%%%%%%%%%%%%%%

A curve $\gamma:[0,1] \lra X$ is called a \emph{minimal geodesic}
if $d(\gamma(s),\gamma(t))=(t-s)d(\gamma(0),\gamma(1))$ holds for any $0 \leq s<t \leq 1$.

\begin{definition}[$(p,\lambda)$-convexity]\label{df:pl-conv}
For $\lambda\in \mathbb{R}$,
we say that $\phi$ is \emph{$(p,\lambda)$-convex}
(resp.\ \emph{$(p,\lambda)$-geodesically convex}) if, for any $x_0,x_1\in X$,
there is a curve (resp.\ a minimal geodesic) $\gamma: [0,1] \lra X$ from $x_0$ to $x_1$
such that
\[
 \phi\big( \gamma(t) \big)
 \leq (1-t)\phi(x_0) +t\phi(x_1) -\frac{\lambda}{p} t(1-t^{p-1}) d^p(x_0,x_1)
 \quad \text{for all}\,\ t \in [0,1].
\]
In the particular case of $p=2$, $\phi$ is also called a \emph{$\lambda$-convex}
(resp.\ \emph{$\lambda$-geodesically convex}) function.
\end{definition}

{
\begin{remark}\label{rm:p-conv}
We remark that our definition of $(p,\lambda)$-convexity slightly differs from
the \emph{$(\lambda,p)$-convexity} in \cite[Definition~2.5]{RSS} (discussed on symmetric metric spaces),
which requires
\[
 \phi\big( \gamma(t) \big)
 \leq (1-t)\phi(x_0) +t\phi(x_1) -\frac{\lambda}{p} t(1-t) d^p(x_0,x_1)
 \quad \text{for all}\,\ t \in [0,1]
\]
(see also \cite[Remark~2.4.7]{AGS}).
Since $t \in [0,1]$, this is stronger (resp.\ weaker) than our $(p,\lambda)$-convexity
for $\lambda>0$ and $p \in (1,2)$ or $\lambda<0$ and $p \in (2,\infty)$
(resp.\ for $\lambda>0$ and $p \in (2,\infty)$ or $\lambda<0$ and $p \in (1,2)$).
Our $(p,\lambda)$-convexity is well-behaved in subdivisions in the sense that combining
\[ \phi\bigg( \gamma\bigg( \frac{1}{4} \bigg) \bigg)
 \le \frac{1}{2} \phi(x_0) +\frac{1}{2} \phi\bigg( \gamma\bigg( \frac{1}{2} \bigg) \bigg)
 -\frac{\lambda}{p} \frac{1}{2} \bigg( 1-\frac{1}{2^{p-1}} \bigg) \bigg( \frac{d(x_0,x_1)}{2} \bigg)^p \]
and
\[ \phi\bigg( \gamma\bigg( \frac{1}{2} \bigg) \bigg)
 \le \frac{1}{2} \phi(x_0) +\frac{1}{2} \phi(x_1)
 -\frac{\lambda}{p} \frac{1}{2} \bigg( 1-\frac{1}{2^{p-1}} \bigg) d^p(x_0,x_1) \]
implies
\[ \phi\bigg( \gamma\bigg( \frac{1}{4} \bigg) \bigg)
 \le \frac{3}{4} \phi(x_0) +\frac{1}{4}  {\phi(x_1)}
 -\frac{\lambda}{p} \frac{1}{4} \bigg( 1-\frac{1}{4^{p-1}} \bigg) d^p(x_0,x_1).\]
Moreover, our usage of $\Phi$ below in the spirit of \cite{AGS} seems also advantageous.
See Remarks~\ref{rm:lp-pl_1}, \ref{rm:lp-pl_2} for further discussions.
\end{remark}
}

Now we introduce an important convexity assumption
(in the spirit of \cite[Assumption~2.4.5]{AGS}).

\begin{assumption}\label{firstassup} %[AGS, Assumption 2.4.5]
Let $\lambda \in \mathbb{R}$ and put $\lambda_-:=-\min\{\lambda,0\}$.
We assume that, for any $x_0,x_1 \in \mathfrak{D}(\phi)$,
there exists a curve $\gamma:[0,1] \lra X$ from $x_0$ to $x_1$ such that
\begin{equation}\label{eq:Phi}
 \Phi\big( \tau,x_0;\gamma(t) \big)
 \leq (1-t)\Phi(\tau,x_0;x_0) +t\Phi(\tau,x_0;x_1)
 -\frac1p \bigg( \lambda +\frac{1}{\tau^{p-1}} \bigg) t(1-t^{p-1}) d^p(x_0,x_1)
\end{equation}
for all $\tau\in (0,\lambda_-^{-1/(p-1)})$ and $t\in [0,1]$,
where we set $\lambda_-^{-1/(p-1)}:=\infty$ if $\lambda_-=0$.
\end{assumption}

In other words, $\Phi(\tau,x_0;\cdot)$ is $(p,\lambda +\tau^{1-p})$-convex
for a common curve $\gamma$ for all $\tau \in (0,\lambda_-^{-1/(p-1)})$.
We remark that $\gamma$ is always emanating from $x_0$
(cf.\ \cite[Assumption~4.0.1]{AGS}, which is a stronger convexity condition
requiring the convexity of $\Phi(\tau,x_0;\cdot)$ between any pair of points).

\begin{example}\label{pgeodesicconve2}
\begin{enumerate}[(a)]
\item \label{lambda-a}
It is readily seen that a $(p,\lambda)$-geodesically convex function
$\phi:X \lra (-\infty,\infty]$ satisfies Assumption~\ref{firstassup}
(regardless of the convexity of the distance function).
Indeed, for a minimal geodesic $\gamma:[0,1] \lra X$ from $x_0$ to $x_1$ along which
the $(p,\lambda)$-convexity holds, we have
\begin{align*}
 \Phi\big( \tau,x_0;\gamma(t) \big)
 &= \phi\big( \gamma(t) \big) +\frac{d^p(x_0,\gamma(t))}{p\tau^{p-1}} \\
 &\leq (1-t)\phi(x_0) +t\phi(x_1) -\frac{\lambda}{p} t(1-t^{p-1}) d^p(x_0,x_1)
 +\frac{t^p}{p\tau^{p-1}}d^p(x_0,x_1) \\
 &= (1-t)\Phi(\tau,x_0;x_0) +t\Phi(\tau,x_0;x_1)
 -\frac1p \bigg( \lambda +\frac{1}{\tau^{p-1}} \bigg) t(1-t^{p-1}) d^p(x_0,x_1).
\end{align*}

\item \label{lambda-b}
The convexity of the distance function is intimately related to how the space is curved.
For a complete, simply-connected Riemannian manifold $(M,g)$,
$d_g^2(x,\cdot)$ is $2$-geodesically convex for every $x \in M$
if and only if the sectional curvature is nonpositive.
This is a fact at the origin of the fruitful theory of CAT$(0)$-spaces
(see, e.g., \cite{Jost_book}).
In the Finsler case, only the flag curvature is not sufficient to control the convexity
of the distance function and we need some additional conditions.
For example, for forward complete, simply-connected Finsler manifold $(M,F)$ of Berwald type
with nonpositive flag curvature, $d_F^2(x,\cdot)$ is $0$-geodesically convex for every $x \in M$.
We refer to \cite[\S 15.1]{Sh1} and \cite[\S 5]{Ouni} for this and more general results.
%
%\item \label{lambda-c}
%{\color{blue}some examples...}
\end{enumerate}
\end{example}

We collect some immediate consequences of Assumption~$\ref{firstassup}$ in the next lemma.

\begin{lemma}\label{convedefremark}
Suppose that $\phi$ satisfies Assumption~$\ref{firstassup}$ for some $(p,\lambda)$.
Then we have the following.
\begin{enumerate}[{\rm (i)}]
\item \label{pcon-1}
$\phi$ satisfies Assumption~$\ref{firstassup}$ for all $(p,\lambda')$ with $\lambda'<\lambda$.

\item \label{pcon-2}
For any $\tau\in (0,\lambda_-^{-1/(p-1)})$ and $t\in (0,1]$, we have
\begin{equation}\label{constaweak3.24}
 \frac{\phi(\gamma(t))-\phi(x_0)}{t}
 \leq \phi(x_1) -\phi(x_0) +\frac{t^{p-1}-\lambda\tau^{p-1} (1-t^{p-1})}{p\tau^{p-1}} d^p(x_0,x_1),
\end{equation}
where $\gamma$ is the curve satisfying \eqref{eq:Phi}.
In particular,
\begin{equation}\label{disctancecontrl}
 d\big( x_0,\gamma(t) \big) \leq td(x_0,x_1)
 \quad \text{for all}\,\ t \in [0,1].
\end{equation}

\item \label{pcon-3}
If $\lambda\geq 0$, then we have
\begin{equation}\label{longcopexlambda}
 \frac{\phi(\gamma(t))-\phi(x_0)}{t}
 \leq \phi(x_1) -\phi(x_0) -\frac{\lambda}{p} (1-t^{p-1}) d^p(x_0,x_1)
 \quad \text{for all}\,\ t \in (0,1].
\end{equation}
In particular, $\phi$ is $(p,\lambda)$-convex.
\end{enumerate}
\end{lemma}

\begin{proof}
\eqref{pcon-1} is trivial by definition.
In \eqref{pcon-2}, \eqref{constaweak3.24} follows from
$\phi(\gamma(t)) \le \Phi(\tau,x_0;\gamma(t))$ and \eqref{eq:Phi} as
\begin{align}
 \Phi \big( \tau,x_0;\gamma(t) \big)
 &\leq (1-t)\phi(x_0) +t\phi(x_1) +t\frac{d^p(x_0,x_1)}{p\tau^{p-1}}
 -\frac1p \bigg( \lambda +\frac{1}{\tau^{p-1}} \bigg) t(1-t^{p-1}) d^p(x_0,x_1)
 \nonumber\\
 &= (1-t)\phi(x_0) +t\phi(x_1) +t\frac{t^{p-1} -\lambda \tau^{p-1} (1-t^{p-1})}{p\tau^{p-1}} d^p(x_0,x_1).
 \label{3.22remark30}
\end{align}
We obtain \eqref{disctancecontrl} by letting $\tau \to 0$ in the above inequality.
For \eqref{pcon-3}, the assumption $\lambda\geq 0$ implies $\lambda_-=0$ and hence
\eqref{constaweak3.24} holds for all $\tau \in (0,\infty)$,
which furnishes \eqref{longcopexlambda} as  $\tau \to \infty$.
\end{proof}

Next we give an estimate of $\tau_*(\phi)$ (see Definition~\ref{myadef})
under Assumption~\ref{firstassup} (cf.\ \cite[Lemma~2.4.8]{AGS}).

\begin{lemma}\label{cocervityconvexfun} %[AGS, Lemma 2.4.8]
Suppose that Assumption~$\ref{firstassup}$ holds for some $(p,\lambda)$
and there are $x_* \in \mathfrak{D}(\phi)$ and $r_* >0$ such that
\begin{equation}\label{inffunctionobund}
 m_* :=\inf \{ \phi(x) \,|\, d(x_*,x) \le r_* \} >-\infty.
\end{equation}
\begin{enumerate}[{\rm (i)}]
\item \label{m*-1}
We have $\tau_*(\phi)\geq \lambda^{-1/(p-1)}_-$.
In particular, $\tau_*(\phi)=\infty$ if $\lambda\geq 0$.

\item \label{m*-2}
If $\lambda >0$, then $\phi$ is bounded from below.
Moreover, if $\phi$ is $\mathcal{T}_+$-lower semicontinuous,
then it has a unique minimizer $\bar{x} \in X$.
\end{enumerate}
\end{lemma}

\begin{proof}
\eqref{m*-1}
Given $\tau=(\lambda_- +\varepsilon)^{-1/(p-1)} \in (0,\lambda^{-1/(p-1)}_-)$, we claim that
\begin{equation}\label{keyparrl}
 \inf\{ \Phi(\tau,x_*;y) \,|\, d(x_*,y)>r_* \} >-\infty.
\end{equation}
Combining this with \eqref{inffunctionobund} implies
$\tau_*(\phi)\geq (\lambda_-+\varepsilon)^{-1/(p-1)}$,
and then \eqref{m*-1} follows by letting $\varepsilon \to 0$.

In order to show \eqref{keyparrl},
take an arbitrary point $y \in \mathfrak{D}(\phi)$ with $d(x_*,y)>r_*$,
let $\gamma:[0,1] \lra X$ be a curve from $x_0=x_*$ to $x_1=y$ satisfying \eqref{eq:Phi},
and set $y_*:=\gamma(r_*/d(x_*,y))$.
Then \eqref{disctancecontrl} implies $d(x_*,y_*) \leq r_*$ and hence $\phi(y_*) \geq m_*$.
We deduce from \eqref{constaweak3.24} that
\begin{align*}
 &\frac{d(x_*,y)}{r_*}\big( \phi(y_*)-\phi(x_*) \big) \\
 &\leq \phi(y)-\phi(x_*) +\frac{1}{p\tau^{p-1}} \bigg\{
 \bigg( \frac{r_*}{d(x_*,y)} \bigg)^{p-1} -\lambda \tau^{p-1}
 \bigg( 1-\bigg(\frac{r_*}{d(x_*,y)} \bigg)^{p-1} \bigg) \bigg\} d^p(x_*,y).
\end{align*}
This together with $\phi(y_*) \geq m_*$ yields
\begin{align}
 &\phi(x_*) +c_* d(x_*,y)
 \leq \phi(y) -\frac{\lambda}{p} d^p(x_*,y)
 \leq \phi(y) +\frac{\lambda_-}{p} d^p(x_*,y),\label{p1controllinequality} \\
 &\text{where}\,\ c_* := \frac{m_*-\phi(x_*)-(\tau^{1-p}+\lambda)r^p_*/p}{r_*}.\notag
\end{align}
Combining this with the Young inequality
\[
 -c_* d(x_*,y) \leq |c_*| d(x_*,y)
 \leq \frac{\varepsilon}{p} d^p(x_*,y) +\frac{\varepsilon^{-q/p}}{q}|c_*|^q
\]
and recalling $\tau=(\lambda_- +\varepsilon)^{-1/(p-1)}$, we find
\[
 \Phi(\tau,x_*;y)
 = \phi(y) +\frac{\lambda_- +\varepsilon}{p} d^p(x_*,y)
 \geq \phi(x_*) +c_*d(x_*,y) +\frac{\varepsilon}{p} d^p(x_*,y)
 \geq \phi(x_*) -\frac{\varepsilon^{-q/p}}{q} |c_*|^q.
\]
This implies \eqref{keyparrl} and completes the proof.

\eqref{m*-2}
If $\lambda>0$, then for any  $\varepsilon>\lambda$, the above argument with $\tau=(\varepsilon -\lambda)^{-1/(p-1)}$ shows
\[  \phi(y) +\frac{\varepsilon -\lambda}{p} d^p(x_*,y)
 \geq \phi(x_*) +c_*d(x_*,y) +\frac{\varepsilon}{p} d^p(x_*,y)
 \geq \phi(x_*) -\frac{\varepsilon^{-q/p}}{q} (-c_*)^q. \]
Letting $\varepsilon \to \lambda$, we deduce from the above inequality and \eqref{inffunctionobund} that
\[ \inf_X \phi \geq \min \Bigg\{ m_*,
 \phi(x_*) -\frac{\lambda^{-q/p}}{q} \frac{(\phi(x_*) -m_* +(\lambda r_*^p)/p)^q}{r_*^q} \Bigg\}. \]

When $\phi$ is lower semicontinuous,
take any minimizing sequence $(x_i)_{i \ge 1}$,
i.e., $\phi(x_i) \to \inf_X \phi$.
For any $0<i \leq j$, let $\gamma_{i,j}$ be a curve from $x_i$ to $x_j$ satisfying \eqref{eq:Phi}.
It follows from \eqref{longcopexlambda} (with $t=1/2$) that
\[
 2\bigg\{ \phi\bigg( \gamma_{i,j} \bigg( \frac{1}{2} \bigg) \bigg) -\phi(x_i) \bigg\}
 \leq \phi(x_j)-\phi(x_i) -\frac{\lambda}{p} (1-2^{1-p}) d^p(x_i,x_j),
\]
which furnishes
\[
\frac{\lambda}{p} (1-2^{1-p}) d^p(x_i,x_j)
 \leq \phi(x_i) +\phi(x_j) -2\phi\bigg( \gamma_{i,j} \bigg( \frac{1}{2} \bigg) \bigg)
 \leq \phi(x_i) +\phi(x_j) -2\inf_X \phi \to 0
\]
as $i,j \to \infty$.
Thus, $(x_i)_{i \ge 1}$ is a forward Cauchy sequence and converges to some $\bar{x} \in X$,
which is a minimizer of $\phi$ due to the  lower semicontinuity of $\phi$.
The uniqueness also follows from the above argument, as any minimizing sequence is convergent.
\end{proof}

Two remarks on the relationship between Assumption~\ref{continudef} and
\eqref{inffunctionobund} are in order.

\begin{remark}\label{assmupabcimpb}
\begin{enumerate}[(a)]
\item \label{m*-a}
Under Assumption~\ref{firstassup},
Assumption~\ref{continudef}\eqref{ass-b} is equivalent to \eqref{inffunctionobund}.
On the one hand, it follows from Lemma~\ref{cocervityconvexfun}\eqref{m*-1} that
\eqref{inffunctionobund} implies Assumption~\ref{continudef}\eqref{ass-b}.
On the other hand, if Assumption~\ref{continudef}\eqref{ass-b}
($\Phi_{\tau_*}(x_*)>-\infty$) holds, then we have,
for any $r_*>0$ and $x \in \overline{B^+_{x_*}(r_*)}$,
\[
 \phi(x) \geq \Phi_{\tau_*}(x_*) -\frac{d^p(x_*,x)}{p\tau^{p-1}_*}
 \geq \Phi_{\tau_*}(x_*) -\frac{r_*^p}{p\tau^{p-1}_*}.
\]
Hence, \eqref{inffunctionobund} holds.

\item \label{m*-b}
Assumption~\ref{continudef}\eqref{ass-a}, \eqref{ass-c} imply \eqref{inffunctionobund}.
Assume on the contrary that there is a sequence $(x_i)_{i \ge 1}$ in $\overline{B^+_{x_*}(r_*)}$
with $\phi(x_i) \to -\infty$.
By Assumption~\ref{continudef}\eqref{ass-c},
we can find a subsequence of $(x_i)_{i \ge 1}$
which is $\sigma$-converging to some point $x_{\infty}$.
Then Assumption~\ref{continudef}\eqref{ass-a} furnishes
$\phi(x_{\infty}) \leq \lim_{i \to \infty} \phi(x_i) =-\infty$,
which contradicts $\phi(x_{\infty})>-\infty$.
\end{enumerate}
\end{remark}

\subsection{Existence of curves of maximal slope}%%%%%%%%%%%%%%
%%%%%%%%%%%%%%%%%%%%%%

Under the convexity as in Assumption~\ref{firstassup},
we can show a global formula for the local slope $|\partial\phi|$ (recall Definition~\ref{wekafordef}),
which plays an important role in proving the existence of curves of maximal slope
(cf.\ \cite[Theorem~2.4.9]{AGS}, {\cite[Lemma~5.3]{RMS}, \cite[Proposition~2.7]{RSS}}).

\begin{theorem}[Global formula of $|\partial\phi|$]\label{thremforwposrt} %[AGS, Theorem 2.4.9]
If Assumption~$\ref{firstassup}$ holds for some $(p,\lambda)$, then
\begin{equation}\label{fowradpostivgrade}
 |\partial\phi|(x)
 =\sup_{y \neq x} \Bigg[ \frac{\phi(x)-\phi(y)}{d(x,y)} +\frac{\lambda}{p} d^{p-1}(x,y) \Bigg]_+
 \quad \text{for all}\,\ x \in \mathfrak{D}(\phi).
\end{equation}
Moreover, when $\lambda\geq 0$, we have $|\partial\phi|(x) =\mathfrak{l}_\phi(x)$
for all $x \in \mathfrak{D}(\phi)$.
\end{theorem}

\begin{proof}
Let $x \in \mathfrak{D}(\phi)$.
Clearly we have
\[
 |\partial\phi|(x) =\limsup_{y \to x} \frac{[\phi(x)-\phi(y)]_+}{d(x,y)}
 \leq \sup_{y \neq x} \Bigg[ \frac{\phi(x)-\phi(y)}{d(x,y)}+\frac{\lambda}{p} d^{p-1}(x,y) \Bigg]_+.
\]
In order to show the reverse inequality, without loss of generality,
we assume that there is $y \neq x$ such that
\[
 \phi(x)-\phi(y) +\frac{\lambda}{p} d^p(x,y) >0.
\]
Let $\gamma$ be a curve from $x$ to $y$ satisfying \eqref{eq:Phi}.
Then \eqref{constaweak3.24} yields that, for any $\tau \in (0,\lambda_-^{-1/(p-1)})$,
\[
 \frac{\phi(x)-\phi(\gamma(t))}{d(x,\gamma(t))}
 \geq \bigg\{ \frac{\phi(x)-\phi(y)}{d(x,y)}
 +\frac{\lambda\tau^{p-1}(1-t^{p-1})-t^{p-1}}{p\tau^{p-1}} d^{p-1}(x,y) \bigg\}
 \frac{td(x,y)}{d(x,\gamma(t))}.
\]
Combining this with $d(x,\gamma(t))\leq td(x,y)$ from \eqref{disctancecontrl},
we find
\[
 |\partial\phi|(x)
 \geq \limsup_{t \to 0} \frac{\phi(x)-\phi(\gamma(t))}{d(x,\gamma(t))}
 \geq \frac{\phi(x)-\phi(y)}{d(x,y)} +\frac{\lambda}{p} d^{p-1}(x,y) >0.
\]
Now \eqref{fowradpostivgrade} follows by taking the supremum in $y$.

When $\lambda \geq 0$, it is straightforward from \eqref{fowradpostivgrade} that
$|\partial\phi|(x) \geq \mathfrak{l}_\phi(x)$.
Since the reverse inequality clearly holds by definition,
we have $|\partial\phi|(x)=\mathfrak{l}_\phi(x)$.
\end{proof}

\begin{corollary}\label{upperconsidrgrage} %[AGS, Corollary 2.4.10]
Suppose that Assumption~$\ref{firstassup}$ holds for some $(p,\lambda)$
and $\phi$ is $\mathcal{T}_+$-lower semicontinuous.
Then $|\partial\phi|$ is a strong upper gradient for $-\phi$ and $\mathcal{T}_+$-lower semicontinuous.
\end{corollary}

\begin{proof}
It follows from Theorem~\ref{slopeuppgr} that
$|\partial \phi|$ is a weak upper gradient and $\mathfrak{l}_\phi$ is a strong upper gradient.
Hence, if $\lambda \geq 0$, Theorem~\ref{thremforwposrt} implies that
$|\partial\phi|$ is a strong upper gradient.

Next, suppose that $\lambda<0$ and $\diam(X)<\infty$.
Observe from \eqref{fowradpostivgrade} that
\begin{equation}\label{reverwseinr}
 \mathfrak{l}_\phi(x) \leq |\partial\phi|(x) -\frac{\lambda}{p}\diam(X)^{p-1}
 \quad \text{for all}\,\ x \in \mathfrak{D}(\phi).
\end{equation}
Hence, for any $\gamma \in \FAC([a,b];X)$,
$|\partial\phi| \circ \gamma\, |\gamma'_+| \in L^1(a,b)$ implies that
$\mathfrak{l}_\phi \circ \gamma\, |\gamma'_+| \in L^1(a,b)$
and $-\phi \circ \gamma$ is absolutely continuous by Remark~\ref{trickforwardupp}.
Thus, Proposition~\ref{weakbeomcestronggra} shows that
$|\partial\phi|$ is a strong upper gradient.

Finally, when $\lambda<0$ and $\diam(X)=\infty$,
we shall again prove that $-\phi\circ\gamma$ is absolutely continuous
for $\gamma \in \FAC([a,b];X)$ with $|\partial\phi| \circ \gamma\, |\gamma'_+|\in L^1(a,b)$.
Given such a curve $\gamma$,
we find from Lemma~\ref{uniformlyconverge} that $\gamma([a,b])$ is compact.
Let $X_0:=\gamma([a,b])$ be the compact forward metric space
equipped with the restricted metric $d$,
and denote by $\mathfrak{l}^0_{\phi}$ the corresponding global slope of $\phi$.
Then,  the same argument as above together with
\eqref{reverwseinr} yields $\mathfrak{l}^0_\phi \circ \gamma\, |\gamma'_+| \in L^1(a,b)$
and $-\phi\circ \gamma$ is absolutely continuous.
Therefore, $|\partial\phi|$ is always a strong upper gradient for $-\phi$.

The lower semicontinuity of $|\partial\phi|$
can be shown in the same way as Theorem~\ref{slopeuppgr}\eqref{slope-2}
by \eqref{fowradpostivgrade}.
\end{proof}

Now we present two existence results under Assumption~\ref{firstassup}
(cf.\ \cite[Corollaries~2.4.11, 2.4.12]{AGS}).
Compare them with Corollary~\ref{generalforwardexistssolu}.
Recall Remark~\ref{strongassupm} for Assumption~\ref{continudef} in the $\sigma=\mathcal{T}_+$ case.

\begin{proposition}\label{conextI} %[AGS, Corollary 2.4.11]
Suppose Assumption~$\ref{continudef}$\eqref{ass-a}, \eqref{ass-c} with $\sigma=\mathcal{T}_+$
and Assumption~$\ref{firstassup}$ for some $(p,\lambda)$.
Then, for every $x_0 \in \mathfrak{D}(\phi)$, there exists a $p$-curve
$\xi:[0,\infty) \lra X$ of maximal slope for $\phi$ with respect to $|\partial\phi|$ with $\xi(0)=x_0$.
In particular, $\xi$ satisfies the energy identity \eqref{eq:energy}.
\end{proposition}

\begin{proof}
Observe from Remark~\ref{assmupabcimpb} that
Assumption~\ref{continudef}\eqref{ass-b} also holds.
Then, from an arbitrary sequence of time steps
$(\gtau_\alpha)_{\alpha \ge 1}$ such that $\|{\gtau_\alpha}\| \to 0$
and a corresponding sequence of initial data $(\Xi^0_{\gtau_\alpha})_{\alpha \ge 1}$
with $\Xi^0_{\gtau_\alpha} =x_0$ for all $\alpha$,
we can find a limit curve $\xi \in \GMM_p(\phi;x_0)$
by Corollary~\ref{compacttexistsencesolu}.
Furthermore, since $|\partial^-\phi|=|\partial\phi|$ is a strong upper gradient for $-\phi$
by Remark~\ref{lowerbackreamrk} and Corollary~\ref{upperconsidrgrage},
we deduce from Theorem~\ref{existsassupab} that
$\xi$ is a $p$-curve of maximal slope and satisfies the energy identity.
\end{proof}

\begin{proposition} %[AGS, Corollary 2.4.12]
Suppose Assumption~$\ref{continudef}$\eqref{ass-a}, \eqref{ass-c}
and Assumption~$\ref{firstassup}$ for some $(p,\lambda)$.
If $|\partial\phi|$ is $\sigma$-sequentially lower semicontinuous
on forward bounded subsets of sublevel sets of $\phi$,
then for every $x_0 \in \mathfrak{D}(\phi)$, there exists a $p$-curve
$\xi:[0,\infty) \lra X$ of maximal slope for $\phi$ with respect to $|\partial\phi|$ satisfying $\xi(0)=x_0$
and the energy identity \eqref{eq:energy}.
\end{proposition}

\begin{proof}
One can show $|\partial^-\phi|=|\partial\phi|$ in the same way as in Remark~\ref{lowerbackreamrk}.
Moreover, it follows from Assumption~\ref{continudef}\eqref{ass-a}
and Corollary~\ref{upperconsidrgrage} that $|\partial\phi|$ is a strong upper gradient for $-\phi$.
Then we obtain the claim by a similar argument to Proposition~\ref{conextI}
together with Corollary~\ref{compacttexistsencesolu} and Theorem~\ref{existsassupab}.
\end{proof}

\begin{remark}\label{rm:lp-pl_1}
When we replace $t(1 - t^{ p-1} )$ with $t(1- t)$ in \eqref{eq:Phi}
(as in the $(\lambda,p)$-convexity; recall Remark~\ref{rm:p-conv}),
we have $d(x_0,\gamma(t)) \le t^{2/p}d(x_0,x_1)$ in place of \eqref{disctancecontrl}.
Although this does not affect Lemma~\ref{cocervityconvexfun} up to suitable modifications
(e.g., $y_* :=\gamma((r_*/d(x_*,y))^{p/2})$), however the proof of Theorem~\ref{thremforwposrt}
works only when $p \le 2$.
The lack of \eqref{fowradpostivgrade} causes problems when one tries to follow the succeeding arguments.
\end{remark}

\subsection{Regularizing effects}%%%%%%%%%%%%%%
%%%%%%%%%%%%%%%%%

In this subsection, we shall show that Assumption~\ref{firstassup} with $\lambda>0$
implies various finer properties,
including the exponential convergence to a minimizer of $\phi$
(Theorem~\ref{direestimde}; cf.\ \cite[Theorem~2.4.14]{AGS})
as well as a decay estimate of the local slope $|\partial\phi|$
(Theorem~\ref{theogmmd}; cf.\ \cite[Theorem~2.4.15]{AGS}, {\cite[Proposition~2.7]{RSS}}).

\begin{lemma}\label{phibarphi} %[AGS, Lemma 2.4.13]
Let $\phi$ satisfy Assumption~$\ref{firstassup}$ with $\lambda>0$.
Then we have
\begin{equation}\label{estimateofphi}
 \phi(x) -\inf_X \phi \leq \frac{|\partial\phi|^q(x)}{q\lambda^{q/p}} \quad
 \text{for all}\,\ x \in \mathfrak{D}(\phi).
\end{equation}
Moreover, if $\bar{x} \in \mathfrak{D}(\phi)$ is a minimizer of $\phi$, then
\begin{equation}\label{esitephi2}
 \frac{\lambda}{p} d^p(\bar{x},x)
 \leq \phi(x) -\phi(\bar{x}) \quad
 \text{for all}\,\ x \in \mathfrak{D}(\phi).
\end{equation}
\end{lemma}

\begin{proof}
For any $x,y \in \mathfrak{D}(\phi)$ with $\phi(x) \geq \phi(y)$,
we deduce from \eqref{fowradpostivgrade} and the Young inequality that
\[
 \phi(x)-\phi(y) \leq |\partial\phi|(x) d(x,y) -\frac{\lambda}{p} d^p(x,y)
 \leq \frac{|\partial\phi|^q(x)}{q\lambda^{q/p}},
\]
which implies \eqref{estimateofphi}.
When $\bar{x}$ is a minimizer of $\phi$,
\eqref{longcopexlambda} with $x_0=\bar{x}$ and $x_1=x \in \mathfrak{D}(\phi)$ yields
\[
 \frac{\lambda}p(1-t^{p-1}) d^p(\bar{x},x) \leq \phi(x)-\phi(\bar{x})
\]
since $\phi(\gamma(t)) \geq \phi(\bar{x})$.
Letting $t \to 0$ completes the proof.
\end{proof}

\begin{theorem}\label{direestimde} %[AGS, Theorem 2.4.14]
Suppose that Assumption~$\ref{firstassup}$ holds
and $\phi$ is $\mathcal{T}_+$-lower semicontinuous with $\inf_X \phi >-\infty$.
Then, for any $p$-curve $\xi:[0,\infty) \lra X$ of maximal slope for $\phi$
with respect to $|\partial\phi|$, we have
\begin{equation}\label{uupersestimate}
 \phi\big( \xi(t) \big) -\inf_X \phi
 \leq \Big\{ \phi\big( \xi(t_0) \big) -\inf_X \phi \Big\}
 \cdot \exp\big( {-q}\sgn(\lambda) |\lambda|^{q/p}(t-t_0) \big) \quad
 \text{for all}\,\ t \geq t_0>0.
\end{equation}
In particular, if $\lambda>0$ and $\bar{x} \in \mathfrak{D}(\phi)$ is a minimizer of $\phi$, then
\[
 d^p \big( \bar{x},\xi(t) \big)
 \leq \frac{p}{\lambda} \Big\{ \phi\big( \xi(t_0) \big) -\inf_X \phi \Big\}
 \cdot \exp\big( {-q}\lambda^{q/p}(t-t_0) \big) \quad
 \text{for all}\,\ t \geq t_0>0.
\]
\end{theorem}

\begin{proof}
Owing to Corollary~\ref{upperconsidrgrage}, $|\partial\phi|$ is a strong upper gradient for $-\phi$.
Thus, we observe from Proposition~\ref{curaboslpssss} that
$\phi \circ \xi$ is locally absolutely continuous and
\[
 |\xi'_+|^p(t) =|\partial\phi|^q \big( \xi(t) \big) =-(\phi \circ \xi)'(t) \quad
 \text{for $\mathscr{L}^1$-a.e.}\ t \in (0,\infty).
\]
Note that
\[
 q\sgn(\lambda)|\lambda|^{q/p} \Big\{ \phi\big( \xi(t) \big) -\inf_X \phi \Big\}
 \leq |\partial\phi|^q \big( \xi(t) \big),
\]
which follows from \eqref{estimateofphi} when $\lambda >0$ and is trivial for $\lambda \leq 0$.
Thus, $\Delta(t):=\phi(\xi(t)) -\inf_X \phi \geq 0$ satisfies
\[
 \Delta'(t) =-|\partial\phi|^q \big( \xi(t) \big)
 \leq -q\sgn(\lambda) |\lambda|^{q/p} \Delta(t) \quad
 \text{for $\mathscr{L}^1$-a.e.}\ t \in (0,\infty).
\]
This implies \eqref{uupersestimate} by integration.
The second assertion is straightforward from \eqref{esitephi2} and \eqref{uupersestimate}.
\end{proof}

We can derive from Assumption~\ref{firstassup} some estimates
stronger than those in Subsection~\ref{ssc:MY}
(cf.\ \cite[Theorem~3.1.6, Remark~3.1.7]{AGS}).

\begin{theorem}\label{slpconvexfunc} %[AGS, Theorem 3.1.6]
Suppose that Assumptions~$\ref{nonemtpassumpto}$ and $\ref{firstassup}$ hold
and let $\tau \in (0,\tau_*(\phi))$.
\begin{enumerate}[{\rm (i)}]
\item \label{delphi-1}
If $1+\lambda\tau^{p-1}>0$,
then we have, for any $x \in \mathfrak{D}(\phi)$ and $y_\tau \in J_\tau[x]$,
\begin{equation}\label{therconvexineu}
 (1+\lambda\tau^{p-1}) |\partial\phi|^q(y_\tau)
 \leq (1+\lambda\tau^{p-1}) \frac{d^p(x,y_\tau)}{\tau^p}
 \leq q\frac{\phi(x)-\Phi_\tau(x)}{\tau}
 \leq \frac{|\partial\phi|^q(x)}{(1+\lambda\tau^{p-1})^{q/p}}.
\end{equation}
The last inequality holds even when $J_\tau[x]=\emptyset$.

\item \label{delphi-2}
If $\lambda \geq 0$, then we have, for any $x \in \mathfrak{D}(\phi)$ and $y_\tau \in J_\tau[x]$,
\[
 \phi(y_\tau)-\inf_X\phi
 \leq \frac{1}{1+\tau \lambda^{q/p}(q+\lambda \tau^{p-1})} \Big\{ \phi(x)-\inf_X\phi \Big\},
 \qquad
 \sup_{\tau>0} \frac{\phi(x)-\Phi_\tau(x)}{\tau} =\frac{|\partial\phi|^q(x)}{q}.
\]
\end{enumerate}
\end{theorem}

\begin{proof}
\eqref{delphi-1}
The first inequality in \eqref{therconvexineu} was shown in Lemma~\ref{slopfirstexsts}.
In order to prove the second one, set
\[
 f(t) :=(1-t)\phi(x) +t\phi(y_\tau)
 +\frac{t}{p\tau^{p-1}}\big\{ t^{p-1}-\lambda \tau^{p-1} (1-t^{p-1}) \big\} d^p(x,y_\tau),
 \quad t \in[0,1].
\]
By \eqref{3.22remark30} with $x_0=x$ and $x_1=y_\tau$, we obtain
\[
 f(1) =\Phi(\tau,x;y_{\tau})
 \leq \Phi\big( \tau,x;\gamma(t) \big) \leq f(t)
\]
for any $t \in (0,1)$.
This implies $f'(1) \leq 0$, thereby
\[
 \phi(y_\tau)-\phi(x) +\bigg( \frac{1}{\tau^{p-1}}  +\frac{\lambda}{q} \bigg) d^p(x,y_\tau)
 \leq 0.
\]
One can rearrange this inequality to see
\[
 \frac{1+\lambda\tau^{p-1}}{q} \frac{d^p(x,y_\tau)}{\tau^{p-1}}
 \leq \phi(x)-\phi(y_\tau) -\frac{d^p(x,y_\tau)}{p\tau^{p-1}}
 =\phi(x)-\Phi_\tau(x).
\]
This is the second inequality in \eqref{therconvexineu}.
The last inequality follows from \eqref{fowradpostivgrade} and the Young inequality as
\begin{align*}
 \frac{\phi(x)-\Phi_\tau(x)}{\tau}
 &= \bigg\{ \frac{\phi(x)-\phi(y_\tau)}{d(x,y_\tau)} +\frac{\lambda}{p} d^{p-1}(x,y_\tau) \bigg\}
 \frac{d(x,y_\tau)}{\tau}
 -(1+\lambda\tau^{p-1}) \frac{d^p(x,y_\tau)}{p\tau^p} \\
 &\leq |\partial\phi|(x) \frac{d(x,y_\tau)}{\tau} -(1+\lambda\tau^{p-1}) \frac{d^p(x,y_\tau)}{p\tau^p}
 \leq \frac{|\partial\phi|^q(x)}{q(1+\lambda\tau^{p-1})^{q/p}}.
\end{align*}
If $J_{\tau}[x] =\emptyset$,
then we replace $y_{\tau}$ with $y \neq x$ and take the supremum in $y$.

\eqref{delphi-2}
Owing to \eqref{therconvexineu}, we obtain
\[
 \frac{\phi(x)-\phi(y_\tau)}{\tau}
 =\frac{\phi(x)-\Phi_\tau(x)}{\tau} +\frac{d^p(x,y_\tau)}{p\tau^p}
 \geq \bigg( \frac{1+\lambda\tau^{p-1}}{q} +\frac1p \bigg) |\partial\phi|^q(y_\tau)
 = \bigg( 1+\frac{\lambda\tau^{p-1}}{q} \bigg) |\partial\phi|^q(y_\tau).
\]
This together with \eqref{estimateofphi} implies the first assertion as
\[
 \Big\{ \phi(x)-\inf_X \phi \Big\} -\Big\{ \phi(y_\tau)-\inf_X \phi \Big\}
 \geq \tau\bigg( 1+\frac{\lambda\tau^{p-1}}{q} \bigg) q\lambda^{q/p} \Big\{ \phi(y_\tau)-\inf_X \phi \Big\}.
\]
The second assertion follows from \eqref{contrllphandtheota} and \eqref{therconvexineu} as
\[
 \frac{|\partial\phi|^q(x)}{q}
 =\limsup_{\tau \to 0} \frac{\phi(x)-\Phi_\tau(x)}{\tau}
 \leq \sup_{\tau>0} \frac{\phi(x)-\Phi_\tau(x)}{\tau}
 \leq \frac{|\partial\phi|^q(x)}{q}.
\]
\end{proof}

\begin{lemma}\label{esnonmoto} %[AGS, Lemma 3.4.1]
Suppose that Assumption~$\ref{firstassup}$ holds for some $(p,\lambda)$
satisfying one of the following$:$
$(1)$ $p \in (1,2)$ and $\lambda\geq 0;$
$(2)$ $p=2$ and $\lambda \in \mathbb{R};$
$(3)$ $p\in (2,\infty)$ and $\lambda=0$.
Given a sequence of time steps $\gtau=(\tau_k)_{k \ge 1}$ with $\lambda\|{\gtau}\|^{p-1}>-1$,
we set
\[
 \lambda_{\gtau}
 :=\frac{\log(1+\lambda\|{\gtau}\|^{p-1})}{\|{\gtau}\|^{p-1}}
 =\inf_{k \ge 1} \frac{\log(1+\lambda \tau_k^{p-1})}{\tau^{p-1}_k} \leq \lambda.
\]
Then ${\ee}^{\lambda_{\gtau}(t^k_{\gtau})^{p-1}} |\partial\phi|(\Xi^k_{\gtau})$
is non-increasing in $k$ for any $(\Xi^k_{\gtau})_{k \ge 0}$ solving \eqref{discreteequation}.
\end{lemma}

\begin{proof}
We consider only the case (1) (the other cases can be seen similarly).
Owing to $p-1\in (0,1)$,
\[
 (t^k_{\gtau})^{p-1}
 =(t^{k-1}_{\gtau} +\tau_k)^{p-1}
 \leq (t^{k-1}_{\gtau})^{p-1} +\tau_k^{p-1}.
\]
Combining this with \eqref{therconvexineu}, $\lambda_{\gtau} \geq 0$ and
${\ee}^{\lambda_{\gtau} \tau^{p-1}_k} \leq {1+\lambda \tau^{p-1}_k}$,
we obtain
\[
 {\ee}^{\lambda_{\gtau}(t^k_{\gtau})^{p-1}} |\partial\phi|(\Xi^k_{\gtau})
 \leq {\ee}^{\lambda_{\gtau}(t^{k-1}_{\gtau})^{p-1}}
 {\ee}^{\lambda_{\gtau}\tau^{p-1}_k}
 \frac{|\partial\phi|(\Xi^{k-1}_{\gtau})}{1+\lambda \tau^{p-1}_k}
 \leq {\ee}^{\lambda_{\gtau}(t^{k-1}_{\gtau})^{p-1}}|\partial\phi|(\Xi^{k-1}_{\gtau})
\]
as desired.
\end{proof}

The above lemma furnishes the following results including a decay estimate of $|\partial\phi|$
(a kind of regularizing effect).

\begin{theorem}\label{theogmmd} %[AGS, Theorem 2.4.15]
Suppose that $\phi$ is $\mathcal{T}_+$-lower semicontinuous,
and that Assumptions~$\ref{nonemtpassumpto}$, $\ref{firstassup}$ hold
for some $(p,\lambda)$ satisfying one of the following$:$
$(1)$ $p \in (1,2)$ and $\lambda\geq 0;$
$(2)$ $p=2$ and $\lambda \in \mathbb{R};$
$(3)$ $p\in (2,\infty)$ and $\lambda=0$.
Then, for every $x_0 \in \mathfrak{D}(\phi)$,
each element $\xi \in \GMM_p(\phi;x_0)$ is locally Lipschitz in $(0,\infty)$
and satisfies the following.
\begin{enumerate}[{\rm (i)}]
\item \label{reg-1}
For any $t>0$, the right forward metric derivative
\[
 |\xi'_R|(t) :=\lim_{s \to t^+} \frac{d(\xi(t),\xi(s))}{s-t}
\]
exists,
$\xi(t) \in \mathfrak{D}(|\partial\phi|)$, and
\[
 \frac{\dd}{{\dd}t_+} \phi \big( \xi(t) \big)
 =-|\partial\phi|^q \big( \xi(t) \big)
 =-|\xi'_R|^p(t)
 =-|\partial\phi| \big( \xi(t) \big) |\xi'_R|(t)
\]
holds, where $\frac{\dd}{{\dd}t_+}$ denotes the right derivative.

\item \label{reg-2}
$\phi(\xi(t))$ is convex in $t \ge 0$ if $\lambda \geq 0$,
and ${\ee}^{\lambda t^{p-1}} |\partial\phi|(\xi(t))$ is non-increasing and right continuous in $t>0$.
Moreover, we have
\begin{align}
 \frac{t}{q} |\partial\phi|^q \big( \xi(t) \big)
 &\leq {\ee}^{q\lambda_- t^{p-1}} \big\{ \phi(x_0)-\Phi_t(x_0) \big\},
 \label{frsitest}\\
 t|\partial \phi|^q\big( \xi(t) \big)
 &\leq \big( 1+p\lambda_+ C(p,\lambda,t) \big)
 {\ee}^{-q\lambda t^{p-1}} \Big\{ \phi(x_0)-\inf_X \phi \Big\},
 \label{secondest}
\end{align}
where $\lambda_+:=\max\{\lambda,0\}$ and
\[
 C(p,\lambda,t):=
 \int^t_0 s^{p-2} {\ee}^{q\lambda s^{p-1}} {\ee}^{-q\sgn(\lambda)|\lambda|^{q/p} s} \,{\dd}s.
\]
In particular, $C(p,0,t)=t^{p-1}/(p-1)$ and $C(2,\lambda,t)=t$.
\end{enumerate}
\end{theorem}

\begin{proof}
Take $(\gtau_\alpha)_{\alpha \ge 1}$ associated with $\xi \in \GMM_p(\phi;x_0)$
as in Definition~\ref{defominvcurve}.
First, we shall construct a right continuous function $\mathcal{G}_R$ on $(0,\infty)$ such that
\[
 \mathcal{G}_R(t)
 \geq {\ee}^{\lambda t^{p-1}} |\partial\phi| \big( \xi(t) \big)\,\ \text{for any}\ t>0, \quad
 \mathcal{G}_R(t)
 ={\ee}^{\lambda t^{p-1}} |\partial\phi| \big( \xi(t) \big)
 \,\ \text{for $\mathscr{L}^1$-a.e.}\ t \in (0,\infty)
\]
(i.e., $\mathcal{G}_R$ is the right continuous representative
of ${\ee}^{\lambda t^{p-1}} |\partial\phi| (\xi(t))$).
Define a function $\mathcal{G}_{\alpha}$ on $[0,\infty)$ by
\[
 \mathcal{G}_{\alpha}(0) := |\partial\phi|(\Xi^0_{\gtau_{\alpha}}), \qquad
 \mathcal{G}_{\alpha}(t)
 := {\ee}^{\lambda_{\gtau_{\alpha}}(t^k_{\gtau_{\alpha}})^{p-1}} |\partial\phi|(\Xi^k_{\gtau_{\alpha}})
 \,\ \text{ for } t \in (t^{k-1}_{\gtau_{\alpha}},t^k_{\gtau_{\alpha}}],
\]
where $\lambda_{\gtau_{\alpha}}$ is as in Lemma~\ref{esnonmoto}.
Owing to Lemma~\ref{esnonmoto},
$\mathcal{G}_{\alpha}$ is a non-increasing function provided $\lambda\|{\gtau_{\alpha}}\|^{p-1}>-1$.
By \cite[Lemma~3.3.3]{AGS}, passing to a subsequence if necessary,
we can assume that $\lim_{\alpha \to \infty} \mathcal{G}_{\alpha}(t)$ exists for all $t \ge 0$
and the limit function $\mathcal{G}$ is non-increasing.
Hence, $\mathcal{G}$ has at most countably many discontinuous points.
Since $\lambda_{\gtau_\alpha} \to \lambda$,
we deduce from Corollary~\ref{upperconsidrgrage},
Theorem~\ref{existsassupab}\eqref{exstasolucon3} and Remark~\ref{lowerbackreamrk} that
$\mathcal{G}(t) ={\ee}^{\lambda t^{p-1}} |\partial^- \phi|(\xi(t)) ={\ee}^{\lambda t^{p-1}} |\partial\phi|(\xi(t))$
for $\mathscr{L}^1$-a.e.\ $t\in (0,\infty)$.
Now, define
\[
 \mathcal{G}_R(t) := \lim_{s \to t^+} \mathcal{G}(s) \,\ \text{ for } t \geq 0.
\]
Then $\mathcal{G}_R(t) \leq \mathcal{G}(t)$ for all $t \geq 0$, since $\mathcal{G}$ is non-increasing.
Moreover, the almost everywhere continuity of $\mathcal{G}$ implies that
$\mathcal{G}_R$ is right continuous and $\mathcal{G}_R(t)=\mathcal{G}(t)$
for $\mathscr{L}^1$-a.e.\ $t \in (0,\infty)$,
and the lower semicontinuity of $|\partial\phi|$ yields that
$\mathcal{G}_R(t) \geq {\ee}^{\lambda t^{p-1}}|\partial\phi|(\xi(t))$ for all $t \geq 0$.
We also set $\mathcal{S}_R(t):={\ee}^{-\lambda t^{p-1}}\mathcal{G}_R(t)$,
which is the right continuous representative of $|\partial\phi|(\xi(t))$.

\eqref{reg-1}
We deduce from Theorem~\ref{existsassupab} and Proposition~\ref{curaboslpssss}
that $|\xi'_+|^p(t)=|\partial\phi|^q(\xi(t)) =-(\phi \circ \xi)'(t)$ for $\mathscr{L}^1$-a.e.\ $t\in (0,\infty)$
and the energy identity \eqref{eq:energy} holds.
Thus, for any $t \geq 0$, we have
\begin{equation}\label{rightderviephi}
 \limsup_{\delta \to 0^+} \frac{d(\xi(t),\xi(t+\delta))}{\delta}
 =\limsup_{\delta \to 0^+} \frac{1}{\delta}
 \int^{t+\delta}_t |\partial\phi|^{q/p} \big( \xi(s) \big) \,{\dd}s
 \leq \mathcal{S}^{q/p}_R(t).
\end{equation}

Given $T>0$, put $\mu:=\min\{\lambda,0\}=-\lambda_-$ and $\eta :=\inf_{t \in [0,T]}\phi(\xi(t))$,
and consider the function $h(t):={\ee}^{q\mu t^{p-1}} \{ \phi(\xi(t)) -\eta \}$.
We claim that $h$ is convex in $[0,T]$.
Since $h$ is continuous, it suffices to show that
$h'$ is $\mathscr{L}^1$-a.e.\ equal to a non-decreasing function.
Note that, for $\mathscr{L}^1$-a.e.\ $t \in (0,\infty)$,
\begin{equation}\label{dervativeofer}
 h'(t) =-{\ee}^{q\mu t^{p-1}}|\partial\phi|^q \big( \xi(t) \big) +p\mu t^{p-2} h(t) \leq 0
\end{equation}
since $h(t) \geq 0$ and $\mu \leq 0$.
Thus, $h$ is non-increasing.
Now, if $\lambda \leq 0$,
then we find $\mu=\lambda$ and $\mu=0$ unless $p=2$,
and hence the latter term of $h'(t)$ in \eqref{dervativeofer} is non-decreasing.
As for the former term, recall that
$\mathcal{G}(t)={\ee}^{\lambda t^{p-1}}|\partial\phi|(\xi(t))$ for $\mathscr{L}^1$-a.e.\ $t\in (0,\infty)$
and that $\mathcal{G}$ is non-increasing.
Therefore, we find that $h$ is convex.
In the case of $\lambda>0$, Lemma~\ref{convedefremark}\eqref{pcon-1} yields that
Assumption~\ref{firstassup} holds also for $(p,0)$.
Thus we can apply the above argument, for $\mu=0$ in both cases.

The convexity implies that $h$ is right differentiable
and the right derivative $\frac{\dd}{{\dd}t_+}h$ is non-decreasing.
Thus, $\phi \circ \xi$ is also right differentiable and
\[
 \frac{\dd}{{\dd}t_+} \phi \big( \xi (t) \big)
 \leq \lim_{s \to t^+} \frac{\dd}{{\dd}t_+} \phi\big( \xi(s) \big).
\]
By choosing a sequence $(s_i)_{i \ge 1}$ with $s_i \to t^+$
and $(\phi \circ \xi)'(s_i)=-|\partial\phi|^q(\xi(s_i))=-\mathcal{S}_R^q(s_i)$, we find
\begin{equation}\label{righttruderiv}
 \frac{\dd}{{\dd}t_+} \phi \big( \xi(t) \big)
 \leq \lim_{i \to \infty} (\phi \circ \xi)'(s_i)
 =-\mathcal{S}^q_R(t) \quad \text{for any}\,\ t>0
\end{equation}
(this also implies that $\mathcal{S}_R(t)<\infty$ for all $t>0$).
A direct calculation yields
\begin{align}
 \frac{\dd}{{\dd}t_+}\phi \big( \xi(t) \big)
 &\geq -\liminf_{s \to t^+} \bigg\{
 \frac{[\phi(\xi(t))-\phi(\xi(s))]_+}{d(\xi(t),\xi(s))} \frac{d(\xi(t),\xi(s))}{s-t} \bigg\} \nonumber\\
 &\geq-\limsup_{s \to t^+} \frac{[\phi(\xi(t))-\phi(\xi(s))]_+}{d(\xi(t),\xi(s))}
 \liminf_{s \to t^+} \frac{d(\xi(t),\xi(s))}{s-t} \nonumber\\
 &\geq -|\partial\phi| \big( \xi(t) \big) \liminf_{s \to t^+} \frac{d(\xi(t),\xi(s))}{s-t}.
 \label{lastfinequa}
\end{align}
Combining this with \eqref{rightderviephi} and \eqref{righttruderiv} furnishes
\[
 \mathcal{S}_R(t) \limsup_{s \to t^+} \frac{d(\xi(t),\xi(s))}{s-t}
 \leq \mathcal{S}^q_R(t)
 \leq |\partial\phi| \big( \xi(t) \big)\liminf_{s \to t^+} \frac{d(\xi(t),\xi(s))}{s-t}
 \leq \mathcal{S}_R(t) \liminf_{s \to t^+} \frac{d(\xi(t),\xi(s))}{s-t}.
\]
Hence, $|\xi'_R|(t)$ exists and
\[
 |\xi'_R|(t) =\mathcal{S}^{q/p}_R(t)
 =|\partial\phi|^{q/p} \big( \xi(t) \big) <\infty
\]
for all $t>0$, and especially $\xi(t) \in \mathfrak{D}(|\partial\phi|)$.
Moreover, since equality holds in \eqref{lastfinequa}, we find
\[
 \frac{\dd}{{\dd}t_+}\phi \big( \xi(t) \big)
 =-|\partial\phi| \big( \xi(t) \big) |\xi'_R|(t)
 =-|\partial\phi|^q \big( \xi(t) \big)
 =-|\xi'_R|^p(t).
\]
Hence, \eqref{reg-1} follows.
Note also that $\mathcal{S}_R(t)=|\partial\phi|(\xi(t))$ and
$\mathcal{G}_R(t)={\ee}^{\lambda t^{p-1}}|\partial\phi|(\xi(t))$ hold for all $t \geq 0$.

\eqref{reg-2}
First of all, if $\lambda \geq 0$,
then the convexity of $\phi(\xi(t))$ follows from that of $h$.
Moreover, since $\mathcal{G}$ is non-increasing,
so is $\mathcal{G}_R(t)={\ee}^{\lambda t^{p-1}}|\partial\phi|(\xi(t))$.
To see \eqref{frsitest}, assuming $\lambda \leq 0$, we have
\[
 \frac{t}{q} {\ee}^{q\lambda t^{p-1}}|\partial\phi|^q \big( \xi(t) \big)
 \leq \frac1q \int^t_0 {\ee}^{q\lambda s^{p-1}} |\partial\phi|^q \big( \xi(s) \big) \,{\dd}s
 \leq \frac1q \int^t_0  |\partial\phi|^q \big( \xi(s) \big) \,{\dd}s.
\]
Combining this with the energy identity \eqref{eq:energy},
H\"older inequality and $\xi(0)=x_0$, we obtain
\begin{align*}
 \frac{t}q {\ee}^{q\lambda t^{p-1}} |\partial\phi|^q \big( \xi(t) \big)
 &\leq \phi(x_0) -\phi \big( \xi(t) \big) -\frac1p \int^t_0 |\xi'_+|^p(s) \,{\dd}s \\
 &\leq \phi(x_0) -\phi \big( \xi(t) \big) -\frac{d^p(\xi(0),\xi(t))}{pt^{p-1}} \\
 &\leq \phi(x_0)-\Phi_t(x_0).
\end{align*}
This shows \eqref{frsitest}.
We can reduce the $\lambda>0$ case to the above argument with $\lambda=0$
by Lemma~\ref{convedefremark}\eqref{pcon-1}.

Finally, we prove \eqref{secondest}.
Since there is nothing to prove if $\inf_X \phi=-\infty$, we assume $\eta:=\inf_X \phi>-\infty$.
As ${\ee}^{\lambda t^{p-1}}|\partial\phi|(\xi(t))$ is non-increasing,
it follows from \eqref{uupersestimate} that
\begin{align*}
 &{\ee}^{q\lambda t^{p-1}} \big\{ \phi \big( \xi(t) \big) -\eta \big\} -\big\{ \phi(x_0)-\eta \big\} \\
 &= \int^t_0 -{\ee}^{q\lambda s^{p-1}} |\partial\phi|^q \big( \xi(s) \big) \,{\dd}s
 +p\lambda \int^t_0 s^{p-2} {\ee}^{q\lambda s^{p-1}} \big\{ \phi \big( \xi(s) \big) -\eta \big\} \,{\dd}s \\
 &\leq -t{\ee}^{q\lambda t^{p-1}} |\partial\phi|^q \big( \xi(t) \big)
 +p\lambda_+ \int^t_0 s^{p-2} {\ee}^{q\lambda s^{p-1}}
 {\ee}^{-q\sgn(\lambda)|\lambda|^{q/p} s} \big\{ \phi(x_0)-\eta \big\} \,{\dd}s.
\end{align*}
Therefore, we have
\begin{align*}
 t{\ee}^{q\lambda t^{p-1}}|\partial\phi |^q \big( \xi(t) \big)
 &\leq \big( 1+p\lambda_+ C(p,\lambda,t) \big)
 \big\{ \phi(x_0)-\eta \big\}
\end{align*}
as desired.
One can see $C(p,0,t)=t^{p-1}/(p-1)$ and $C(2,\lambda,t)=t$ by a direct calculation.
\end{proof}

Note that, if $\inf_X \phi>-\infty$,
then the convexity of ${\ee}^{-q\lambda_- t^{p-1}} \{ \phi(\xi(t))-\inf_X \phi \}$ for $\lambda<0$
can also be seen in the same way as above.

\begin{remark}\label{rm:lp-pl_2}
Continuing the discussion in Remark~\ref{rm:lp-pl_1},
suppose that $t(1 - t^{ p-1} )$ is replaced with $t(1- t)$ in \eqref{eq:Phi}.
Then, we have
\[
(1+\lambda\tau^{p-1}) |\partial\phi|^q(y_\tau)
 \leq \frac{p}{q} \frac{|\partial\phi|^q(x)}{(1+\lambda\tau^{p-1})^{q/p}}
\]
instead of \eqref{therconvexineu}, provided that \eqref{fowradpostivgrade} holds.
Thus, we can follow the lines of Lemma~\ref{esnonmoto} and Theorem~\ref{theogmmd}
only when $p \le 2$ (i.e., $p/q \le 1$).
\end{remark}

We conclude this subsection by presenting a regularity result for gradient curves in the Finsler case.

\begin{corollary}\label{regularoffINSLERCASE}
Let $(M,F)$ be a forward complete Finsler manifold and
$\phi:M \lra \mathbb{R}$ be a lower semicontinuous, $\lambda$-geodesically convex function.
\begin{enumerate}[{\rm (i)}]
\item \label{Fin-1}
For any $x_0 \in \mathfrak{D}(\phi)$,
there exists a curve $\xi:[0,\infty) \lra M$ of maximal slope
for $\phi$ with respect to $|\partial\phi|$ with $\xi(0)=x_0$.

\item \label{Fin-2}
If $\lambda>0$, then $\phi$ has a unique minimizer $\bar{x} \in M$,
and $|\partial\phi|(\xi(t))$ decreases to $0$ and $\xi(t) \to \bar{x}$ as $t \to \infty$.

\item \label{Fin-3}
If $\lambda \geq 0$ and $\phi \in C^l(M)$ for some $l \ge 1$,
then $\xi$ is $C^l$ in $(0,T)$ with $T :=\inf\{ t \geq 0 \,|\, \xi(t)=\inf_X \phi \}$.
Moreover, if $\lambda>0$ and $T<\infty$,
then we have $\xi(t)=\bar{x}$ for all $t \geq T$.
\end{enumerate}
\end{corollary}

\begin{proof}
\eqref{Fin-1}
Thanks to Example~\ref{pgeodesicconve2}\eqref{lambda-a},
Assumption~\ref{firstassup} holds for $(2,\lambda)$.
Since $(M,d_F)$ is forward boundedly compact by the Hopf--Rinow theorem,
we find that Assumption~\ref{continudef}\eqref{ass-a}, \eqref{ass-c} hold.
Then, the existence of a ($2$-)curve $\xi$ of maximal slope follows from Proposition~\ref{conextI}.

\eqref{Fin-2}
Note that
\eqref{inffunctionobund} and Assumption~\ref{nonemtpassumpto} also hold due to Remarks~\ref{existsju}, \ref{assmupabcimpb}.
Then, in view of Lemma~\ref{cocervityconvexfun}\eqref{m*-2},
$\phi$ has a unique minimizer $\bar{x} \in M$.
Moreover, on the one hand, Theorem~\ref{direestimde} yields
\[
 d_F^2 \big( \bar{x},\xi(t) \big)
 \leq \frac{2}{\lambda} {\ee}^{-2\lambda t} \big\{ \phi(x_0) -\phi(\bar{x}) \big\}
\]
for any $t \geq 0$, thereby $\xi(t) \to \bar{x}$ as $t \to \infty$.
On the other hand, it follows from Theorem~\ref{theogmmd}\eqref{reg-2} that
$|\partial\phi|(\xi(t))$ is decreasing. Thus  \eqref{secondest} yields
\[
 0 \leq \lim_{t \to \infty} |\partial \phi|^2 \big( \xi(t) \big)
 \leq \lim_{t \to \infty} \frac{1+2\lambda t}{t} {\ee}^{-2\lambda t} \{ \phi(x_0)-\phi(\bar{x}) \} =0.
\]

\eqref{Fin-3}
When $\phi \in C^l(M)$, we find from Corollary~\ref{bascirgeulfinsler}
that $\xi$ is $C^1$ and satisfies $\xi'(t)=\nabla(-\phi)(\xi(t))$.
Now, when $\phi(x)>\inf_X \phi$, we deduce $F(\nabla(-\phi)(x))=|\partial\phi|(x)>0$ from the $\lambda$-convexity
along a minimal geodesic from $x$ to some $y$ with $\phi(y)<\phi(x)$.
This implies that $\nabla(-\phi)$ is a $C^{l-1}$-vector field around $\xi(t)$ for any $t \in (0,T)$
(recall Remark~\ref{regularifinslerm}),
and hence $\xi$ is $C^l$ in $(0,T)$ by $\xi'(t)=\nabla(-\phi)(\xi(t))$.
In the case of $\lambda>0$ and $T<\infty$,  we see that $\xi(t)=\bar{x}$ for all $t \geq T$
since $\phi \circ \xi$ is non-increasing.
\end{proof}

We remark that the convergences in \eqref{Fin-2} are exponentially fast.
Note also that the argument in \eqref{Fin-3} shows that
$\phi$ is of lower $2$-growth for general $\lambda \in \mathbb{R}$
(or of lower $p$-growth if $\phi$ is $(p,\lambda)$-geodesically convex).
Indeed, given $x \in \mathfrak{D}(|\partial\phi|)$ and any $y \in M$,
the $\lambda$-convexity along a minimal geodesic $\gamma:[0,1] \lra M$ from $x$ to $y$
yields
\begin{align*}
 \phi(y)
 &\geq \phi(x) +\lim_{t \to 0}\frac{\phi(\gamma(t))-\phi(x)}{t} +\frac{\lambda}{2}d^2(x,y)
 \geq \phi(x) -|\partial\phi|(x)d(x,y) +\frac{\lambda}{2}d^2(x,y) \\
 &\geq \phi(x) -\frac{1}{2}|\partial\phi|^2(x) +\frac{\lambda -1}{2}d^2(x,y).
\end{align*}

\subsection{Heat flow on compact Finsler manifolds}\label{heatflowfins}%%%%%%
%%%%%%%%%%%%%%%%%%%%%%%%%%%%%

This subsection is devoted to a study of heat flow on compact Finsler manifolds.
It is well known that heat flow can be regarded as gradient flow in the Wasserstein space;
we refer to \cite{Erb,Ograd,OS,Vi} for more details.
In what follows, let $(M,F,\m)$ be a compact Finsler manifold
endowed with a smooth positive measure $\m$.
Along \cite{Erb}, we shall present a slightly more general framework than \cite{OS}.

\subsubsection{Finsler structure of the Wasserstein space}%%%%%%%
%%%%%%%%%%%%%%%%%%%%%%

Denote by $\mathcal{P}(M)$ the collection of Borel probability measures on $M$.
For $\mu_0,\mu_1\in \mathcal{P}(M)$, the \emph{$L^2$-Wasserstein distance} is defined by
\[
 d_W(\mu_0,\mu_1)
 :=\inf_{\pi} \bigg( \int_{M \times M} d^2_F(x,y) \,\pi({\dd}x\,{\dd}y) \bigg)^{1/2},
\]
where $\pi \in \mathcal{P}(M \times M)$ runs over all couplings of $(\mu_0,\mu_1)$.
The fundamental theory of optimal transport yields that
the $L^2$-Wasserstein space $(\mathcal{P}(M),d_W)$ is a compact $\lambda_F(M)$-metric space
and any $\mu_0,\mu_1 \in \mathcal{P}(M)$ admit a minimal geodesic $(\mu_t)_{t\in [0,1]}$
from $\mu_0$ to $\mu_1$ (cf.\ \cite[Corollary~4.17]{KZ}).

Let $\mathcal{P}_{\mathrm{ac}}(M;\m) \subset \mathcal{P}(M)$
be the set of measures absolutely continuous with respect to $\m$.
Then we have the following standard fact (cf.\ \cite[Theorem~4.10]{O}).

\begin{proposition}\label{basicminpo}
For any $\mu_0 \in \mathcal{P}_{\mathrm{ac}}(M;\m)$ and $\mu_1\in \mathcal{P}(M)$,
there exists a $\mu_0$-a.e.\ unique vector field  $\Psi_{\mu_0}^{\mu_1}$ on $M$ such that
\begin{enumerate}[{\rm (i)}]
\item \label{nup1}
$\pi =(\id_M, \exp\Psi_{\mu_0}^{\mu_1})_\sharp \mu_0$ is a unique optimal coupling of $(\mu_0,\mu_1);$

\item \label{nup4}
$\mu_t =[\exp(t\Psi^{\mu_1}_{\mu_0})]_\sharp \mu_0$, $t\in [0,1]$,
is a unique minimal geodesic from $\mu_0$ to $\mu_1;$

\item \label{nup5}
$d^2_W(\mu_0,\mu_1)=\int_M F^2( \Psi^{\mu_1}_{\mu_0}) \,{\dd}\mu_0;$

\item \label{nup3}
there exists a $(d_F^2/2)$-convex function $\varphi:M \lra \mathbb{R}$
such that $\Psi_{\mu_0}^{\mu_1}=\nabla\varphi$,
\end{enumerate}
where $f_\sharp \mu$ denotes the push-forward measure of $\mu$ by $f$.
\end{proposition}

The vector field $\Psi^{\mu_1}_{\mu_0}$ will be referred to as the \emph{optimal transport vector field}.
Observe that $\Psi^{\mu_t}_{\mu_0}=t\Psi_{\mu_0}^{\mu_1}$ holds for all $t \in [0,1]$.
By \cite{Ouni}, the function $\varphi$ in \eqref{nup3} is Lipschitz and twice differentiable $\m$-a.e.

Now we recall the Finsler structure of $\mathcal{P}(M)$ introduced in \cite{OS}
(see \cite{Erb,Vi} for the Riemannian case).
For $\mu\in \mathcal{P}(M)$,
let $L^2(\mu;TM)$ denote the space of measurable vector fields $\w$
with the asymmetric norm
\[
 \|{\w}\|_{\mu} :=\bigg( \int_M F^2(\w) \,{\dd}\mu \bigg)^{1/2} <\infty.
\]
We similarly define $L^2(\mu;T^*M)$ as the space of measurable $1$-forms $\zeta$ with
\[
 \|\zeta\|^*_{\mu} :=\bigg( \int_M F^*(\zeta)^2 \,{\dd}\mu \bigg)^{1/2} <\infty.
\]
We also define
\[
 \langle \zeta,\w \rangle_{\mu} :=\int_M \zeta(\w) \,{\dd}\mu
\]
for $\w \in L^2(\mu;TM)$ and $\zeta \in L^2(\mu;T^*M)$,
and the \emph{Legendre transformation} $\LL_\mu :L^2(\mu;TM) \lra L^2(\mu;T^*M)$
by applying the pointwise Legendre transformation induced by $F$
(recall Subsection~\ref{ssc:Finsler}).
The \emph{tangent} and \emph{cotangent spaces} of $\mathcal{P}(M)$ at $\mu$ are defined by
\[
 T_\mu \mathcal{P}(M) :=\overline{\{ \nabla\varphi \,|\, \varphi \in C^\infty(M) \}}^{\|\cdot\|_{\mu}},
 \qquad
 T^*_\mu \mathcal{P}(M) :=\overline{\{ {\dd}\varphi \,|\, \varphi \in C^\infty(M) \}}^{\|\cdot\|^*_{\mu}},
\]
respectively.
We summarize some properties of these spaces (see Appendix \ref{addwasser} for the proof).

\begin{proposition}\label{normstrcturelemma}
Given $\mu\in \mathcal{P}(M)$, we have the following.
\begin{enumerate}[{\rm (i)}]
\item \label{dualpro2}
The Legendre transformation $\LL_\mu:L^2(\mu;TM) \lra L^2(\mu;T^*M)$ is a homeomorphism.
In particular, its restriction $\LL_\mu:T_\mu \mathcal{P}(M) \lra T^*_\mu \mathcal{P}(M)$
is also a homeomorphism.

\item\label{dualpro1}
For any $\zeta \in L^2(\mu;T^*M)$, we have
$\|\zeta\|^*_{\mu} =\sup_{\w \in L^2(\mu;TM) \setminus \{0\}} \langle \zeta,\w \rangle_{\mu}/\|{\w}\|_\mu$
and $\w=\LL^{-1}_{\mu}(\zeta)$ is a unique element in $L^2(\mu;TM)$ satisfying
$\|{\w}\|_{\mu}^2 =(\|\zeta\|^*_{\mu})^2 =\langle \zeta,\w \rangle_{\mu}$.

\item\label{Riesz}
For every bounded linear functional $\mathcal{Q}$ on $L^2(\mu;TM)$,
there exists a unique $1$-form $\zeta \in L^2(\mu;T^*M)$ such that
$\mathcal{Q}(\w) =\langle \zeta,\w \rangle_\mu$ for all $\w \in L^2(\mu;TM)$.

\item\label{basicconnection}
We have
$T^*_\mu \mathcal{P}(M) =\{ \zeta \in L^2(\mu;T^*M) \,|\,
 \langle \zeta,\w \rangle_\mu=0\ \text{for all}\ \w \in \mathbf{Ker(div)}(\mu) \}$, where
\[
 \mathbf{Ker(div)}(\mu)
 := \big\{ {\w} \in L^2(\mu;TM) \,\big|\,
 \langle {\dd}\varphi,\w \rangle_{\mu} =0\ \text{for all}\ \varphi \in C^\infty(M) \big\}.
\]

\item\label{identityrelation}
If $\vv,\w \in T_\mu \mathcal{P}(M)$ satisfy $\vv-\w \in \mathbf{Ker(div)}(\mu)$,
then we have $\vv=\w$.
\end{enumerate}
\end{proposition}

The next theorem is seen in the same manner as \cite[Lemma~7.2, Theorem~7.3]{OS}
or \cite[Proposition~2.5]{Erb}.
(In fact, \eqref{continequa} does not depend on the metric
and can be reduced to any bi-Lipschitz equivalent Riemannian metric.)

\begin{theorem}\label{existfloweq}
Let $I \subset \mathbb{R}$ be an open interval
and $\mu=(\mu_t)_{t \in I}$ be a continuous curve in $\mathcal{P}(M)$.

\begin{enumerate}[{\rm (i)}]
\item\label{continqeuq1}
If $\mu \in \FAC^2_{\loc}(I;\mathcal{P}(M))$,
then there exists a Borel vector field $\vv :I \times M \lra TM$
with $\|{\vv}_t\|_{\mu_t} \in L^2_{\loc}(I)$ $(\vv_t(x):=\vv(t,x) \in T_xM)$ such that
\begin{itemize}
\item[(a)]
$\vv_t \in T_{\mu_t} \mathcal{P}(M)$ for $\mathscr{L}^1$-a.e.\ $t \in I;$

\item[(b)]
the \emph{continuity equation} $\partial_t \mu_t+\di(\vv_t\cdot {\mu_t})=0$ holds
in the sense of distributions, i.e.,
\begin{equation}\label{continequa}
 \int_I \int_M \{ \partial_t \varphi(t,x) +\langle {\dd}\varphi(t,x),\vv_t(x) \rangle \} \,\mu_t({\dd}x)\,{\dd}t =0
 \quad \text{for all}\,\ \varphi \in C^\infty_0(I \times M),
\end{equation}
where $C^{\infty}_0(I \times M)$ denotes the set of $C^{\infty}$-functions on $I \times M$
of compact support.
\end{itemize}
Moreover, such a vector field $\vv_t$ is unique and satisfies $\|{\vv}_t\|_{\mu_t} =|\mu'_+|(t)$
for $\mathscr{L}^1$-a.e.\ $t  \in I$.

\item\label{continqeuq2}
If $\mu$ satisfies \eqref{continequa} above for some Borel vector field $(\vv_t)_{t \in I}$
with $\|{\vv}_t\|_{\mu_t} \in L^2_{\loc}(I)$,
then $\mu$ is locally forward absolutely continuous with $|\mu'_+|(t) \leq \|{\vv}_t\|_{\mu_t}$.
\end{enumerate}
\end{theorem}

The vector field $\vv=(\vv_t)_{t \in I}$ in \eqref{continqeuq1} above is called
the \emph{tangent vector field} of the curve $\mu$.

\begin{lemma}\label{tangerverproer}
\begin{enumerate}[{\rm (i)}]
\item\label{phisctur1}
For any $\mu_0\in \mathcal{P}_{\mathrm{ac}}(M;\m)$ and $\mu_1\in \mathcal{P}(M)$,
we have $\Psi^{\mu_1}_{\mu_0} \in T_{\mu_0}\mathcal{P}(M)$.

\item\label{weakconvegoptimavect}
For any $(\mu_t)_{t \in I} \in \FAC^2_{\loc}(I;\mathcal{P}_{\mathrm{ac}}(M;\m))$
and its tangent vector field $(\vv_t)_{t \in I}$, we have, for $\mathscr{L}^1$-a.e.\ $t \in I$,
\begin{equation}\label{eq:Psi/h}
 \frac1h \Psi_{\mu_t}^{\mu_{t+h}} \to \vv_t \text{ weakly in }L^2(\mu_t;TM) \text{ as } h \to 0^+.
\end{equation}
\end{enumerate}
\end{lemma}

\begin{proof}
Thanks to Proposition~\ref{basicminpo}\eqref{nup3},
one can show \eqref{phisctur1} via an approximation by smooth functions.
\eqref{weakconvegoptimavect} follows from a similar argument
to the proof of \cite[Lemma~2.7]{Erb} along with
Proposition~\ref{normstrcturelemma}\eqref{identityrelation}, Theorem~\ref{existfloweq}
and \eqref{phisctur1} above.
\end{proof}

\subsubsection{Subdifferentials and gradient flows}%%%%%%%%%%%%
%%%%%%%%%%%%%%%%%%%

Let $\phi:\mathcal{P}(M) \lra (-\infty,\infty]$ be a $\mathcal{T}_+$-lower semicontinuous function
on the Wasserstein space $(\mathcal{P}(M),d_W)$
with $\mathfrak{D}({\phi}) \subset \mathcal{P}_{\mathrm{ac}}(M;\m)$.

\begin{definition}[Subdifferentials]\label{badfscsubdiff}
For $\mu \in \mathfrak{D}({\phi})$, a $1$-form $\zeta\in L^2(\mu;T^*M)$ is said to  \emph{belong to
the subdifferential} $\partial{\phi}(\mu)$ if
\[
 {\phi}(\nu)-{\phi}(\mu)
 \geq \langle \zeta,\Psi^{\nu}_{\mu} \rangle_{\mu}
 +o\big( d_W(\mu,\nu) \big) \quad \text{ for all}\,\ \nu \in \mathcal{P}(M).
\]
We call $\zeta \in \partial{\phi}(\mu)$ a \emph{strong subdifferential} if it satisfies
\[
 \phi \big( (\exp\w)_\sharp \mu \big) -\phi(\mu)
 \geq \langle \zeta,\w \rangle_{\mu}
 +o(\|{\w}\|_{\mu}) \quad \text{ for all}\,\ {\w} \in L^2(\mu;TM).
\]
\end{definition}

\begin{proposition}\label{properofsubdiff}
Given $\mu\in \mathfrak{D}({\phi})$, we have the following.
\begin{enumerate}[{\rm (i)}]
\item\label{lowerbounded}
If $\partial{\phi}(\mu) \neq \emptyset$,
then $|\partial {\phi}|(\mu) \leq \inf_{\zeta \in \partial{\phi}(\mu)}\|{-}\zeta\|^*_{\mu}$.

\item\label{charastrongsubd}
If $\zeta \in \partial{\phi}(\mu) \cap T^*_\mu \mathcal{P}(M)$,
then $\zeta$ is a strong subdifferential.

\item\label{weaksubdfiffde}
If $\phi$ is $\lambda$-geodesically convex,
then $\zeta \in L^2(\mu;T^*M)$ belongs to $\partial{\phi}(\mu)$ if and only if
\begin{equation}\label{stillconvexloer}
 \phi(\nu) -\phi(\mu)
 \geq \langle \zeta,\Psi^\nu_\mu \rangle_{\mu}
 +\frac{\lambda}{2} d^2_W(\mu,\nu) \quad \text{ for all}\,\ \nu \in \mathcal{P}(M).
\end{equation}
\end{enumerate}
\end{proposition}

\begin{proof}
Definitions \ref{wekafordef} and \ref{badfscsubdiff} directly yield \eqref{lowerbounded} directly.
\eqref{charastrongsubd} and \eqref{weaksubdfiffde} follow from the same arguments as in
\cite[Lemmas~3.2, 3.5]{Erb}, respectively.
\end{proof}

\begin{proposition}\label{chainrule} %[Erbar, Prop 3.6]
Let $I \subset \mathbb{R}$ be an open interval and $(\mu_t)_{t \in I}\in \FAC^2(I;\mathfrak{D}({\phi}))$
with the tangent vector field $(\vv_t)_{t \in I}$.
If $\phi$ is $\lambda$-geodesically convex and $\int_I |\partial{\phi}|(\mu_t) |\mu'_+|(t) \,{\dd}t <\infty$,
then $t\longmapsto{\phi}(\mu_t)$ is absolutely continuous in $I$
and we have, for $\mathscr{L}^1$-a.e.\ $t \in I$,
\[
 \frac{\dd}{{\dd}t}{\phi}(\mu_t)
 = \langle \zeta,\vv_t \rangle_{\mu_t} \quad \text{for all}\,\ \zeta \in \partial{\phi}(\mu_t).
\]
\end{proposition}

\begin{proof}
The absolute continuity of ${\phi}(\mu_t)$ follows from Example~\ref{pgeodesicconve2}\eqref{lambda-a},
Corollary~\ref{upperconsidrgrage} and Remark~\ref{trickforwardupp}.
It follows from the assumption and Lemma~\ref{tangerverproer}\eqref{weakconvegoptimavect}
that, for $\mathscr{L}^1$-a.e.\ $t\in I$, $|\partial{\phi}|(\mu_t)<\infty$,
$s \longmapsto {\phi}(\mu_s)$ is differentiable at $t$, and \eqref{eq:Psi/h} holds.
Then, for any $\zeta \in \partial{\phi}(\mu_t)$, we deduce from \eqref{stillconvexloer} that
\begin{align*}
 {\phi}(\mu_{t+h})-{\phi}(\mu_t)
 &\geq \langle \zeta,\Psi^{\mu_{t+h}}_{\mu_t} \rangle_{\mu_t}
 +\frac{\lambda}2 d^2_W(\mu_t,\mu_{t+h}) \\
 &= h \bigg\langle \zeta,\frac{1}{h}\Psi^{\mu_{t+h}}_{\mu_t} -\vv_t \bigg\rangle_{\mu_t}
 +h \langle \zeta,\vv_t \rangle_{\mu_t}
 +\frac{\lambda}2 d^2_W(\mu_t,\mu_{t+h}) \\
 &= h \langle \zeta,\vv_t \rangle_{\mu_t} +o(h)
\end{align*}
for $h>0$.
This implies
\[
 \frac{{\dd}}{{\dd}t_+}{\phi}(\mu_t) \geq \langle \zeta,\vv_t \rangle_{\mu_t}.
\]
Moreover, by considering the reverse curve $\bar{\mu}_t:=\mu_{-t}$
with the tangent vector field $\bar{\vv}_t :=-{\vv}_{-t}$
for the reverse Finsler structure $\overleftarrow{F}(v):=F(-v)$,
we also obtain
\[
 -\frac{{\dd}}{{\dd}t_-}{\phi}(\mu_t) \geq -\langle \zeta,\vv_t \rangle_{\mu_t}.
\]
This completes the proof.
\end{proof}

Now, we introduce gradient flows in the Wasserstein space in the same spirit as
\cite[Definition~3.7]{Erb} (slightly weaker than \cite[Definition~7.6]{OS}),
which is compatible with the notion of ($2$-)curves of maximal slope
as we will see in Proposition~\ref{equivalecegratrajec} below.

\begin{definition}[Gradient flows in the Wasserstein space]\label{wassgradfl}
A curve $(\mu_t)_{t>0} \in \FAC^2_{\loc}((0,\infty);\mathcal{P}(M))$
with the tangent vector field $(\vv_t)_{t>0}$ is called a \emph{trajectory of the gradient flow}
for $\phi$ if
\begin{equation}\label{wassgradf}
 -\LL_{\mu_t}(\vv_t) \in \partial{\phi}(\mu_t) \quad \text{for $\mathscr{L}^1$-a.e.}\ t \in (0,\infty).
\end{equation}
%where $\partial^\circ{\phi}(\mu):=\left\{\eta\in \partial{\phi}(\mu):\, \|-\eta\|^*_{\mu}\leq \|-\zeta\|^*_{\mu}\text{ for any }\zeta\in \partial{\phi}(\mu_t)\right\}$.
\end{definition}

\begin{remark}\label{weassquafl}
Owing to Theorem~\ref{existfloweq}\eqref{continqeuq1} and
Proposition~\ref{properofsubdiff}\eqref{lowerbounded}, we have
\[
 |\mu'_+|(t) |\partial{\phi}|(\mu_t)
 \leq \|{\vv}_t\|_{\mu_t} \|{\LL}_{\mu_t}(\vv_t)\|^*_{\mu_t}
 =\|{\vv}_t\|^2_{\mu_t} \in L^1_{\loc}\big( (0,\infty) \big).
\]
Thus, it follows from Proposition~\ref{chainrule}
(provided that $\phi$ is $\lambda$-geodesically convex)
that $t\longmapsto {\phi}(\mu_t)$ is locally absolutely continuous in $(0,\infty)$ and
\[
 -\frac{\dd}{{\dd}t}{\phi}(\mu_t)=\|{\vv}_t\|^2_{\mu_t}
 \quad \text{for $\mathscr{L}^1$-a.e.}\ t \in (0,\infty),
\]
which implies ${\phi}(\mu_t)<\infty$ for all $t>0$ and the energy identity
\[
 \phi(\mu_r) -\phi(\mu_s) =\int^s_r \|{\vv}_t\|^2_{\mu_t} \,{\dd}t
 \quad \text{ for all}\,\ 0<r<s.
\]
\end{remark}

\begin{proposition}\label{equivalecegratrajec}
Let $\phi$ be a $\lambda$-geodesically convex function such that,
for every $\mu \in \mathfrak{D}(|\partial{\phi}|)$,
there is $\zeta \in \partial{\phi}(\mu)$ with $\|{-}\zeta\|^*_\mu =|\partial{\phi}|(\mu)$.
Then, $(\mu_t)_{t>0} \in \FAC^2_{\loc}((0,\infty);\mathcal{P}(M))$
is a curve of maximal slope for ${\phi}$ with respect to $|\partial{\phi}|$
if and only if it is a trajectory of the gradient flow for ${\phi}$.
\end{proposition}

\begin{proof}
First, suppose that $(\mu_t)_{t>0}$ is a trajectory of the gradient flow for ${\phi}$.
Since $|\partial{\phi}|$ is a weak upper gradient for $-{\phi}$
by Theorem~\ref{slopeuppgr}\eqref{slope-1},
it follows from Definition~\ref{weakforwaupp22} together with Remark~\ref{weassquafl} that
\[
 -\|{\vv}_t\|^2_{\mu_t} =\frac{\dd}{{\dd}t}{\phi}(\mu_t)
 \geq -|\partial {\phi}|(\mu_t) |\mu'_+|(t) \geq -\|{\vv}_t\|^2_{\mu_t}
\]
for $\mathscr{L}^1$-a.e.\ $t \in (0,\infty)$.
This implies
\begin{equation}\label{basisiqu2}
 \frac{{\dd}}{{\dd}t}{\phi}(\mu_t)
 =-\frac12 |\mu'_+|^2(t) -\frac12 |\partial {\phi}|^2(\mu_t), \qquad
 |\partial{\phi}|(\mu_t) =\|{\vv}_t\|_{\mu_t} =|\mu'_+|(t).
\end{equation}
Thus, $(\mu_t)_{t>0}$ is a curve of maximal slope for ${\phi}$ with respect to $|\partial{\phi}|$.

Conversely, if $(\mu_t)_{t>0}$ is a curve of maximal slope with the tangent vector field $(\vv_t)_{t>0}$,
then \eqref{basisiqu2} holds (see \eqref{maxlinequforthreego}) and we have
\[
 \int_r^s |\partial{\phi}|(\mu_t) |\mu'_+|(t) \,{\dd}t
 = \int_r^s |\mu'_+|^2(t) \,{\dd}t <\infty
 \quad \text{for all}\,\ 0<r \leq s<\infty.
\]
By hypothesis, there exists $\zeta_t \in \partial\phi(\mu_t)$
with $\|{-}\zeta_t\|^*_{\mu_t}=|\partial\phi|(\mu_t)$.
Then Proposition~\ref{chainrule} yields
\[
 |\partial {\phi}|(\mu_t) |\mu'_+|(t) =-\frac{\dd}{{\dd}t}{\phi}(\mu_t)
 =-\langle \zeta_t,\vv_t \rangle_{\mu_t} \leq \|{-}\zeta_t\|^*_{\mu_t} \|{\vv}_t\|_{\mu_t}
 =|\partial{\phi}|(\mu_t) |\mu'_+|(t),
\]
which implies $-\LL_{\mu_t}(\vv_t) =\zeta_t \in \partial{\phi}(\mu_t)$ as desired.
\end{proof}

\begin{remark}\label{generacasep2}
Although we need only the $p=2$ case in the analysis of heat flow,
one can generalize the concepts and results in this subsection to the $L^p$-Wasserstein distance with $p>1$
by suitable modifications.
For instance, it follows from \cite[Theorem 3.4]{Ke2} that Proposition \ref{basicminpo} remains valid
by replacing $\nabla\varphi$ with $F^{q-2}(\nabla\varphi)\,\nabla\varphi$,
where $q$ is the conjugate exponent of $p$.
The $L^p$-Wasserstein distance can be used to analyze the gradient flow structure of the $q$-heat flow
as studied in the symmetric case by Kell \cite{Ke};
one could apply our method to generalize it to the asymmetric setting.
\end{remark}

\subsubsection{Gradient flow for the relative entropy}%%%%%%%%%%%%
%%%%%%%%%%%%%%%%%%%%%

The \emph{relative entropy} $H_{\m}: \mathcal{P}(M) \lra (-\infty,\infty]$
with respect to $\m$ is defined by
\[
 H_\mathfrak{m}(\mu) :=\int_M \rho \log\rho \,{\dd}\mathfrak{m}
\]
if $\mu=\rho\m$ and $[\rho \log\rho]_+$ is integrable;
otherwise we set $H_{\m}(\mu):=+\infty$.
In particular, $\mathfrak{D}(H_{\m}) \subset \mathcal{P}_{\mathrm{ac}}(M;\m)$.
We recall some basic properties (see \cite{O,OS}).

\begin{proposition}\label{basicpropertiesfoH}
\begin{enumerate}[{\rm (i)}]
\item\label{lowerboundH}
$H_{\m}(\mu) \ge -\log \m(M)$ for any $\mu\in \mathcal{P}(M);$

\item\label{lersim}
$H_{\m}$ is $\mathcal{T}_+$-lower semicontinuous$;$

\item\label{CDKNcon}
There exists some $K \in \mathbb{R}$ such that $H_{\m}$ is $K$-geodesically convex.
\end{enumerate}
\end{proposition}

%\begin{proof}
%Since $M$ is compact, the convergence of $d_W$ is exactly the weak convergence of $\mathcal{P}(M)$. Let $g$ be a Riemannian metric on $M$ and let $\vol_g$ be the Riemannian measure. There exists a positive function $f\in C^\infty_0(M)$ such that $\m=f\vol_g$. Thus, $H_{\m}(\mu)=H_{\vol_g}(\mu)+\int_M \log f {\dd}\mu$. Since $\mu\longmapsto H_{\vol_g}$ satisfies the properties %(\ref{lowerboundH})(\ref{lersim}) (cf.  \cite[Lemma 4.1, Proposition 5.4]{Erb})), so does $H_{\m}$. On the other hand, the compactness of $M$ implies that $(M,F,\mathfrak{m})$ satisfies $\CD(K,\infty)$ condition for some $K\in \mathbb{R}$ (cf. \cite{KZ,OS}), in which case $H_{\m}$ is $K$-geodesically convex.
 %\end{proof}

The Sobolev space $W^{1,1}(M)$ is defined as the closure of $C^\infty(M)$
with respect to the asymmetric norm $\|u\|_{W^{1,1}} :=\int_M (|u| +F^*({\dd}u)) \,{\dd}{\m}$.
The proof of the following key proposition is postponed to Appendix~\ref{addwasser}.

\begin{proposition}\label{strongHWthe} %[Erbar, Prop 4.3]
For $\mu=\rho\m \in \mathfrak{D}(H_{\m})$, the following are equivalent$:$
\begin{enumerate}[{\rm (I)}]
\item\label{fintieweaksubd}
$|\partial H_{\m}|(\mu)<\infty;$

\item\label{sobolespace}
$\rho\in W^{1,1}(M)$ with ${\dd}\rho =\rho\zeta$ for some $\zeta\in L^2(\mu;T^*M)$.
\end{enumerate}
In this case, $\zeta \in \partial H_{\m}(\mu) \cap T^*_\mu \mathcal{P}(M)$
and is a unique strong subdifferential with $|\partial H_{\m}|(\mu) =\|{-}\zeta\|^*_{\mu}$.
\end{proposition}

The following existence result follows from Proposition~\ref{conextI}
with the help of Propositions~\ref{equivalecegratrajec}, \ref{strongHWthe}.

\begin{theorem}\label{existenceoftranflow}
For any $\mu\in \mathfrak{D}(H_{\m})$,
there exists a trajectory $(\mu_t)_{t\geq 0}$ of the gradient flow for $H_{\m}$ with $\mu_0=\mu$.
\end{theorem}

For $u\in W^{1,1}(M)$, the \emph{distributional Laplacian} $\Delta_{\m}u$ is defined as
\[
 \int_M \varphi \Delta_{\m}u \,{\dd}{\m}
 :=-\int_M \langle {\dd}\varphi,\nabla u \rangle \,{\dd}{\m}
 \quad \text{for all}\,\ \varphi \in C^\infty(M).
\]
Note that the Laplacian $\Delta_{\m}$ is a nonlinear operator
(unless $F$ comes from a Riemannian metric).

\begin{definition}[Weak solutions to the heat equation]\label{df:heat}
We say that $u \in L^1_{\loc}((0,\infty);W^{1,1}(M))$ is a \emph{weak solution
to the heat equation} $\partial_t u_t =\Delta_{\m}u_t$ if
\[
 \int_0^\infty \int_M
 \big\{ u_t \cdot \partial_t \varphi -\langle {\dd}\varphi,\nabla u_t \rangle \big\} \,{\dd}{\m}\,{\dd}t =0
 \quad \text{for all}\,\ \varphi\in C^\infty_0\big( (0,\infty) \times M \big).
\]
\end{definition}

\begin{remark}\label{rm:reg}
The $W^{1,1}$-regularity along with \eqref{deltarohoprperties} below is the minimum regularity
to which our characterization result applies.
In the Riemannian case, it is well known that such a weak solution
has a smooth version (e.g., by using the heat kernel).
In the Finsler setting, however, there is no heat kernel due to the nonlinearity and
it seems unclear (to the authors) if a solution to the heat equation
given in Theorem~\ref{equvthegra}\eqref{heatflowcondi}
automatically enjoys a higher regularity (e.g., $C^{1,\alpha}$ as in \cite{GS,OS}).
\end{remark}

\begin{theorem}\label{equvthegra}
Let $(\mu_t)_{t\geq 0}$ be a continuous curve in $(\mathcal{P}(M),d_W)$.
Then the following are equivalent$:$
\begin{enumerate}[{\rm (I)}]
\item\label{gradflowcondi}
$(\mu_t)_{t>0}$ is a trajectory of the gradient flow for $H_{\m};$

\item\label{heatflowcondi}
$\mu_t$ is given by $\mu_t =\rho_t \m \in \mathcal{P}_{\mathrm{ac}}(M;\m)$ for $t>0$
and $(-\rho_t)_{t>0}$ is a weak solution to the heat equation
\begin{equation}\label{reverseheatflow}
\partial_t(-\rho_t) =\Delta_{\m} (-\rho_t)
\end{equation}
satisfying
\begin{equation}\label{deltarohoprperties}
 H_{\m}(\rho_t\m) <\infty \,\ \text{for all}\ t>0,\qquad
 \int^{t_1}_{t_0} \int_M \frac{F^2(\nabla(-\rho_t))}{\rho_t} \,{\dd}{\m}\,{\dd}t <\infty
 \,\ \text{for all}\,\ 0<t_0<t_1.
\end{equation}
\end{enumerate}
\end{theorem}

\begin{proof}
\eqref{gradflowcondi}$\,\Rightarrow\,$\eqref{heatflowcondi}
Let $(\vv_t)_{t>0}$ be the tangent vector field of $(\mu_t)_{t>0}$.
Recall from Remark~\ref{weassquafl} that $H_{\m}(\mu_t)<\infty$ holds for all $t>0$,
and hence, we have $\mu_t =\rho_t \m \in \mathfrak{D}(H_{\m})$.
Moreover, we deduce from Proposition~\ref{normstrcturelemma}\eqref{dualpro2}
and \eqref{wassgradf} that
$-\LL_{\mu_t}(\vv_t) \in \partial H_{\m}(\mu_t) \cap T_{\mu_t}^*\mathcal{P}(M)$
for $\mathscr{L}^1$-a.e.\ $t \in (0,\infty)$.
Then it follows from Proposition~\ref{properofsubdiff} that
$-\LL_{\mu_t}(\vv_t)$ is a strong subdifferential and
\[
 |\partial H_{\m}|(\mu_t) \leq \|\LL_{\mu_t}(\vv_t)\|^*_{\mu_t}
 =\|{\vv}_t\|_{\mu_t} <\infty.
\]
Hence, Proposition~\ref{strongHWthe} yields that $\rho_t \in W^{1,1}(M)$
and $-\LL_{\mu_t}(\vv_t) ={\dd}\rho_t/\rho_t$ for $\mathscr{L}^1$-a.e.\ $t \in (0,\infty)$.
Note that $\vv_t =\LL^{-1}_{\mu_t}(-{\dd}\rho_t/\rho_t) =\nabla(-\rho_t)/{\rho_t}$
and then the continuity equation \eqref{continequa} is exactly the heat equation \eqref{reverseheatflow}.
Moreover, thanks to $\|{\vv}_t\|_{\mu_t} \in L^2_{\loc}((0,\infty))$, for any $0<t_0< t_1$,
we have
\[
 \int^{t_1}_{t_0} \int_M \frac{F^2(\nabla(-\rho_t))}{\rho_t} \,{\dd}{\m} \,{\dd}t
 =\int^{t_1}_{t_0} \int_M F^2(\vv_t) \,{\dd}\mu_t \,{\dd}t
 =\int^{t_1}_{t_0} \|{\vv}_t\|^2_{\mu_t} \,{\dd}t <\infty.
\]

\eqref{heatflowcondi}$\,\Rightarrow\,$\eqref{gradflowcondi}
Note again that the heat equation \eqref{reverseheatflow} is equivalent to
the continuity equation \eqref{continequa} for $\mu_t$ with $\vv_t:={\nabla(-\rho_t)}/{\rho_t}$.
Then the assumption \eqref{deltarohoprperties} implies $\vv_t \in L^2(\mu_t;TM)$
for $\mathscr{L}^1$-a.e.\ $t \in (0,\infty)$.
Now it follows from Proposition~\ref{existfloweq}\eqref{continqeuq2} that
$(\mu_t)_{t>0} \in \FAC^2_{\loc}((0,\infty);\mathcal{P}(M))$. Moreover,
the H\"older inequality furnishes
\[
 \frac{\|\rho_t\|_{W^{1,1}}}{\lambda_F(M)}
 \leq \|{-}\rho_t\|_{W^{1,1}}
 = 1+\int_M \frac{F(\nabla(-\rho_t))}{\rho_t} \,{\dd}\mu_t
 \leq 1+\|{\vv}_t\|_{\mu_t} <\infty,
\]
thereby $\rho_t\in W^{1,1}(M)$.
We deduce from Proposition~\ref{strongHWthe} that
$\zeta_t :=-\LL_{\mu_t}(\vv_t) ={\dd}\rho_t/\rho_t$
belongs to $\partial H_{\m}(\mu_t) \cap T_{\mu_t}^*\mathcal{P}(M)$
and is a unique strong subdifferential at $\mu_t$.
Thus $\vv_t \in T_{\mu_t}\mathcal{P}(M)$,
and $(\vv_t)_{t>0}$ is the tangent vector field of $(\mu_t)_{t>0}$
(recall Proposition~\ref{existfloweq}\eqref{continqeuq1}).
Since $-\LL_{\mu_t}(\vv_t) =\zeta_ t\in \partial H_{\m}(\mu_t)$ for $\mathscr{L}^1$-a.e.\ $t \in (0,\infty)$,
we conclude that $(\mu_t)_{t>0}$ is a trajectory of the gradient flow for $H_{\m}$.
\end{proof}

We remark that, due to the irreversibility of $F$,
$(\rho_t)_{t>0}$ in Theorem~\ref{equvthegra} is not necessarily a weak solution to the heat equation.
Instead, it is a weak solution to the heat equation with respect to the reverse Finsler structure
$\overleftarrow{F}(v):=F(-v)$.
We denote by $\overleftarrow{\Delta}_{\m}$ the Laplacian for $(M,\overleftarrow{F},\m)$.

\begin{corollary}\label{existsofheatlow}
\begin{enumerate}[{\rm (i)}]
\item\label{norsolheatfl}
For any $u \in L^2(M)$ with $\sup_M u <\infty$,
there exists a weak solution $(u_t)_{t \geq 0}$ to the heat equation
$\partial_t u_t=\Delta_{\m} u_t$ with $u_0=u$.

\item\label{reversolheatfl}
For any $u \in L^2(M)$ with $\inf_M u >-\infty$,
there exists a weak solution $(u_t)_{t \geq 0}$ to the heat equation
$\partial_t u_t=\overleftarrow{\Delta}_{\m} u_t$ with respect to $\overleftarrow{F}$ with $u_0=u$.
\end{enumerate}
\end{corollary}

\begin{proof}
\eqref{norsolheatfl}
Suppose that $u$ is not constant and set $\rho :=(-u+\sup u)/\|u-\sup u\|_{L^1}$.
Then $\mu:=\rho \m \in \mathcal{P}_{\mathrm{ac}}(M;\m)$ and, moreover,
$\mu \in \mathfrak{D}(H_{\m})$ (see, e.g., the proof of \cite[Lemma~16.2]{Obook}).
Hence, there exists a trajectory $\mu_t=\rho_t \m$ of the gradient flow for $H_{\m}$ with $\mu_0=\mu$
by Theorem~\ref{existenceoftranflow},
and $(-\rho_t)_{t>0}$ is a weak solution to the heat equation by Theorem~\ref{equvthegra}.
This completes the proof by letting $u_t :=\sup u -\|u-\sup u\|_{L^1} \cdot \rho_t$.

\eqref{reversolheatfl}
Thanks to $\overleftarrow{\Delta}_{\m}f=-{\Delta}_{\m}(-f)$,
this is shown in the same way as \eqref{norsolheatfl}.
\end{proof}

\begin{remark}\label{rm:L^2}
In \cite[\S 3]{OS}, we constructed a weak solution to the heat equation starting from $u \in H^1(M)$
as a trajectory of the gradient flow for the energy in the Hilbert space $L^2(M)$.
Then, by the $L^2$-contraction property,
we can extend it to a contraction semigroup acting on $L^2(M)$.
In contrast, Corollary~\ref{existsofheatlow} provides a direct construction for (bounded) $u \in L^2(M)$.
Because of the lack of a higher regularity (Remark~\ref{rm:reg})
and the Wasserstein contraction (see \cite{OSnc} and Subsection~\ref{ssc:rem} below),
it seems unclear if these flows starting from $u \in L^2(M) \setminus H^1(M)$ coincide.
\end{remark}

\subsection{Further problems and related works}\label{ssc:rem}%%%%%%%%
%%%%%%%%%%%%%%%%%%%%

In ``Riemannian-like'' spaces,
one can proceed to the \emph{contraction property} asserting that
two gradient curves emanating from different points
are getting closer with an exponential rate depending on the convexity of $\phi$:
\[ d\big( \xi_1(t),\xi_2(t) \big) \leq {\ee}^{-\lambda t} d\big( \xi_1(0),\xi_2(0) \big) \]
(see, e.g., \cite[Theorem~4.0.4(iv)]{AGS}).
This property is, however, known to fail in Finsler-like spaces (see \cite{OSnc}).
It is an intriguing open problem if one can obtain
any weaker contraction estimate for convex functions on Finsler manifolds or normed spaces.
An important apparatus to study the contraction property is
the \emph{evolution variational inequality} (see \cite{AGS,MS}),
which also forces the space to be Riemannian.

The contraction property plays a vital role in the study of heat flow.
In view of Section~\ref{heatflowfins} and \cite{OS},
heat flow can be regarded as the gradient flow of the relative entropy
in the $L^2$-Wasserstein space, and the convexity of the relative entropy
is equivalent to the associated lower Ricci curvature bound (see \cite{O,vRS,Vi}).
Moreover, the contraction property of heat flow is
equivalent to the corresponding gradient estimate (see \cite{Kuwada}).
Though we know some gradient estimates in the Finsler setting (see \cite{Oisop,Obook,OSbw}),
the lack of the Riemannian-like structure prevents us to obtain a contraction property of heat flow
(precisely, the nonlinearity of the Finsler Laplacian causes an essential difference).
Therefore, generalizing the contraction property to the Finsler setting
will make a breakthrough also in the study of heat flow on Finsler manifolds.

We refer to \cite{Ohyp} for a recent work on discrete-time gradient flows in Gromov hyperbolic spaces;
note that some non-Riemannian Finsler manifolds can be Gromov hyperbolic.
Let us also mention another related work \cite{LOZ} concerning \emph{self-contracted curves},
which is available for some Finsler manifolds.

\appendix
\section{An auxiliary lemma}%%%%%%%%%%%%%%%
%%%%%%%%%%%%%%%

\begin{lemma}\label{basicpestimate}
Given $p \in [1,\infty)$ and any $\epsilon>0$, we have
\[
 (1+\epsilon) a^p +\mathfrak{C}(p,\epsilon) b^p \geq (a+b)^p \quad
 \text{for all}\,\ a,b \geq 0,
\]
where
\[
 \mathfrak{C}(1,\epsilon) =1, \qquad
 \mathfrak{C}(p,\epsilon) =\frac{1+\epsilon}{((1+\epsilon)^{1/(p-1)} -1)^{p-1}}
 \quad \text{for}\,\ p>1.
\]
In particular, $\lim_{\epsilon \to \infty}\mathfrak{C}(p,\epsilon)=1$ for all $p \geq 1$.
\end{lemma}

\begin{proof}
We assume $p>1$ and $a,b>0$ without loss of generality.
Moreover, by dividing both sides with $b^p$, it suffices to consider the case of $b=1$.
Put $f(a):=(a+1)^p -(1+\epsilon)a^p$ for $a>0$.
Since
\[
 f'(a) =p(a+1)^{p-1} -p(1+\epsilon)a^{p-1}
\]
attains $0$ only at $\bar{a}=\{ (1+\epsilon)^{1/(p-1)} -1 \}^{-1}$, we have
\[
 f(a) \leq f(\bar{a})
 =\frac{1+\epsilon}{((1+\epsilon)^{1/(p-1)} -1)^{p-1}}
\]
for all $a>0$.
This completes the proof.
\end{proof}

\section{Complementary results for generalized Funk spaces}\label{Funkspaces}
%%%%%%%%%%%%%%%%%%%%%%%%%%%%

Let $(\mathbb{B},d)$ be a generalized Funk space as in Subsection~\ref{funkspacesgrad}.

\begin{proposition}\label{bascifunkproperty}
For any $x\in \mathbb{B}$, we have the following.
\begin{enumerate}[{\rm (i)}]
\item\label{fungrad1}
$F(x,\cdot)$ is convex on $T_x\mathbb{B} \cong \mathscr{H}$,
and $F(x,v) \leq d(x,x+v)$ for any $v \in \mathscr{H}$ with $x+v \in \mathbb{B}$.

\item\label{fungrad2}
$F^*(x,\cdot)$ is convex on $T^*_x\mathbb{B} \cong \mathscr{H}^*$
and weakly*-lower semicontinuous, i.e.,
if $\zeta_i \overset{*}\rightharpoonup \zeta$ $($weakly* convergent$)$,
then we have $F^*(x,\zeta) \leq \liminf_{i \to \infty} F^*(x,\zeta_i)$.

\item\label{fungrad3}
The set $\mathfrak{J}_p(x,v)$ is at most a singleton
for every $v \in T_x\mathbb{B} \cong \mathscr{H}$.
\end{enumerate}
\end{proposition}

\begin{proof}
\eqref{fungrad1}
This is reduced to the finite-dimensional case.
The inequality $F(x,v) \leq d(x,x+v)$ can be seen from
$d(x,x+\varepsilon v) \le d(x+tv,x+(t+\varepsilon)v)$ for $0 \le t<t+\varepsilon \le 1$
(by the interpretation of $d$ as in \cite[(1.1)]{Sh1}).

\eqref{fungrad2}
This is a direct consequence of the definition of $F^*$.

\eqref{fungrad3}
Assume $v \neq \mathbf{0}$ without loss of generality.
Given any $\zeta_1,\zeta_2\in \mathfrak{J}_p(x,v)$,
set $\bar{v}:=v/F(x,v)$ and $\bar{\zeta}_i:=\zeta_i/F^*(x,\zeta_i)$, $i=1,2$.
It suffices to show $\bar{\zeta}_1=\bar{\zeta}_2$.
Since $f:=F(x,\cdot)$ is differentiable at $\bar{v}$, for any $\varepsilon>0$,
there exists $\delta>0$ such that
$f(\bar{v}+w)+f(\bar{v}-w) \leq 2f(\bar{v})+\varepsilon\|w\| =2+\varepsilon\|w\|$
for any $w \in \mathscr{H}$ with $\|w\|< \delta$
(see, e.g., \cite[Lemma~8.3]{FHH}).
Combining this with $F^*(x,\bar{\zeta}_i)=\langle \bar{\zeta}_i,\bar{v} \rangle=1$, we find
\[
 \langle \bar{\zeta}_1,w \rangle -\langle \bar{\zeta}_2,w \rangle
 =\langle \bar{\zeta}_1,\bar{v}+w \rangle +\langle \bar{\zeta}_2,\bar{v}-w \rangle
 -\langle \bar{\zeta}_1,\bar{v} \rangle -\langle \bar{\zeta}_2,\bar{v} \rangle
 \leq f(\bar{v}+w)+f(\bar{v}-w)-2
 \leq \varepsilon\|w\|
\]
for any $w$ with $\|w\|<\delta$.
This implies $\bar{\zeta}_1 =\bar{\zeta}_2$.
\end{proof}

\begin{proposition}\label{phigradproe} %[AGS, Prop 1.4.4]
Let $\phi:\mathbb{B} \lra (-\infty,\infty]$ be a proper lower semicontinuous function.
\begin{enumerate}[{\rm (i)}]
\item\label{convexfun1}
We have $|\partial\phi|(x)\leq F^*(-\partial^\circ\phi(x))$ for all $x\in \mathbb{B}$.
In particular, $x\longmapsto F^*(-\partial^\circ\phi(x))$ is a weak upper gradient for $-\phi$.

\item\label{convexfun2}
If $\phi$ is convex, then $|\partial\phi|$ is a strong upper gradient and we have
\begin{equation}\label{basiupperproper1}
 |\partial\phi|(x) =\mathfrak{l}_\phi(x) =F^* \big( {-}\partial^\circ \phi(x) \big)
 \quad \text{for all}\,\ x \in \mathbb{B}.
\end{equation}

\item\label{convexfun3}
If $\phi$ is convex, then the graph of $\partial\phi$ in $\mathbb{B} \times \mathscr{H}^*$
is strongly-weakly* closed.
Moreover, we have
\begin{equation}\label{strongweakclosed}
 \zeta_i \in \partial\phi(x_i),\quad x_i \to x,\quad \zeta_i \overset{*}\rightharpoonup \zeta
 \quad \Longrightarrow \quad \zeta \in \partial\phi(x),\quad \phi(x_i) \to \phi(x).
\end{equation}
\end{enumerate}
\end{proposition}

\begin{proof}
\eqref{convexfun1}
Assume $x \in \mathfrak{D}(\partial\phi)$ and $|\partial \phi|(x)>0$ without loss of generality.
For any $\zeta \in \partial\phi(x)$, we have
\[
 |\partial \phi|(x)
 =\limsup_{v \to \mathbf{0}} \frac{\phi(x)-\phi(x+v)}{d(x,x+v)}
 \leq \limsup_{v \to \mathbf{0}} \frac{\langle -\zeta,v \rangle}{F(x,v)}
 \leq F^*(x,-\zeta).
\]
This shows the former assertion,
and the latter one is a consequence of Theorem~\ref{slopeuppgr}\eqref{slope-1}.

\eqref{convexfun2}
It follows from Theorem~\ref{slopeuppgr}\eqref{slope-2} and Theorem~\ref{thremforwposrt} that
$|\partial\phi|=\mathfrak{l}_{\phi}$ is a strong upper gradient for $-\phi$.
In view of (\ref{convexfun1}), it remains to show the existence of $\zeta \in \partial\phi(x)$ with $F^*(x,-\zeta)\leq \mathfrak{l}_\phi(x)$
for $x \in \mathfrak{D}(\phi)$ satisfying $\mathfrak{l}_\phi(x) <\infty$.
Observe from Theorem~\ref{thremforwposrt} and Proposition~\ref{bascifunkproperty}\eqref{fungrad1} that
\[
 -\mathfrak{l}_{\phi}(x) F(x,v)
 \leq \phi(x+v) -\phi(x) \quad \text{for all}\ v \in \mathscr{H} \text{ with } x+v \in \mathbb{B},
\]
i.e., the convex set
$\{ (v,r) \in \mathscr{H} \times \mathbb{R} \,|\, r \geq \phi(x+v)-\phi(x),\, x+v \in \mathbb{B}\}$
is disjoint from the open convex set
$\{ (v,r) \in \mathscr{H} \times \mathbb{R} \,|\, r<-\mathfrak{l}_\phi(x)F(x,v),\, x+v \in \mathbb{B}\}$.
Therefore, we can apply a geometric version of the Hahn--Banach theorem
to obtain $\zeta \in \mathscr{H}^*$ and $\alpha \in \mathbb{R}$ such that
\[
 -\mathfrak{l}_\phi(x) F(x,v)
 \leq \langle \zeta,v \rangle +\alpha
 \leq \phi(x+v)-\phi(x)
 \quad \text{for all}\ v \in \mathscr{H} \text{ with } x+v \in \mathbb{B}.
\]
Taking $v= \mathbf{0}$ implies $\alpha=0$.
Thus, the first inequality shows $F^*(x,-\zeta)\leq \mathfrak{l}_{\phi}(x)$
while the second one means $\zeta \in \partial\phi(x)$.
This completes the proof.

\eqref{convexfun3}
Owing to Proposition~\ref{basicproperofFunk}(\ref{funpor2}),
this claim can be proved in the same way as \cite[Proposition 1.4.4]{AGS}.
\end{proof}
%ここ%

\section{Complementary results for the Wasserstein space}\label{addwasser}
%%%%%%%%%%%%%%%%%%%%%%%%%%%%

Let $(M,F,\m)$ be a compact Finsler manifold with a smooth positive measure
as in Subsection~\ref{heatflowfins}.

\begin{proof}[Proof of Proposition~$\ref{normstrcturelemma}$]
\eqref{dualpro2} and \eqref{dualpro1} readily follow from the properties of the Legendre transformation.

\eqref{Riesz}
Let $g$ be an arbitrary Riemannian metric on $M$,
which is bi-Lipschitz equivalent to $F$ by the compactness of $M$.
Then $L^2(\mu;TM)$ can be regarded as a Hilbert space equipped with the inner product
\[
 G(\vv,\w) :=\int_M g(\vv,\w) \,{\dd}\mu.
\]
Hence, for any bounded linear functional $\mathcal{Q}$ on $L^2(\mu;TM)$,
the Riesz representation theorem yields unique $\vv \in L^2(\mu;TM)$
such that $\mathcal{Q}(\w)=G(\vv,\w)$.
We conclude the proof by setting $\zeta:=g(\vv,\cdot)$.

\eqref{basicconnection}
Let $g$ and $G$ be as above and
$T^g_\mu \mathcal{P}(M)$ be the tangent space of $\mathcal{P}(M)$ at $\mu$ induced from $g$.
Note that $L^2(\mu;T^*M)$, $L^2(\mu;TM)$ and $\mathbf{Ker(div)}(\mu)$
are independent of the choice of a metric,
and $T^g_\mu \mathcal{P}(M)$ is identified with $T^*_{\mu}\mathcal{P}(M)$
via the Legendre transformation of $g$.
Then the claim follows from
$L^2(\mu;TM) =T^g_\mu \mathcal{P}(M) \oplus \mathbf{Ker(div)}(\mu)$
in \cite[Lemma~2.4]{Erb},
where $\oplus$ is the orthogonal direct sum with respect to $G$.

\eqref{identityrelation}
Owing to \eqref{basicconnection}, we have
$0 =\langle \LL_\mu(\vv),\vv-\w \rangle_\mu =\|{\vv}\|_\mu^2 -\langle \LL_\mu(\vv),\w \rangle_\mu$,
and hence
\[
 \|{\vv}\|_\mu^2 =\langle \LL_\mu(\vv),\w \rangle_{\mu} \leq \|{\vv}\|_\mu \|{\w}\|_{\mu}.
\]
This yields $\|{\vv}\|_\mu \leq \|{\w}\|_\mu$, and similarly $\|{\vv}\|_\mu \geq \|{\w}\|_\mu$ holds.
Therefore we find $\|{\vv}\|^2_\mu=\|{\w}\|^2_\mu =\langle \LL_\mu(\vv),\w \rangle_{\mu}$,
and then $\vv=\w$ follows from \eqref{dualpro1} above.
\end{proof}

\begin{proof}[Proof of Proposition~$\ref{strongHWthe}$]
Note that the equivalence between \eqref{fintieweaksubd} and \eqref{sobolespace}
can be reduced to a Riemannian metric $g$ on $M$.
Hence, it follows from \cite[Proposition~4.3]{Erb} that
\eqref{fintieweaksubd} and \eqref{sobolespace} are equivalent and
${\dd}\rho=\rho \zeta$ for some $\zeta \in T^*_{\mu}\mathcal{P}(M)$.

Next we prove that $\zeta$ is a subdifferential at $\mu\in \mathfrak{D}(H_{\m})$.
Given $\nu\in \mathcal{P}(M)$,
let $\Psi^\nu_\mu$ be the unique optimal transport vector field
and set $\mu_t :=\exp(t\Psi^\nu_\mu)_\sharp \mu$.
On the one hand, by the same argument as in \cite[Proposition 7.7]{OS}, we have
\[
 \lim_{t \to 0^+} \frac{H_{\m}(\mu_t) -H_{\m}(\mu)}{t}
 =\int_M \langle {\dd}\rho, \Psi_{\mu}^{\nu} \rangle \,{\dd}{\m}
 =\int_M \langle \zeta,\Psi_{\mu}^{\nu} \rangle \,{\dd}\mu
 \leq \|\zeta\|_{\mu}^* \|\Psi_{\mu}^{\nu}\|_\mu <\infty.
\]
On the other hand, Proposition~\ref{basicpropertiesfoH}\eqref{CDKNcon} furnishes
\[
 \frac{H_{\m}(\mu_t) -H_{\m}(\mu)}{t}
 \leq H_{\m}(\nu) -H_{\m}(\mu) -\frac{K}{2}(1-t) d^2_W(\mu,\nu), \quad t\in (0,1].
\]
Letting $t\rightarrow0^+$, we find
\[
 H_{\m}(\nu) -H_{\m}(\mu)
 \geq \langle \zeta,\Psi_{\mu}^{\nu} \rangle_{\mu} +\frac{K}{2} d^2_W(\mu,\nu).
\]
Then it follows from Proposition~\ref{properofsubdiff} that
$\zeta$ is a strong {subdifferential} and $|\partial H_{\m}|(\mu) \leq \|{-}\zeta\|^*_{\mu}<\infty$.

To see the reverse inequality,
we consider an arbitrary smooth vector field $\w$ and a map $T_t:M \lra M$
such that, for every $x \in M$, $t \longmapsto T_t(x)$ is the geodesic with
$\frac{\dd}{{\dd}t}|_{t=0} [T_t(x)] =\w(x)$.
Thus, by the same discussion as in \cite[Lemma~4.2]{Erb},
we find that $\mu_t:=(T_t)_{\sharp}\mu$ satisfies
\begin{equation}\label{derivativeH}
 \lim_{t \to 0} \frac{H_{\m}(\mu_t)-H_{\m}(\mu)}{t}
 =-\int_M \rho \di_{\m}(\w) \,{\dd}{\m},
\end{equation}
where $\di_{\m}$ denotes the divergence with respect to $\m$
(see also the calculation after the proof).
This implies that
\[
 \mathcal{Q}({\w}) :=\int_M \rho \di_{\m}(\w) \,{\dd}{\m}
 \leq |\partial H_{\m}|(\mu) \|{\w}\|_\mu
\]
is extended to a linear bounded functional on $L^2(\mu;TM)$.
Then Proposition~\ref{normstrcturelemma}\eqref{Riesz} provides
a unique $1$-form $\bar{\zeta} \in L^2(\mu;T^*M)$ with
$\mathcal{Q}(\w)=\langle -\bar{\zeta},\w \rangle_{\mu}$
and $\|{-}\bar{\zeta}\|_{\mu}^* \le |\partial H_{\m}|(\mu)$.
Since
\[
 \int_M \rho \di_{\m}(\w) \,{\dd}{\m}
 = -\int_M \rho \langle \bar{\zeta},\w \rangle \,{\dd}{\m}
\]
for all smooth vector fields $\w$, we obtain $\rho \bar{\zeta} ={\dd}\rho=\rho\zeta$.
Therefore, $\|{-}\zeta\|_{\mu}^* \le |\partial H_{\m}|(\mu)$
and hence $|\partial H_{\m}|(\mu) =\|{-}\zeta\|^*_{\mu}$ holds as desired.

Finally, owing to \eqref{derivativeH}, the uniqueness of $\zeta$
can be shown in a similar way to \cite[Proposition~4.3]{Erb}.
\end{proof}

For completeness, we give a (standard) calculation of
$\di_{\m}(\w)$ needed in the proof of \eqref{derivativeH}:
\begin{align*}
 \int_M \rho \di_{\m}(\w) \,{\dd}{\m}
 &=-\int_M \w(\rho) \,{\dd}{\m}
 =\lim_{t \to 0} \int_M \frac{\rho -\rho(T_t)}{t} \,{\dd}{\m}
 =\lim_{t \to 0} \frac{1}{t} \int_M \rho \big( 1- \det[{\dd}(T_t^{-1})] \big) \,{\dd}{\m} \\
 &= \lim_{t \to 0} \frac{1}{t} \int_M \rho \left( 1- \frac{1}{\det[{\dd}T_t]\circ T_t^{-1}} \right) \,{\dd}{\m}
 =\int_M \rho \cdot \frac{\dd}{{\dd}t}\bigg|_{t=0} \det[{\dd}T_t] \,{\dd}{\m}.
\end{align*}

\end{document}